\documentclass[11pt]{article}

\RequirePackage{etex}
\usepackage[T1]{fontenc}
\usepackage[utf8]{inputenx}
\usepackage{lmodern}
\usepackage{textcomp}

\usepackage{babel}

\usepackage{geometry}
\usepackage{amsmath}
\usepackage{amsfonts}
\usepackage{amstext}
\usepackage{amssymb}
\usepackage{amsthm}
\usepackage{mathtools}
\usepackage{amscd}
\usepackage{mathrsfs}
\usepackage{dsfont}
\usepackage{latexsym}
\usepackage{braket}

\usepackage{graphicx}
\usepackage{xcolor}
\usepackage{tikz}
\usepackage{tikz-cd}

\usepackage[shortlabels]{enumitem}
\usepackage{etaremune}

\usepackage{soul}
\usepackage{parskip}
\usepackage{caption}
\usepackage{abstract}
\usepackage{comment}
\usepackage{varwidth}

\usepackage{url}
\usepackage{etoolbox}
\usepackage{makeidx}
\usepackage{totcount}
\usepackage{blindtext}

\definecolor{myurlcolor}{rgb}{0,0,0.4}
\definecolor{mycitecolor}{rgb}{0,0.5,0}
\definecolor{myrefcolor}{rgb}{0.5,0,0}

\setstcolor{blue}

\usepackage[pagebackref,draft=false]{hyperref}
\hypersetup{
  colorlinks,
  linkcolor=myrefcolor,
  citecolor=mycitecolor,
  urlcolor=myurlcolor
}

\usepackage[capitalize]{cleveref}
\usepackage{aliascnt}

\usepackage{autonum}

\newtheoremstyle{maybestyle}
  {12pt}
  {12pt}
  {\itshape}
  {0pt}
  {\normalfont\bfseries}
  {.}
  {6pt}
  {}

\newtheoremstyle{mystyle}
  {12pt}
  {12pt}
  {\normalfont}
  {0pt}
  {\bfseries}
  {.}
  {6pt}
  {}

\usepackage{float}
\usepackage{yhmath}

\theoremstyle{maybestyle}
\newtheorem{proposition}{Proposition}[section]

\newaliascnt{lemma}{proposition}
\newtheorem{lemma}[lemma]{Lemma}
\aliascntresetthe{lemma}

\newaliascnt{corollary}{proposition}

\aliascntresetthe{corollary}

\theoremstyle{mystyle}

\newaliascnt{problem}{proposition}
\newtheorem{problem}[problem]{Problem}
\aliascntresetthe{problem}

\newaliascnt{remark}{proposition}
\newtheorem{remark}[remark]{Remark}
\aliascntresetthe{remark}

\newaliascnt{definition}{proposition}
\newtheorem{definition}[definition]{Definition}
\aliascntresetthe{definition}

\newaliascnt{example}{proposition}
\newtheorem{example}[example]{Example}
\aliascntresetthe{example}

\crefname{proposition}{Proposition}{Propositions}
\crefname{lemma}{Lemma}{Lemmas}
\crefname{corollary}{Corollary}{Corollaries}
\crefname{problem}{Problem}{Problems}
\crefname{remark}{Remark}{Remarks}
\crefname{definition}{Definition}{Definitions}
\crefname{example}{Example}{Examples}

\newtheorem*{proof*}{Proof}

\makeatletter

\newif\ifqedcircenvironment
\newcount\qedcircplaced

\newcommand{\qedcirc}{%
  \begingroup
  \renewcommand{\qedsymbol}{$\circ$}%
  \qed
  \endgroup
}

\newcommand{\placeqedcirc}{%
  \ifnum\qedcircplaced=0
    \popQED
    \global\qedcircplaced=1
  \fi
}

\newcommand{\placeqedcircinlist}{%
  \ifqedcircenvironment
    \ifnum\@listdepth=1
      \placeqedcirc
    \fi
  \fi
}

\newcommand{\enableexampleqed}[1]{%
  \AtBeginEnvironment{#1}{%
    \qedcircenvironmenttrue
    \global\qedcircplaced=0
    \pushQED{\qedcirc}%
  }%
  \AtEndEnvironment{#1}{%
    \placeqedcirc
    \qedcircenvironmentfalse
  }%
}

\AtEndEnvironment{enumerate}{\placeqedcircinlist}
\AtEndEnvironment{itemize}{\placeqedcircinlist}
\AtEndEnvironment{description}{\placeqedcircinlist}

\makeatother

\enableexampleqed{example}
\enableexampleqed{remark}
\enableexampleqed{definition}

\newcommand{\blue}[1]{\textcolor{blue}{#1}}
\newcommand{\red}[1]{\textcolor{red}{#1}}

\newcommand{\dd}{\mathrm{d}}

\DeclareMathOperator{\T}{\mathbf{T}}
\DeclareMathOperator{\id}{id}
\DeclareMathOperator{\ad}{ad}
\DeclareMathOperator{\SU}{SU}
\DeclareMathOperator{\U}{U}

\DeclareMathOperator{\Ad}{\mathrm{Ad}}
\DeclareMathOperator{\Tr}{Tr}

\DeclareMathOperator{\FS}{FS}

\DeclareMathOperator{\Tor}{Tor}

\DeclareMathOperator{\su}{\mathfrak{su}}
\DeclareMathOperator{\CP}{\mathbb{C}\mathbb{P}}

\title{A two-point approach to the inverse problem in information geometry}

\author{
F. M. Ciaglia$^{1,6}$ 
\href{https://orcid.org/0000-0002-8987-1181}{\includegraphics[scale=0.7]{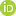}}, 
G. Marmo$^{2,3,7}$ 
\href{https://orcid.org/0000-0003-2662-2193}{\includegraphics[scale=0.7]{ORCID.png}}, 
M. Pacelli$^{2,4,5,8}$ 
\href{https://orcid.org/0009-0008-4437-6970}{\includegraphics[scale=0.7]{ORCID.png}}, \\
L. Schiavone$^{5,9}$ 
\href{https://orcid.org/0000-0002-1817-5752}{\includegraphics[scale=0.7]{ORCID.png}}, 
A. Zampini$^{2,4,5,10}$ 
\href{https://orcid.org/0000-0003-0980-6003}{\includegraphics[scale=0.7]{ORCID.png}}
}

\begin{document}

\maketitle 

\noindent

{\footnotesize{$^{1}$   Universidad Carlos III de Madrid, %ROR: \href{https://ror.org/03ths8210}{03ths8210}, 
Departamento de Matemáticas, %Avenida de la Universidad, 30 (edificio Sabatini), 28911 
Leganés (Madrid), Spain\\
$^{2}$ INFN-Sezione di Napoli, Naples, Italy  \\
$^{3}$ Dipartimento di Fisica ``E. Pancini'', Università degli Studi di Napoli Federico II,  Naples, Italy  \\
$^{4}$ Scuola Superiore Meridionale,  Naples, Italy  \\
$^{5}$ Dipartimento di Matematica e Applicazioni ``R. Caccioppoli'', Università degli Studi di Napoli Federico II, Naples, Italy

\noindent
$^{6}$\texttt{fciaglia[at]math.uc3m.es} \quad 
$^{7}$\texttt{marmo[at]na.infn.it} \quad 
$^{8}$\texttt{marco.pacelli-ssm[at]unina.it}  \\
$^{9}$\texttt{luca.schiavone[at]unina.it}  \quad
$^{10}$\texttt{alessandro.zampini[at]unina.it}}

\begin{abstract}
\footnotesize
We formulate the inverse problem in information geometry within a two-point tensorial framework and solve it for general metric-affine manifolds, without imposing any curvature or torsion constraints. The construction is explicit and starts directly from the given geometric data: the metric tensor is paired to the affine structure, realized through a local parallelism obtained from parallel transport. This yields a contrast bi-form inducing the original metric-affine manifold. The inverse problems for statistical manifolds admitting torsion and for statistical manifolds are then recovered by homotopical reduction. In this way, suitable pre-contrast and contrast functions are obtained, including several established constructions for statistical manifolds and SMATs. We apply the general theory to reductive homogeneous pseudo-Riemannian manifolds endowed with invariant affine connections. Particular attention is devoted to semisimple Lie groups with Cartan-Schouten connections and to odd-dimensional spheres equipped with Berger metrics.
\end{abstract}

{\footnotesize
\tableofcontents
}

\section*{Introduction}
\addcontentsline{toc}{section}{Introduction}

The analysis  of the geometric properties of a smooth manifold $M$ equipped with a so-called metric-affine structure, namely a metric tensor $g$ and an affine connection $\nabla$, originates in several areas of differential geometry and mathematical physics, encompassing classical and quantum information geometry \cite{Jencova-2001, Nagaoka-1995, C-DC-I-M-2023, F-N-2023}, theories of gravity \cite{Hehl-Obukhov-2003}, models of continuous media with topological defects \cite{K-E-1983,Y-G-2012}, and information theory (see \cite{A-N-2000,A-J-L-S-2017}).

What is referred to as the inverse problem on a given metric-affine manifold $(M,g,\nabla)$ is to analyse whether a geometric potential %(i.e. a suitable tensor on $M$) 
exists, such that, together with an algorithm, provides both the metric and the affine connection\footnote{We find it interesting to notice that the quest for a suitable potential is indeed not limited to the metric-affine setting. Consider the case of a K\"ahlerian manifold $(M,g,\omega)$, where both the Hermitian metric and the symplectic form are encoded into a suitable local $\mathbb{C}$-valued potential on $M$ (see, for a modern reference, \cite{Huybrechts2005}), the case of a Lagrangian system on a tangent bundle $\T M$, where the pre-symplectic form is obtained by fiberwise differentiation of an $\mathbb{R}$-valued function (see \cite{M-F-LV-M-R-1990}), or the case of canonical transformations in classical mechanics, which can be described in terms of two-point functions on a phase space (see \cite{C-G-M-M-M-L-2006}).}.

This problem was first studied, under different formulations, in general relativity for Riemannian and pseudo-Riemannian manifolds equipped with their Levi-Civita connection, in terms of suitable two-point functions on $M\times M$.
%and in information geometry for the so called \emph{dually flat statistical manifolds}, where distinguished geometric potentials were constructed directly from the underlying geometric data.} 
In the Riemannian setting, where $g$ is positive definite and $\nabla$ is the corresponding Levi-Civita connection, one half of the square Riemannian distance
\begin{equation}
    \sigma(m,n)=\frac{1}{2}d_g^2(m,n)
\end{equation}
is smooth as a map defined on pairs $(m,n)$ of points on $M$ which are suitably close, and determines (see \cite{KobayashiNomizu1963}) the metric $g$ and therefore its Levi-Civita connection. When $g$ is pseudo-Riemannian, the so-called Synge's world function, defined for suitably close pairs $(m,n)$ of points of $M$ joined by a unique affinely parametrized geodesic $\gamma_{n\leftarrow m}\colon[0,1]\to M$ by
\begin{equation}
\Omega(m,n)=\frac12\int_0^1 g_{\gamma_{n\leftarrow m}(t)}\left(\dot\gamma_{n\leftarrow m}(t),\dot\gamma_{n\leftarrow m}(t)\right)\,\mathrm dt\,,
\end{equation}
is again a two-point function and determines (see \cite{Ruse1931,Synge1931}) the metric $g$.

Within the setting of the so-called dually flat statistical manifolds, namely geometries $(M,g,\nabla)$ such that $\nabla$ and its $g$-dual connection $\nabla^\dagger$ are curvature-free and torsion-free, it is known that this metric-affine structure is generated by the canonical divergence introduced in \cite{NagaokaAmari1982}. In a system of $\nabla$-affine coordinates $(\xi^i)$ and the corresponding dual system of $\nabla^\dagger$-affine coordinates $(\eta_i)$, with associated Legendre-dual potentials $\psi$ and $\varphi$, the canonical divergence is again a two-point function, namely 
\begin{equation}
D(m,n)=\psi\left(\xi(m)\right)+\varphi\left(\eta(n)\right)-\xi^i(m)\,\eta_i(n)\,.
\end{equation}
Through Legendre duality, this expression can equivalently be written as a Bregman divergence and, for regular exponential families of probability distributions, reduces to the Kullback-Leibler relative entropy.

This is the pioneering result %which originates and shapes
that led to the notion of contrast function (see \cite{Eguchi1992,Eguchi1985}) in information geometry, namely suitable smooth real-valued functions on $M\times M$, which %are intended to 
generate a metric $g$ together with a pair of $g$-conjugate connections. When they exist, they provide measures of separation between probability distributions and form the basis of methods of statistical inference \cite{Pardo-2006}, including parameter estimation, hypothesis testing, and model selection. %Their geometric role stems from a construction based on the canonical product structure of $M\times M$, which associates to a contrast function a metric $g$ together with a pair of $g$-conjugate torsion-free affine connections $\nabla$ and $\nabla^\dagger$ \cite{Eguchi1985}. The resulting triple $(M,g,\nabla)$ is called a statistical manifold \cite{Lauritzen-1987}.

The inverse problem for $(M,g,\nabla)$ where both connections $\nabla, \nabla^\dagger$ have vanishing torsion, but there are neither constraints on the curvatures of any of the connections, or metric-compatibility assumptions, was studied and solved by Matumoto in \cite{Matumoto-1993}. His construction is performed in product square charts and yields local solutions that are extended by means of cut-off functions and then glued together via a partition of unity. Since the inverse problem for statistical manifolds admits infinitely many solutions, further analysis has focused on explicit and intrinsic constructions determined directly by the dualistic geometry of the statistical manifold \cite{C-DC-M-2017,C-DC-F-M-M-PP-2017,FeliceAy2018}. Given the paradigmatic role of the self-dual and dually flat cases, a further requirement is that such constructions recover Synge's world function and the canonical divergence, respectively (see, for example, the discussion in \cite[4.4.2]{A-J-L-S-2017}). Among these, we recall two constructions in particular. The first is the contrast function introduced by Ay and Amari \cite{A-A-2015},
\begin{equation}\label{Eq: AA introduction}
D(m,n)=\int_0^1 t\,g_{\gamma_{n\leftarrow m}(t)}\left(\dot\gamma_{n\leftarrow m}(t),
\dot\gamma_{n\leftarrow m}(t)\right)\,\mathrm dt\,.
\end{equation}
defined for pairs $(m,n)$ joined by a unique $\nabla$-geodesic segment $\gamma_{n\leftarrow m}$. The second is the construction of Henmi and Kobayashi \cite{H-R-2000}, involving both conjugate connections,
\begin{equation}\label{Eq: Hook law introduction}
C(m,n)=-\int_0^1 g_{\gamma^\dag_{n\leftarrow m}(t)}\left(\dot\gamma^\dag_{n\leftarrow m}(t),\left.\frac{\partial}{\partial s}\right|_{s=0}H(t,s)\right)\,\mathrm dt\,,
\end{equation}
defined for pairs $(m,n)$ joined by a unique $\nabla^\dag$-geodesic segment $\gamma^\dag_{n\leftarrow m}$ and such that, for every $t\in[0,1]$, there exists a unique $\nabla$-geodesic segment $H(t,\cdot)$ from $\gamma^\dag_{n\leftarrow m}(t)$ to $m$.

The setting analysed above has been extended, in order to accommodate models arising in quantum information geometry \cite{Kurose07,Jencova-2001,C-DC-I-M-2023} and metric-affine theories of gravitation \cite{HehlEtAl1995}, in which one or both of the $g$-conjugate affine connections $\nabla$ and $\nabla^\dagger$ have non-vanishing torsion.
In the case of SMATs, namely triples $(M,g,\nabla)$ for which the $g$-conjugate connection $\nabla^\dagger$ is torsion-free\footnote{In the rest of this paper, we adopt the dual convention and call a triple $(M,g,\nabla)$ a SMAT when $\nabla$ is torsion-free.}, geometric potentials are given by smooth functions $\rho\colon M\times\T M\to\mathbb R$ that are fiberwise linear in the tangent variable, called pre-contrast functions \cite{H-M-2011,H-M-2019}. In the general setting, where no torsion constraint is imposed on either $\nabla$ or $\nabla^\dagger$, the triple $(M,g,\nabla)$ is referred to as a Lauritzen manifold \cite{CMPSZ2026}. Its metric-affine structure can be generated by smooth functions $K\colon\T M\times\T M\to\mathbb R$ that are fiberwise bilinear in the tangent variables, called super-contrast functions \cite{Z-K-2020}.

Explicit solutions to the inverse problem for metric-affine structures with non-vanishing torsion are known only under additional assumptions. In the SMAT setting, Henmi and Matsuzoe considered partially flat SMATs \cite{H-M-2019}, namely SMATs $(M,g,\nabla)$ for which the conjugate connection $\nabla^\dag$ is torsion-free and curvature-free. For such structures, they introduced the pre-contrast function
\begin{equation}\label{Eq: HM}
\rho(m,v;n)=-g_n\left(v,\dot\gamma^\dag_{n\leftarrow m}(0)\right)\,,
\end{equation}
defined for pairs $(m,n)$ joined by a unique $\nabla^\dag$-geodesic segment $\gamma^\dag_{n\leftarrow m}$. Since this expression is defined intrinsically in terms of the metric and the conjugate connection, the authors conjectured that the same formula solves the inverse problem for arbitrary SMATs, without any curvature assumption \cite{H-M-2019}.

In this paper, we address the inverse problem for arbitrary metric-affine structures $(M,g,\nabla)$, without imposing any torsion or curvature assumptions. In our previous papers \cite{L-M-M-V-V-2018, CMPSZ2025, CMPSZ2026}, we showed that there is a unified description of geometric potentials in information geometry in terms of suitable bi-forms of type $(1,1)$, namely bilinear pairings smoothly depending on two points $m,n$ in $M$,
\begin{equation}
\varpi_{(m,n)}\colon\T_mM\times\T_nM\longrightarrow\mathbb R\,,\qquad (m,n)\in M\times M\,.
\end{equation}
The specific cases of SMATs and of statistical manifolds can be studied in terms of  the canonical bicomplex on bi-forms given by 
\begin{align}
&\mathrm d^L\colon\Omega^{p,q}(M\mid M)\longrightarrow\Omega^{p+1,q}(M\mid M)\,,\\
&\mathrm d^R\colon\Omega^{p,q}(M\mid M)\longrightarrow\Omega^{p,q+1}(M\mid M)\,,
\end{align}
where $\Omega^{p,q}(M\mid M)$ denotes the space of bi-forms of type $(p,q)$. More precisely, the connection $\nabla$ and its $g$-conjugate $\nabla^\dag$ induced by a contrast bi-form $\varpi$ are torsion-free precisely when $\mathrm d^L\varpi$ and $\mathrm d^R\varpi$, respectively, vanish along the diagonal submanifold of $M\times M$. Consequently, left-exact contrast bi-forms induce SMAT structures, while bi-exact contrast bi-forms induce statistical structures. The converse does not hold for a fixed contrast bi-form. Nevertheless, up to statistical equivalence, every SMAT is generated by a left-exact contrast bi-form, and every statistical manifold by a bi-exact one. This reconstruction is obtained through statistical homotopy operators
\begin{align}
&J_L\colon\Omega^{p,q}_{\Delta_M}(M\mid M)\longrightarrow\Omega^{p-1,q}_{\Delta_M}(M\mid M)\,,\\
&J_R\colon\Omega^{p,q}_{\Delta_M}(M\mid M)\longrightarrow\Omega^{p,q-1}_{\Delta_M}(M\mid M)\,,
\end{align}
where $\Delta_M\hookrightarrow M\times M$ denotes the diagonal embedding of $M$ into $M\times M$ and $\Omega^{p,q}_{\Delta_M}(M\mid M)$ denotes the space of bi-forms defined on a neighbourhood of the diagonal $\Delta_M$. These operators invert $\mathrm d^L$ and $\mathrm d^R$ in a homotopical sense and reduce the inverse problems for SMATs and statistical manifolds to the inverse problem for arbitrary metric-affine structures. Once a contrast bi-form generating a given Lauritzen manifold $(M,g,\nabla)$ has been constructed, the corresponding pre-contrast and contrast functions are recovered by homotopical inversion through the pair $(J_L,J_R)$.

%\blue{QUESTO CAPOVERSO SEMBRA UN PO' BUTTATO A CASO, VA LEGATO MEGLIO CON IL RESTO E MOTIVATO MEGLIO. TRA L'ALTRO FA PERDERE UN PO' IL FILO DI QUELLO CHE STAI DICENDO, OVVERO "IN THIS PAPER WE...", UN PO' COME FACEVA PRIMA LA SEZIONE SUI PARALLELISMI CHE C'ERA NELLA SEZIONE 2. SECONDO ME QUESTO CAPOVERSO VA RISCRITTO LEGANDOLO MEGLIO E MOTIVATO MEGLIO.} 

%Two-point descriptions of affine connections occur in 
The solution to the inverse problem that we describe in this paper comes upon  encoding the affine connection $\nabla$ into a suitable two-point tensor, developing an idea from relativity \cite{Synge-1960}, optimization on manifolds \cite{A-M-S-2008}, and the analysis of differential operators on vector bundles \cite{B-G-V-1992}. %They are
Such a tensor is given by $\mathbb R$-linear isomorphisms
\begin{equation}
P_{(m,n)}\colon \T_mM\longrightarrow \T_nM
\end{equation}
smoothly depending on pairs $(m,n)$ in a neighbourhood of the diagonal and satisfying $P_{(m,m)}=\id_{\T_mM}$. We refer to such tensors as \underline{parallelisms} or \underline{direct connections} \cite{KubarskiTeleman2006,Teleman2007,AzzaliBoutaibFrabettiPaycha2022}. One way of encoding an affine connection in a parallelism is to turn the curve dependence of its parallel transport into a two-point dependence. This is achieved by choosing smoothly, for every pair of sufficiently close points $(m,n)$, a curve $\Gamma_0(m,n,\cdot)$ joining $m$ to $n$ and reducing to the constant curve when $m=n$. Parallel transport with respect to $\nabla$ along the chosen curves then defines the parallelism
\begin{equation}
P^{\nabla,\Gamma_0}_{(m,n)}\colon \T_mM\longrightarrow \T_nM\,.
\end{equation}
A concrete choice of $\Gamma_0$ is provided by geodesic interpolation with respect to an auxiliary affine connection $\nabla_0$, whose construction follows from the theory of strongly $\nabla_0$-convex neighbourhoods. The resulting parallelism $P^{\nabla,\nabla_0}$ is therefore defined by $\nabla$-parallel transport along the corresponding $\nabla_0$-geodesic segments and determines the affine connection $\nabla$. Pairing $P^{\nabla,\nabla_0}$ to the metric $g$ yields an explicit and intrinsic contrast bi-form that generates the given metric-affine structure $(M,g,\nabla)$ and therefore solves the inverse problem. When $\nabla$ is curvature-free, this construction recovers the solution bi-form obtained in \cite{CMPSZ2026}.

%\blue{CON L'INTERMEZZO FATTO PRIMA SUI PARALLELISMI QUI SI PERDE UN PO' IL FILO SU SE QUESTO "ARE THEN ADDRESSED BY MEANS..." SI RIFERISCA A QUESTO PAPER O AD UN ALTRO. RENDERLO PIU' CHIARO (MA PROBABILMENTE RISTRUTTURANDO BENE IL CAPOVERSO PRECEDENTE SI RISOLVE ANCHE QUESTO)} 

%The inverse problems for SMATs and statistical manifolds are then addressed by means of statistical homotopy operators. 
The contrast bi-forms constructed above yield solutions to the inverse problems for SMATs and statistical manifolds in terms of pre-contrast and contrast functions, respectively, through statistical homotopy operators. Explicit formulas for these geometric potentials are obtained by choosing auxiliary affine connections, which determine explicit realizations of the homotopy operators \cite{CMPSZ2026}: an affine connection $\nabla_1$ determines a left homotopy operator $J_L^{\nabla_1}$, whereas an affine connection $\nabla_2$ determines a right homotopy operator $J_R^{\nabla_2}$. In the SMAT setting, applying $J_L^{\nabla_1}$ to the solution bi-form yields pre-contrast functions depending on the data $(g,\nabla,\nabla_0,\nabla_1)$. For a suitable choice of the auxiliary connections, this construction reproduces the pre-contrast function introduced by Henmi and Matsuzoe for partially flat SMATs and proves their conjecture that the same expression solves the inverse problem for arbitrary SMATs \cite{H-M-2011}. In the statistical manifold setting, applying both homotopy operators yields contrast functions of the form $J_R^{\nabla_2}J_L^{\nabla_1}\varpi$, depending on $(g,\nabla,\nabla_0,\nabla_1,\nabla_2)$. Appropriate choices of the auxiliary connections recover the divergences of Ay and Amari \cite{A-A-2015} and Henmi and Kobayashi \cite{H-R-2000}.

The paper is organized as follows. In \cref{Sec: two-point formalism in IG}, we introduce the two-point formalism used throughout the paper. We recall the classes of metric-affine structures considered in information geometry, develop the bicomplex of bi-forms near the diagonal, and describe contrast bi-forms together with the metric and conjugate connections they induce. We also introduce parallelisms and explain how affine connections can be recovered from them. In \cref{Sec: inverse problem in Information Geometry}, we formulate the inverse problem in the setting of bi-forms. After proving a general existence result, we construct an explicit solution by pairing the given metric to a parallelism obtained from the parallel transport of the given affine connection. Statistical homotopy operators are then used to reduce this solution to pre-contrast functions for SMATs and to contrast functions for statistical manifolds. In \cref{Sec: examples}, we consider reductive homogeneous metric-affine manifolds and derive explicit solution bi-forms for canonical affine connections on reductive spaces, Cartan-Schouten connections on semisimple Lie groups \cite{CartanSchouten, Postnikov-2001}, and Berger metric-affine structures on odd-dimensional spheres \cite{Berger1961, D-G-P-2016}.

\section{Two-point formalism in Information Geometry}\label{Sec: two-point formalism in IG}
In this section we briefly describe the two-point tensor approach to the analysis of potentials in information geometry, introduced in \cite{L-M-M-V-V-2018, CMPSZ2025, CMPSZ2026}. 

\subsection{Statistical manifolds and their generalisations}

Recall the notion of dualistic structure, which underlies the geometry of statistical manifolds and their generalizations. Let $M$ be a smooth manifold endowed with a (possibly %) 
pseudo-)Riemannian metric $g$. A pair of affine connections $(\nabla,\nabla^\dagger)$ on $M$ is called a $g$-\emph{dualistic structure} if
\begin{equation}\label{Eq: compatibility condition}
\mathcal L_Z\left(g(X,Y)\right)=g\left(\nabla_ZX,Y\right)+g\left(X,\nabla^\dagger_ZY\right)\,, \qquad X,Y,Z\in \mathfrak X(M)\,.
\end{equation} 
Since $g$ is non-degenerate, the connection $\nabla^\dagger$ is uniquely determined by $(g,\nabla)$ and is called the $g$-\emph{conjugate connection} to $\nabla$. The above expression shows that the connection $\nabla$ is $g$-conjugate to   $\nabla^\dagger$, so that $(\nabla^\dagger)^\dagger=\nabla$. If $\nabla=\nabla^\dagger$, then $\nabla$ is said to be $g$-\emph{compatible}.

\begin{comment}
The compatibility condition \eqref{Eq: compatibility condition} admits equivalent formulations in terms of covariant derivatives and parallel transports along smooth curves. In particular, if $\gamma\colon I\to M$ is a smooth curve, then
\begin{equation}\label{Eq: compatibility condition covariant derivative}
\frac{\mathrm d}{\mathrm dt}\left(g(E,F)\circ\gamma\right)
=
g\left(\mathrm D_t^{\nabla,\gamma}E,F\right)
+
g\left(E,\mathrm D_t^{\nabla^\dagger,\gamma}F\right)
\end{equation}
for all vector fields $E,F$ along $\gamma$. Equivalently, for every $t_0,t_1\in I$, $v\in\T_{\gamma(t_0)}M$, and $w\in\T_{\gamma(t_1)}M$, one has
\begin{equation}\label{Eq: compatibility condition parallel transport}
g_{\gamma(t_0)}\left(v,P^{\nabla^\dagger,\gamma}_{t_0\leftarrow t_1}w\right)
=
g_{\gamma(t_1)}\left(P^{\nabla,\gamma}_{t_1\leftarrow t_0}v,w\right)\,.
\end{equation}
\end{comment}

The classes of dualistic structures considered in information geometry are distinguished by the torsion properties of the conjugate connections $\nabla, \nabla^\dagger$. Recall that the torsion tensor of an affine connection $\nabla$ is defined by
\begin{equation}\label{Eq: torsion tensor}
\Tor^\nabla(X,Y)=\nabla_XY-\nabla_YX-[X,Y]\,,
\qquad X,Y\in\mathfrak X(M)\,.
\end{equation}
The connection $\nabla$ is called torsion-free if $\Tor^\nabla=0$. Since the torsion tensors of $\nabla$ and $\nabla^\dagger$ may vanish independently, the torsion properties of a pair of  conjugate connections lead to the following classification \cite{Amari-1985, Amari-2016, H-M-2011, Kurose07}.
\begin{definition}\label{Def: metric-affine manifolds}
Let $M$ be a smooth manifold. A triple $(M,g,\nabla)$, where $g$ is a pseudo-Riemannian metric and $\nabla$ is an affine connection, is called a \underline{Lauritzen manifold}, or a \underline{metric-affine manifold}. In particular, it is called
\begin{enumerate}
\item a \underline{statistical manifold admitting torsion}, or simply a \underline{SMAT}, if $\nabla$ is torsion-free;
\item a \underline{statistical manifold} if both $\nabla$ and $\nabla^\dagger$ are torsion-free.
\end{enumerate}
\end{definition}

\begin{remark}\label{Remark: other names}
The terminology for the geometric structure considered above is not uniform. In classical affine differential geometry, the description in terms of a pseudo-Riemannian metric and a pair of conjugate affine connections appears under the name of \emph{Norden-Sen geometry} \cite{Norden1937,Sen1948,Kostecki2026}. In information geometry, the same structure is also referred to as a \emph{dualistic geometry} \cite{Z-K-2020}. The term \emph{metric-affine manifold} is attested in general relativity and gravity \cite{Hehl-Obukhov-2003}.
\end{remark}

\begin{remark}\label{Remark: smat}
The definition of statistical manifold admitting torsion adopted here differs from that given in \cite{H-M-2011, H-M-2019}, where the connection $\nabla^\dagger$ is required to be torsion-free instead of the connection $\nabla$. Since $(\nabla^\dagger)^\dagger=\nabla$, the two conventions are equivalent up to interchanging the roles of $\nabla$ and $\nabla^\dagger$.
\end{remark}

Unlike the  torsion tensor, the curvature tensor is constrained by duality. Recall that the curvature tensor of an affine connection $\nabla$ is defined by
\begin{equation}\label{Eq: curvature tensor}
R^\nabla(X,Y)\,Z =\nabla_X\nabla_YZ-\nabla_Y\nabla_XZ-\nabla_{[X,Y]}Z\,.
\end{equation}
The connection $\nabla$ is called curvature-free if $R^\nabla=0$. The curvature tensor of $\nabla$ and $\nabla^\dag$ satisfy the equality
\begin{equation}\label{Eq: dual curvature relation}
g\left(R^{\nabla^\dagger}(X,Y)Z,W\right)
=
-g\left(Z,R^\nabla(X,Y)W\right)\, \qquad X,Y,Z,W\in\mathfrak X(M)\,,
\end{equation} 
thus showing that  $\nabla$ is curvature-free if and only if $\nabla^\dagger$ is curvature-free.

\begin{definition}\label{Def: dually curvature-free metric-affine manifolds}
A Lauritzen manifold $(M,g,\nabla)$ is called \underline{dually curvature-free} if $\nabla$, or equivalently $\nabla^\dagger$, is curvature-free. In particular, it is called
\begin{enumerate}
\item a \underline{partially flat SMAT} if $\nabla$ is torsion-free and curvature-free;
\item a \underline{dually flat statistical manifold} if both $\nabla$ and $\nabla^\dagger$ are torsion-free and curvature-free.
\end{enumerate}
\end{definition}

As we recalled in the introduction, suitable potentials generate the metric-affine structures introduced above through algorithmic procedures based on the canonical product structure of $M\times M$. %. By a suitable algorithm on tensorial terms depending on a pair of points in $M$, 
Contrast functions generate statistical manifolds, whereas pre-contrast and super-contrast functions generate SMATs and Lauritzen manifolds, respectively. This suggests (\cite{CMPSZ2025,CMPSZ2026}) to analyse information geometry in terms of suitably defined bi-forms.

\subsection{Bi-forms}

Let $M$ be a smooth manifold, let $\pi_L\,\colon\, M\times M\,\to\, M$ %(given by $\pi_L(m,n)=m$) 
and $\pi_R\,\colon\,M\times M\,\to\,M$ %(given by $\pi_R(m,n)=n$) 
denote the canonical surjective submersions onto the left and right component. For $(p,q)\in\mathbb N_0\times\mathbb N_0$, a $(p,q)$-\underline{bi-form} on $M$ is a smooth section of the vector bundle
\begin{equation}
\label{28.07.2}
\bigwedge^p\pi_L^\ast\T^\ast M\otimes_{M\times M}\bigwedge^q\pi_R^\ast\T^\ast M\longrightarrow M\times M\,.
\end{equation}
We denote by $\Omega^{p,q}(M\mid M)$ the $C^\infty(M\times M)$-module of $(p,q)$-bi-forms, i.e. smooth sections of the vector bundle above.
%\red{As it is customary for diff form on manifolds} 
Equivalently, a bi-form $\varpi\in\Omega^{p,q}(M\mid M)$ assigns smoothly to every $(m,n)\in M\times M$ a bilinear pairing
\begin{equation}
\varpi_{(m,n)}\colon \bigwedge^p\T_mM\times\bigwedge^q\T_nM\longrightarrow\mathbb R\,.
\end{equation}
It can therefore be identified with a smooth fiberwise multilinear map
\begin{equation}
\dot\varpi\colon
\underbrace{\T M\times_M\cdots\times_M\T M}_{p\ \mathrm{times}}
\times
\underbrace{\T M\times_M\cdots\times_M\T M}_{q\ \mathrm{times}}
\longrightarrow\mathbb R
\end{equation}
defined by
\begin{equation}
\dot\varpi\left((m,v_1),\ldots,(m,v_p)\mid(n,w_1),\ldots,(n,w_q)\right)
=
\varpi_{(m,n)}\left(v_1,\ldots,v_p\mid w_1,\ldots,w_q\right)\,,
\end{equation}
for every $(m,n)\in M\times M$, $\{v_i\}_{i=1}^p\subset \T_mM$ and $\{w_j\}_{j=1}^q\subset \T_nM$.
Alternatively, a $(p,q)$-bi-form $\varpi$ can be identified with a pairing
\begin{equation}
\label{28.07.3}
\varpi\colon \mathfrak X(M)^p\times\mathfrak X(M)^q\longrightarrow C^\infty(M\times M)\,,
\end{equation}
denoted with the same symbol, which is independently alternating in the left and right entries and satisfies
\begin{align}
\varpi\left(X_1,\ldots,f\,X_i,\ldots,X_p\mid Y_1,\ldots,Y_q\right)&=
\left(\pi_L^\ast f\right)\,
\varpi\left(X_1,\ldots,X_p\mid Y_1,\ldots,Y_q\right)\,,\label{Eq: left tensoriality}\\
\varpi\left(X_1,\ldots,X_p\mid Y_1,\ldots,h\,Y_j,\ldots,Y_q\right)&=
\left(\pi_R^\ast h\right)\,
\varpi\left(X_1,\ldots,X_p\mid Y_1,\ldots,Y_q\right)\,,\label{Eq: right tensoriality}
\end{align}
for every $f,h\in C^\infty(M)$, $i\in\{1,\ldots,p\}$, and $j\in\{1,\ldots,q\}$.
The correspondence is given by
\begin{equation}
\varpi\left(X_1,\ldots,X_p\mid Y_1,\ldots,Y_q\right)(m,n)
=\varpi_{(m,n)}\left((X_1)_m,\ldots,(X_p)_m\mid(Y_1)_n,\ldots,(Y_q)_n\right)\,,
\end{equation}
for every $\{X_i\}_{i=1}^p,\{Y_j\}_{j=1}^q$ in $\mathfrak X(M)$ and $(m,n)\in M\times M$.

A basic class of examples is provided by the (so-called) pure or decomposable bi-forms. Given differential forms $\alpha\in\Omega^p(M)$ and $\beta\in\Omega^q(M)$, their external tensor product is the $(p,q)$-bi-form that we denote as 
\begin{equation}
\label{28.07.1}
\alpha\boxtimes\beta=\pi_L^\ast\alpha\otimes\pi_R^\ast\beta\,.
\end{equation}
Notice that the above set gives a $C^\infty(M\times M)$-bimodule system of generators for $\Omega^{p,q}(M\mid M)$. 

The basic operations on bi-forms that will be used throughout the paper are a diagonal restriction, a duality, the pullback, together with the left and right exterior derivatives, which we describe as follows.

\begin{itemize}
\item 
Let $\iota\colon M\to M\times M$ be the diagonal embedding, i.e. $\iota(m)=(m,m)$. The \emph{diagonal restriction} of a $(p,q)$-bi-form $\varpi$ on $M$ is the $(p+q)$-covariant tensor $\varpi_{\restriction\Delta_M}$ on $M$ defined by
\begin{equation}\label{Eq: diagonal restriction}
\varpi_{\restriction\Delta_M}(X_1,\ldots,X_p,Y_1,\ldots,Y_q)
=
\iota^\ast\left(\varpi(X_1,\ldots,X_p\mid Y_1,\ldots,Y_q)\right)\,.
\end{equation}

\item Let $\dagger\colon M\times M\to M\times M$ be the flip map, $\dagger(m,n)=(n,m)$. The \emph{dual bi-form} of  a $(p,q)$-bi-form $\varpi$ on $M$ is the $(q,p)$-bi-form $\varpi^\dagger$ defined by
\begin{equation}\label{Eq: dual bi-form}
\varpi^\dagger(Y_1,\ldots,Y_q\mid X_1,\ldots,X_p)
=
\dagger^\ast\left(\varpi(X_1,\ldots,X_p\mid Y_1,\ldots,Y_q)\right)\,.
\end{equation}

\item Given a smooth map $F\colon M_0\to M$, the pullback of the $(p,q)$-bi-form $\varpi$ on $M$ is the $(p,q)$-bi-form $F^\ast\varpi$ on $M_0$ defined in terms of the action of the tangent map $F'$ at each point on $M_0$ by
\begin{align}
\left(F^\ast\varpi\right)_{(m_0,n_0)}
\left(v_1,\ldots,v_p\mid w_1,\ldots,w_q\right)=\varpi_{\left(F(m_0),F(n_0)\right)}
\left(F'_{m_0}v_1,\ldots,F'_{m_0}v_p\mid F'_{n_0}w_1,\ldots,F'_{n_0}w_q\right)\,,
\label{Eq: pullback bi-form}
\end{align}
for every $(m_0,n_0)\in M_0\times M_0$, $\{v_i\}_{i=1}^p\subset \T_{m_0}M_0$, and $\{w_j\}_{j=1}^q\subset \T_{n_0}M_0$.

\item The set $\Omega^{p,q}(M\mid M)$ has a bi-complex structure, with 
\begin{align}
&\mathrm d^L\colon\Omega^{p,q}(M\mid M)\longrightarrow\Omega^{p+1,q}(M\mid M)\,,
\\
&\mathrm d^R\colon\Omega^{p,q}(M\mid M)\longrightarrow\Omega^{p,q+1}(M\mid M)\,,
\end{align}
whose actions, upon $\varpi\in\Omega^{p,q}(M\mid M)$, are defined by
\begin{align}
(\mathrm d^L\varpi)(X_0,\ldots,X_p\mid Y_1,\ldots,Y_q)
&=
\sum_{i=0}^p(-1)^i\mathcal L_{X_i^L}
\left(
\varpi(X_0,\ldots,\widehat X_i,\ldots,X_p\mid Y_1,\ldots,Y_q)
\right)\notag\\
&\quad+
\sum_{0\leq i<j\leq p}(-1)^{i+j}
\varpi([X_i,X_j],X_0,\ldots,\widehat X_i,\ldots,\widehat X_j,\ldots,X_p\mid Y_1,\ldots,Y_q)\,,
\label{Eq: left exterior derivative}
\end{align}
and
\begin{align}
(\mathrm d^R\varpi)(X_1,\ldots,X_p\mid Y_0,\ldots,Y_q)
&=
\sum_{i=0}^q(-1)^i\mathcal L_{Y_i^R}
\left(
\varpi(X_1,\ldots,X_p\mid Y_0,\ldots,\widehat Y_i,\ldots,Y_q)
\right)\notag\\
&\quad+
\sum_{0\leq i<j\leq q}(-1)^{i+j}
\varpi(X_1,\ldots,X_p\mid[Y_i,Y_j],Y_0,\ldots,\widehat Y_i,\ldots,\widehat Y_j,\ldots,Y_q)\,.
\label{Eq: right exterior derivative}
\end{align}
Here $X^L$ and $Y^R$ denote the canonical left and right lifts of $X$ and $Y$ to $M\times M$: if $\Phi^X$ and $\Phi^Y$ are the corresponding local flows on $M$, then the local flows of $X^L$ and $Y^R$ on $M\times M$ are given by
\begin{equation}
\Phi_t^{X^L}(m,n)=\left(\Phi_t^X(m),n\right)\,,\qquad 
\Phi_t^{Y^R}(m,n)=\left(m,\Phi_t^Y(n)\right)\,,
\end{equation}
whenever the right-hand sides are defined. The left and right exterior derivatives satisfy
\begin{align}
\mathrm d^L\,\mathrm d^L&=0\,,
\\
\mathrm d^R\,\mathrm d^R&=0\,,
\\
\mathrm d^L\,\mathrm d^R&=\mathrm d^R\,\mathrm d^L\,.
\end{align} 
\begin{comment}
\item \blue{LA COSA PIU' PULITA DAL PUNTO DI VISTA LOGICO SAREBBE INTRODURRE GLI OPERATORI DI OMOTOPIA QUI, COME UN ALTRO PUNTO DI QUESTO ELENCO. SECONDO ME CI SONO DUE STRADE PERCORRIBILI: 1) (QUELLA CHE PREFERISCO) SI AGGIUNGE UN PUNTO A QUESTO ELENCO DEFINENDO $J_L$ E $J_R$ COME MAPPE DA $\Omega^{(p,q)}$ a $\Omega^{(p-1,q)}$ (RESP. $\Omega^{(p,q-1)}$) "INVERSE" DEI DIFFERENZIALI DESTRI E SINISTRI PROPRIO COME FAI NELLA SEZIONE 1.2.2 (PRATICAMENTE SIGNIFICHEREBBE CHE IL CONTENUTO DELLA SEZIONE 1.2.2 DIVENTA UN PUNTO DI QUESTO ELENCO PUNTATO; NON CREDO CHE QUESTO PUNTO VERREBBE TROPPO PIU' LUNGO DI QUELLO SU $d^L$ E $d^R$) E SI SPOSTA LA 1.2.1 COME APPENDICE; 2) SI LASCIANO LE 1.2.1 E 1.2.2 COME STANNO, MA SI AGGIUNGE COMUNQUE UN PUNTO A QUESTO ELENCO PUNTATO IN CUI ALMENO SI NOMINANO $J_L$ E $J_R$ COME "INVERSI" DEI DIFFERENZIALI DESTRI E SINISTRI, DICENDO CHE LI DEFINIAMO E COSTRUIAMO SOTTO DOPO AVER INTRODOTTO GLI STRONGLY CONVEX COVERINGS. RIPETO, PERSONALMENTE PREFERISCO LA PRIMA, E' PIU' PULITA E COMUNQUE NON COSTA NIENTE, MA SE ENTRAMBI PREFERITE LA SECONDA MI ACCONTENTO ANCHE DELLA SECONDA.}
\end{comment}
\end{itemize}
It follows from \eqref{Eq: left tensoriality}-\eqref{Eq: right tensoriality} that all the preceding operations restrict to the space %\blue{(CAMBIARE ESPRESSIONE. ESEMPIO: ALL THESE OPERATIONS ARE GLOBALLY WELL DEFINED ON $M\times M$ AND, THUSE, MAY BE (OBVIOUSLY) ALSO LOCALLY DEFINED AROUND $\Delta_M$. WE WILL REFER TO THE SPACE OF $(p,q)$-BI-FORMS DEFINED ON AN OPEN NEIGHBORHOOD OF $\Delta_M$ AS LOCAL BI-FORMS. ON LOCAL BI-FORMS THE DIFFERENTIAL OPERATORS $d_L$ AND $d_R$ CAN BE INVERTED IN AN HOMOTOPICAL SENSE BY MEANS OF A SUITABLE SYSTEM OF HOMOTOPY OPERATORS, I.E. A PAIR...)} \red{to the space
\begin{equation}
\Omega_{\Delta_M}^{p,q}(M\mid M)
\end{equation}
of $(p,q)$-bi-forms defined on an open neighbourhood of $\Delta_M$ in $M\times M$, which we call \underline{local bi-forms}.

On this space we may introduce left and right homotopy operators. A \underline{system of homotopy operators} for $(\mathrm d^L,\mathrm d^R)$ is a pair $(J_L,J_R)$ of families of $\mathbb R$-linear maps
\begin{align}
    &J_L^{p,q}\colon \Omega^{p,q}_{\Delta_M}(M\mid M)\to \Omega^{p-1,q}_{\Delta_M}(M\mid M)\,, 
    \\
    &J_R^{p,q}\colon \Omega^{p,q}_{\Delta_M}(M\mid M)\to \Omega^{p,q-1}_{\Delta_M}(M\mid M)\,,
\end{align}
such that
\begin{align}
    &J_L^{p-1,q}J_L^{p,q}=0\,,
    \\
    &J_R^{p,q-1}J_R^{p,q}=0\,,
\end{align}
and such that for every $\varpi\in \Omega^{p,q}_{\Delta_M}(M\mid M)$ there exists an open neighbourhood $\mathscr U$ around $\Delta_M$ in $M\times M$, contained in the domain of $\varpi$, on which the following homotopy identities hold:
\begin{equation}\label{Eq: homotopy formula bi-forms L}
\begin{cases}
\begin{aligned}
\mathrm d^L\,J_L^{p,q}\,\varpi+J_L^{p+1,q}\,\mathrm d^L \varpi
&= \varpi\,,
&& (p,q)\in \mathbb N\times \mathbb N_0\,,\\[4pt]
J_L^{1,q}\,\mathrm d^L \varpi
&= \varpi - 1\boxtimes \varpi_{\restriction \Delta_M}\,,
&& (p,q)=(0,q)\,,
\end{aligned}
\end{cases}
\end{equation}
and
\begin{equation}\label{Eq: homotopy formula bi-forms R}
\begin{cases}
\begin{aligned}
\mathrm d^R\,J_R^{p,q}\,\varpi+J_R^{p,q+1}\,\mathrm d^R \varpi
&= \varpi\,,
&& (p,q)\in \mathbb N_0\times \mathbb N\,,\\[4pt]
J_R^{p,1}\,\mathrm d^R \varpi
&= \varpi - \varpi_{\restriction \Delta_M}\boxtimes 1\,,
&& (p,q)=(p,0)\,.
\end{aligned}
\end{cases}
\end{equation}

The homotopy operators introduced above can be constructed explicitly in terms of an auxiliary affine connection and the corresponding smooth geodesic interpolation between sufficiently close points of $M$ within strongly convex neighbourhoods. The same local geometric construction will also be used later to obtain explicit solution bi-forms for the inverse problem. We therefore think it is useful to recall the main properties of strongly convex neighbourhoods, which provide the appropriate domains for both constructions.

\subsubsection{%A digression: s
Strongly convex coverings}\label{Section: strongly convex coverings}

\begin{comment}
{\tt Let $M$ be a smooth manifold endowed with an affine connection $\nabla$. Denote by $\exp^\nabla$ the exponential map corresponding to $\nabla$, and by $\mathscr V_0\subseteq \T M$ its domain. Recall that $(m,v)\in\mathscr V_0$ if and only if the unique $\nabla$-geodesic $\gamma_{(m,v)}$ with initial conditions $\gamma_{(m,v)}(0)=m$ and $\dot\gamma_{(m,v)}(0)=v$ is defined on an open neighbourhood of $[0,1]\subset\mathbb R$, in which case one has 
\begin{equation}
    \exp^\nabla(m,v)=\gamma_{(m,v)}(1)\,.
\end{equation}

 Let $c$ be a point in $M$, and  $U$  an open neighbourhood of $c$ in $M$. We say that $U$ is $\nabla$-\emph{normal} at $c$ if there exists a star-shaped open subset $V\subseteq \T_cM$ containing the zero tangent vector $\mathbf 0_c\in \T_cM$ such that:
\begin{equation}
    U=\exp^\nabla_c(V)\,,
\end{equation}
where $\exp^\nabla_c$ denotes the restriction of $\exp^\nabla$ to 
$\mathscr V_c=\mathscr V_0\cap(\{c\}\times \T_cM)$. 
Equivalently, $U$ is $\nabla$-normal at $c$ if and only if for every 
$n\in U$ there exists a unique $\nabla$-geodesic segment within $U$, denoted 
$\gamma_{n\leftarrow c}^U$, joining $c$ to $n$, that is
\begin{equation}
\gamma_{n\leftarrow c}^U(0)=c,
\qquad
\gamma_{n\leftarrow c}^U(1)=n.
\end{equation}
We say that an open subset $C\subseteq M$ is $\nabla$-\emph{convex} if it is 
$\nabla$-normal at each of its points. In particular, $C$ is $\nabla$-convex if and only if for every $m,n\in C$ there exists a unique $\nabla$-geodesic segment within $C$ joining $m$ with $n$.
}
\end{comment}

Consider an affine connection $\nabla$ on a smooth manifold $M$. %For a given point $m\in M$, consider a curve $\gamma\,\colon\,t\in\,I\,\mapsto\,\gamma(t)\in M$ (with an open interval $I\subseteq\mathbb R$, such that $\gamma(0)=m$) and $\dot\gamma\in\mathbf T_{\gamma(t)}M$. 
A smooth curve $\gamma\colon I\to M$, defined on an open interval $I\subseteq \mathbb R$ with $0\in I$, is called a \underline{$\nabla$-geodesic} through $m=\gamma(0)$ if the velocity vector field $\dot \gamma$\footnote{By the velocity field $\dot\gamma$ of a smooth curve $\gamma\colon I\to M$, we mean the vector field along $\gamma$ defined by $\dot\gamma(t)=\gamma'(t)\in\T_{\gamma(t)}M$ for every $t\in I$.} is $\nabla$-parallel along $\gamma$.
It is well known %that the curve $\gamma$ is defined to be a  \underline{$\nabla$-geodesic} through $m$ if $\nabla_{\dot\gamma}\dot\gamma=0$ for $t\in I$, and that, for fixed $m\in M$ and $v\in\mathbf T_mM$, there exists an open interval $I$ such that a geodesic curve $\gamma$ is well defined on $I$, 
that for every $m\in M$ and $v\in\T_mM$, there exists an open interval $I\subseteq \mathbb R$ containing $0$ and a unique $\nabla$-geodesic $\gamma\colon I\to M$ satisfying %and satisfies 
$\gamma(0)=m$ and $\dot\gamma(0)=v$. 

An affine connection $\nabla$ on a smooth manifold $M$ allows, via geodesics, to have a correspondence between sufficiently close points. Namely, there exist an open neighbourhood $\mathscr U$ around $\Delta_M$ in $M\times M$ and a smooth map
\begin{equation}
\Gamma\colon\mathscr U\times[0,1]\longrightarrow M
\end{equation}
such that, for every $(m,n)\in\mathscr U$, the curve $\Gamma(m,n,\cdot)(t)=\Gamma(m,n,t)$ is a $\nabla$-geodesic segment joining $m$ to $n$.%, with the germ of $\Gamma$ along $\Delta_M\times [0,1]$ depending on $\nabla$ only. %Moreover, any two such maps coincide on a neighbourhood of $\Delta_M\times[0,1]$ and hence determine the same germ along this subset.
%In particular, we define an open subset $C\subseteq M$ to be  $\nabla$-convex if, for any $m,n\in C$, there exists a unique $\nabla$-geodesic segment within $C$ joining $m$ to $n$.

To construct this correspondence, recall that an open subset $C\subseteq M$ is $\nabla$-convex if, for any $m,n\in C$, there exists a unique $\nabla$-geodesic segment within $C$ joining $m$ to $n$. In particular, a partial geodesic interpolation is associated to any $\nabla$-convex open subset $C\subseteq M$:%and a partial geodesic interpolation 
\begin{equation}
\Gamma^C\colon C\times C\times[0,1]\to C\,,\qquad \Gamma^C(m,n,t)=\gamma^C_{n\leftarrow m}(t)\,,
\end{equation}
where $\gamma^C_{n\leftarrow m}$ is the unique $\nabla$-geodesic segment contained in $C$ and joining $m$ to $n$.
%\blue{FORSE PRIMA DI QUESTA FRASE AGGIUNGEREI LA DEFINIZIONE DI $\nabla_0$-CONVEXITY IN UN PAIO DI RIGHE} 
Since $\nabla$-convexity is not, in general, preserved under intersections, the partial interpolation maps associated to an arbitrary covering by $\nabla$-convex open subsets need not agree on intersections. The following example illustrates this phenomenon.

\begin{example}\label{Example: unit circle}
Consider the unit circle $S^1\subseteq \mathbb C$ endowed with the Levi-Civita connection $\nabla$ given by the round metric. Let $C_1, C_2 \subset S^1$ be two connected open arcs such that $C_1 \cap C_2$ has two connected components. If $z,w$ are %(not coincident)
points which belong to different components in the intersection $C_1 \cap C_2$, there %might exist
are two distinct $\nabla$-geodesic segments joining them, depending on whether such geodesics are in $C_1$ or in $C_2$, as illustrated in \cref{Fig: unit circle}. In such a case, the corresponding geodesic interpolations do not agree on $(C_1 \cap C_2)\times (C_1 \cap C_2)\times [0,1]$.

\begin{figure}[H]
\centering
\resizebox{0.3\textwidth}{!}{%
\begin{tikzpicture}[scale=1.2]
    % radii
    \def\r{1}
    \def\rone{1.18}
    \def\rtwo{1.36}

    % angles
    \def\thetaA{70}
    \def\thetaB{290}

    % base circle
    \draw (0,0) circle (\r);
    \node at (135:0.78) {$S^1$};

    % C1
    \draw[dashed, thick, blue]
        (220:\rone) arc[start angle=220,end angle=500,radius=\rone];
    \node[blue, right] at (0:1.2) {$C_1$};

    % C2
    \draw[dashed, thick, red]
        (40:\rtwo) arc[start angle=40,end angle=320,radius=\rtwo];
    \node[red, left] at (180:1.42) {$C_2$};

    % geodesics
    \draw[ultra thick, blue, -{Latex[length=2mm]}]
        (\thetaA:\rone) arc[start angle=430,end angle=\thetaB,radius=\rone];

    \draw[ultra thick, red, -{Latex[length=2mm]}]
        (\thetaA:\rtwo) arc[start angle=\thetaA,end angle=\thetaB,radius=\rtwo];

    % points
    \fill[blue] ({\rone*cos(\thetaA)},{\rone*sin(\thetaA)}) circle (0.9pt);
    \fill[blue] ({\rone*cos(\thetaB)},{\rone*sin(\thetaB)}) circle (0.9pt);

    \fill[red] ({\rtwo*cos(\thetaA)},{\rtwo*sin(\thetaA)}) circle (0.9pt);
    \fill[red] ({\rtwo*cos(\thetaB)},{\rtwo*sin(\thetaB)}) circle (0.9pt);

    \fill ({\r*cos(\thetaA)},{\r*sin(\thetaA)}) circle (0.9pt);
    \fill ({\r*cos(\thetaB)},{\r*sin(\thetaB)}) circle (0.9pt);

    \node[below] at ({cos(\thetaA)},{sin(\thetaA)}) {$z$};
    \node[above] at ({cos(\thetaB)},{sin(\thetaB)}) {$w$};

    % guides
    \draw[dashed,gray]
        ({\r*cos(\thetaA)},{\r*sin(\thetaA)}) --({\rtwo*cos(\thetaA)},{\rtwo*sin(\thetaA)});
    \draw[dashed,gray]
        ({\r*cos(\thetaB)},{\r*sin(\thetaB)}) --({\rtwo*cos(\thetaB)},{\rtwo*sin(\thetaB)});
\end{tikzpicture}}
\caption{The geodesic segments joining $z$ and $w$ relative to the convex sets $C_1$ and $C_2$.}
\label{Fig: unit circle}
\end{figure}
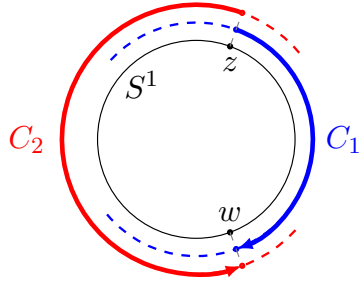

\end{example}

This obstruction is removed by considering only strongly $\nabla$-convex coverings. An open covering $\mathcal C$ of $M$ is called \underline{strongly $\nabla$-convex} if every $C\in\mathcal C$ is $\nabla$-convex and, for any given $C,C'\in\mathcal C$, is the intersection $C\cap C'$  either empty or $\nabla$-convex. What turns out is that every open covering of $M$ admits a strongly $\nabla$-convex refinement (see \cite[Proposition~9 and Remark~12]{Moretti-2021}).

Associated to any strongly %$\nabla_0$
$\nabla$-convex covering $\mathcal C$ is the open neighbourhood
\begin{equation}
\mathscr U_{\mathcal C}=\bigcup_{C\in\mathcal C}\left(C\times C\right)
\end{equation}
around $\Delta_M$ in $M\times M$. An open neighbourhood obtained in this way is called a \underline{strongly $\nabla$-convex neighbourhood}. Such neighbourhoods form a basis of open neighbourhoods around the diagonal submanifold in $M\times M$.

The partial geodesic interpolations associated to a strongly $\nabla$-convex covering $\mathcal C$ glue to a smooth map
\begin{equation}
\Gamma^{\mathcal C}\colon\mathscr U_{\mathcal C}\times[0,1]\longrightarrow M\,,\qquad \Gamma^{\mathcal C}(m,n,t)=\Gamma^C(m,n,t)\,,
\end{equation}
where $C$ is any element of $\mathcal C$ containing both $m$ and $n$. Indeed, if $C'\in\mathcal C$ also contains them, then $C\cap C'$ is $\nabla$-convex, and the geodesic segments joining $m$ to $n$ within $C$ and $C'$ both coincide with the unique segment contained in $C\cap C'$. Moreover, the germ of $\Gamma^{\mathcal C}$ along $\Delta_M\times[0,1]$ depends only on $\nabla$, and not on the particular strongly $\nabla$-convex cover used.

\subsubsection{Homotopy operators on bi-forms}\label{Sec: Homotopy operators on bi-forms}

As a first use of these geodesic interpolations, we construct explicit homotopy operators for the left and right exterior derivatives on local bi-forms.
Let $\nabla_i$, with $i\in\{1,2\}$, be affine connections on $M$, and let
\begin{equation}
\Gamma_i\colon\mathscr U_i\times[0,1]\to M
\end{equation}
be the corresponding $\nabla_i$-geodesic interpolations as in \cref{Section: strongly convex coverings}. We set
\begin{equation}
\gamma^{(i)}_{n\leftarrow m}(t)=\Gamma_i(m,n,t)
\end{equation}
for the $\nabla_i$-geodesic segment from $m$ to $n$.
For every local $(p,q)$-bi-form $\varpi$, the statistical left homotopy operator associated to $\nabla_1$ is defined on $\mathscr U_1$ and given by
\begin{multline}\label{Eq: explicit statistical left homotopy operator}
\left(J_L^{\nabla_1}\varpi\right)(X_2,\ldots,X_p\mid Y_1,\ldots,Y_q)(m,n)\\
=\int_0^1\varpi_{\left(\gamma^{(1)}_{m\leftarrow n}(t),n\right)}\left(\dot{\gamma}^{(1)}_{m\leftarrow n}(t),(X_2)_{\gamma^{(1)}_{m\leftarrow n}(t)},\ldots,(X_p)_{\gamma^{(1)}_{m\leftarrow n}(t)}\mid(Y_1)_n,\ldots,(Y_q)_n\right)\,\mathrm dt\,.
\end{multline}
Similarly, the statistical right homotopy operator associated to $\nabla_2$ is defined on $\mathscr U_2$ and given by
\begin{multline}\label{Eq: explicit statistical right homotopy operator}
\left(J_R^{\nabla_2}\varpi\right)(X_1,\ldots,X_p\mid Y_2,\ldots,Y_q)(m,n)\\
=\int_0^1\varpi_{\left(m,\gamma^{(2)}_{n\leftarrow m}(t)\right)}\left((X_1)_m,\ldots,(X_p)_m\mid\dot{\gamma}^{(2)}_{n\leftarrow m}(t),(Y_2)_{\gamma^{(2)}_{n\leftarrow m}(t)},\ldots,(Y_q)_{\gamma^{(2)}_{n\leftarrow m}(t)}\right)\,\mathrm dt\,.
\end{multline}
After restricting to a sufficiently small common neighbourhood around $\Delta_M$ in $M\times M$, these operators define a system of homotopy operators for $(\mathrm d^L,\mathrm d^R)$.

We can now apply this language to the generation problem of metric-affine structures in information geometry. The next subsection introduces the two-point tensors that generate metrics and affine connections, while encompassing contrast
functions, pre-contrast functions, and super-contrast functions. 

\subsection{Contrast bi-forms}

We say that $\varpi\in \Omega^{1,1}_{\Delta_M}(M\mid M)$ is a \underline{contrast bi-form} if its diagonal restriction
\begin{equation}\label{Eq: metric from bi-forms}
    g^\varpi = \varpi_{\restriction \Delta_M}
\end{equation}
is a pseudo-Riemannian metric on $M$.

\begin{remark}\label{Remark: sign}
    We do not include the conventional minus sign in \eqref{Eq: metric from bi-forms}, which is often adopted in the positive-definite setting; see, for instance, \cite{Eguchi1985,H-M-2011,Z-K-2020}. Our convention is motivated by the fact that we consider also pseudo-Riemannian metrics (see \cref{Sec: CS on Lie groups} and \cref{Sec: Berger metrics}).
\end{remark}

A contrast bi-form also determines an affine connection $\nabla^\varpi$ on $M$ by the identity
\begin{equation}
    g^\varpi\left(\nabla{^\varpi}_Z X,Y\right)
    =
   \iota^\ast\left(\mathcal L_{Z^L}\left(\varpi(X\mid Y)\right)\right)\,, \qquad X,Y,Z\in \mathfrak X(M)\,,
\end{equation}
so providing a Lauritzen manifold %structure 
$(M,g^\varpi,\nabla^\varpi)$.
Contrast bi-forms $\varpi_1$ and $\varpi_2$ are said to be \underline{statistically equivalent} if they generate the same metric-affine structure, namely if one has
\begin{equation}
g^{\varpi_1}=g^{\varpi_2}\,,
\qquad
\nabla^{\varpi_1}=\nabla^{\varpi_2}\,.
\end{equation}
In order to  characterize this property, we recall the notion of vanishing of a covariant tensor along a submanifold. Let  $i\colon S \hookrightarrow Q$, be a smooth injective immersion. We say that an order $p$-covariant tensor $\tau$ on $Q$ \underline{vanishes along} $S$ if $i^*(\tau(X_1,\dots,X_p))=0$ for any set $X_1,\dots,X_p$ of vector fields on $Q$ (notice that such a condition is stronger than requiring $i^*\tau=0$). More generally, for  $k \in \mathbb{N}$, we say that a covariant tensor $\tau$ on $Q$    \underline{vanishes to order $k$ along} $S$ if it vanishes along $S$, and for every vector field $Z$ on $Q$, the Lie derivative $\mathcal{L}_Z \tau$ vanishes to order $(k-1)$ along $S$. This definition can be extended to  pure bi-forms (see \eqref{28.07.1}). We say that the bi-form  $\varpi_{ij}=\alpha_i\boxtimes\beta_j=\pi^*_L\alpha_i\otimes\pi^*_R\beta_j$, with $\alpha_i\in\Omega^p(M)$ and $\beta_j\in\Omega^q(M)$, \underline{vanishes to order $k$ along the diagonal $\Delta_M$} if $\tau_{ij}=\pi^*_L\alpha_i\wedge\pi^*_R\beta_j\in\Omega^{p+q}(M\times M)$ vanishes at order $k$ along the diagonal submanifold $i\colon\Delta_M\hookrightarrow M\times M$. The notion of vanishing of a bi-form along the diagonal $\Delta_M$ comes upon recalling that pure bi-forms provide a %suitable 
system of generators for the module $\Omega^{p,q}(M\mid M)$ and that the above definition can be naturally linearly extended. 
\begin{comment}
Let $\mathcal I_{\Delta_M}\subset\Omega_{\Delta_M}^{0,0}(M\mid M)$ denote the ideal of local functions vanishing on $\Delta_M$. For $k\in\mathbb N_0$, a local $(p,q)$-bi-form $\varpi$ is said to \emph{vanish to order $k$ along $\Delta_M$} if
\begin{equation}
\varpi\in\mathcal I_{\Delta_M}^{k+1}\,\Omega_{\Delta_M}^{p,q}(M\mid M)\,,
\end{equation}
where
\begin{equation}
    \mathcal I_{\Delta_M}^{k+1}=\{F_0\cdots F_k\mid F_j\in \mathcal I_{\Delta_M}\,, j\in \{0,\dots,k\}\}\,.
\end{equation}
For $k=0$, we simply say that $\varpi$ vanishes along $\Delta_M$. With this terminology, $\varpi_1$ and $\varpi_2$ are statistically equivalent if and only if their difference vanishes to first order along $\Delta_M$. Equivalently, in a square product chart $(U\times U,x,y)$ induced by a chart $(U,q)$ on $M$, 
\begin{equation}
\left(\varpi_1-\varpi_2\right)_{\restriction U\times U}=\varpi_{ij,hk}(x,y)\,
\left(x^h-y^h\right)\left(x^k-y^k\right)\,\left(
\mathrm dq^i\boxtimes\mathrm dq^j\right)
\end{equation}
for suitable functions $\varpi_{ij,hk}\in C^\infty(U\times U)$.
\end{comment}

The next result relates the torsion properties of the induced connections to the left and right differential of a contrast bi-form.
\begin{proposition}\label{Prop: torsion bi-form}
Let $\varpi$ be a contrast bi-form on a smooth manifold $M$, and let
$(M,g^\varpi,\nabla^\varpi)$ be the Lauritzen manifold generated by $\varpi$. It is 
\begin{enumerate}[(a)]
    \item $(M,g^\varpi,\nabla^\varpi)$ is a SMAT if and only if
    \begin{equation}
        \big(\mathrm d^L\varpi\big)_{\restriction\Delta_M}=0\,.
    \end{equation}

    \item $(M,g^\varpi,\nabla^\varpi)$ is a statistical manifold if and only if
    \begin{equation}
        \big(\mathrm d^L\varpi\big)_{\restriction\Delta_M}
        =
        \big(\mathrm d^R\varpi\big)_{\restriction\Delta_M}
        =
        0\,.
    \end{equation}
\end{enumerate}
\end{proposition}

%\cref{Prop: torsion bi-form} suggests a cohomological characterization of contrast bi-forms inducing SMATs and statistical manifolds. Indeed, any left-exact contrast bi-form $\mathrm d^L S$ induces a SMAT, since $\mathrm d^L$ is a cohomological operator. Likewise, any bi-exact contrast bi-form $\mathrm d^L\mathrm d^R D$ induces a statistical manifold by the identities of the bi-complex.

Such a \cref{Prop: torsion bi-form} can be inverted by introducing a suitable family of homotopy operators for $(\mathrm d^L,\mathrm d^R)$. So we define a \underline{system of statistical homotopy operators} %is 
as a system of homotopy operators $(J_L,J_R)$ for $(\mathrm d^L,\mathrm d^R)$ such that:
\begin{itemize}
    \item the element $J_L\varpi$ vanishes along $\Delta_M$: furthermore, if $\varpi$ vanishes to order $k$ along $\Delta_M$, then $J_L\varpi$ vanishes to order $k+1$ along $\Delta_M$;
    \item the element $J_R\varpi$ vanishes along $\Delta_M$: furthermore, if $\varpi$ vanishes to order $k$ along $\Delta_M$, then $J_R\varpi$ vanishes to order $k+1$ along $\Delta_M$.
\end{itemize}
The homotopy operators $(J_L^{\nabla_1},J_R^{\nabla_2})$ associated with any pair of affine connections $\nabla_1,\nabla_2$ on $M$, considered in \cref{Sec: Homotopy operators on bi-forms}, also form a system of statistical homotopy operators.

This allows us to give a cohomological characterization of contrast bi-forms inducing SMATs and statistical manifolds.

\begin{proposition}\label{Thm: integration of a contrast bi-form}
Let $\varpi$ be a contrast bi-form on $M$, and let $(J_L,J_R)$ be a system of statistical homotopy operators.
\begin{enumerate}[(a)]
    \item If $(M,g^\varpi,\nabla^\varpi)$ is a SMAT, then $\varpi$ is statistically equivalent to $\mathrm d^LS$, with
    \begin{equation}
        S=J_L\varpi\,.
    \end{equation}

    \item If $(M,g^\varpi,\nabla^\varpi)$ is a statistical manifold, then $\varpi$ is statistically equivalent to $ \mathrm d^L\mathrm d^R D$, with
    \begin{equation}
       D=J_RJ_L\varpi\,.
    \end{equation}
\end{enumerate}
\end{proposition}

We finally recall how the standard potential functions are recovered from these data through the fiberwise formulation. The fiberwise bilinear counterpart of a contrast bi-form $\varpi$,
\begin{equation}
    \dot\varpi\colon \T M\times \T M\to \mathbb R\,,
\end{equation}
is a super-contrast function in the sense of Khan and Zhang \cite{Z-K-2020}. In the SMAT case, the fiberwise linear counterpart of the left primitive $S=J_L^{1,1}\varpi$, that is 
\begin{equation}
    \dot S\colon M\times \T M\to \mathbb R\,,
\end{equation}
is a pre-contrast function in the sense of Henmi and Matsuzoe \cite{H-M-2011,H-M-2019}, up to the convention fixed in \cref{Remark: smat}. Finally, in the statistical case, the potential $D=J_RJ_L\varpi$ is a local $(0,0)$-bi-form, hence a local two-point function, and recovers the classical contrast-function formalism.

This resum\'e shows how the notion of contrast bi-form (and its cohomological properties) encodes the notion of metric tensor and of an affine connection. In order to approach a solution of the (yet to rigorously define!) inverse problem we turn our attention to analyzing how the notion of affine connection comes out from a suitable two-point tensor. 

%The preceding discussion encompasses the bitensorial approach to information geometry through contrast bi-forms. With a view to solving the inverse problem within this framework, we now investigate how an affine connection can be encoded by a suitable two-point tensor.

\subsection{Parallelisms}\label{Section: parallelisms}

An affine connection determines a way of comparing tangent spaces at different points because, as we now recall, every tangent vector at any point on a curve can be uniquely extended to a parallel vector field along an entire curve through that point \cite{LeeRiem}.

Let $\gamma\colon[t_0,t_1]\to M$ be a smooth curve with $\gamma(t_0)=m$ and $\gamma(t_1)=n$. For every $v\in\T_mM$, there exists a unique $\nabla$-parallel vector field $E_v$ along $\gamma$ satisfying
\begin{equation}
\begin{cases}
\mathrm D_t^{\gamma}E_v=0\,, & \textup{on } [t_0,t_1]\,,\\
E_v(t_0)=v\,,
\end{cases}
\end{equation}
where $\mathrm D_t^{\gamma}$ denotes the covariant derivative induced by $\nabla$ along $\gamma$. This allows us to define the
\underline{$\nabla$-parallel transport} along $\gamma$ from $m$ to $n$, namely the map
 $\mathscr P_{t_1\leftarrow t_0}^{\nabla,\gamma}\colon\T_mM\longrightarrow\T_nM\,$  given by 
\begin{equation}
 \mathscr P_{t_1\leftarrow t_0}^{\nabla,\gamma}(v)= E_v(t_1)\,,
\end{equation}
which is an $\mathbb R$-linear isomorphism. Parallel transport depends in general on the curve $\gamma$; it is compatible with restrictions of the curve and reduces to the identity along constant curves, namely,
\begin{equation}
\mathscr P_{t_1\leftarrow t_0}^{\nabla,\gamma}=\mathscr P_{t_1\leftarrow t}^{\nabla,\gamma}\circ\mathscr P_{t\leftarrow t_0}^{\nabla,\gamma}
\end{equation}
for every $t_0\leq t\leq t_1$, and
\begin{equation}
\mathscr P_{t_1\leftarrow t_0}^{\nabla,\gamma}=\id_{\T_mM}
\end{equation}
whenever $\gamma(t)=m$ for every $t\in[t_0,t_1]$.

Several properties of a connection $\nabla$ are encoded in the properties of the parallel transports that it induces. We limit ourselves here to  recall that the connection itself can be recovered from this family through
\begin{equation}\label{Eq: connection from parallel transport}
\left(\nabla_ZX\right)_m=\left.\frac{\mathrm d}{\mathrm dt}\right|_{t=0}\mathscr P_{0\leftarrow t}^{\nabla,\delta}X_{\delta(t)}\,,
\end{equation}
where $Z$ is a vector field on $M$ such that $Z_m\in\mathbf T_mM$,  $\delta$ is \emph{any} smooth curve satisfying $\delta(0)=m$ and $\dot\delta(0)=Z_m$ \cite[Corollary 4.35]{LeeRiem}. Furthermore, a (local) flatness condition for the corresponding curvature is characterized by the homotopy invariance of parallel transport. More precisely, $\nabla$ has a vanishing curvature if and only if
\begin{equation}\label{Eq: homotopy invariance of parallel transport}
\mathscr P_{t_1\leftarrow t_0}^{\nabla,\gamma_0}=\mathscr P_{t_1\leftarrow t_0}^{\nabla,\gamma_1}
\end{equation}
for any pair of piecewise smooth curves $\gamma_0,\gamma_1\colon[t_0,t_1]\to M$ that are homotopic relative to their endpoints \cite[Problem 7-12]{LeeRiem}.

The notion of \underline{parallelism} generalises the constructions by assigning directly an isomorphism between the tangent spaces at each pair of sufficiently close points, without retaining either the affine connection or a connecting curve as part of the data. Such assignments are described as suitable local sections of the Hom-bundle\footnote{Given vector bundles $E,F\longrightarrow B$, the Hom-bundle is the vector bundle $\operatorname{Hom}(E,F)$ whose fiber at $b\in B$ is the vector space of $\mathbb R$-linear morphisms $E_b\longrightarrow F_b$}
\begin{equation}\label{Eq: Hom bundle}
   \operatorname{Hom}(\pi_L^\ast \T M,\pi_R^\ast \T M)\longrightarrow M\times M\,
\end{equation}
whose fiber at $(m,n)\in M\times M$ is the vector space $\operatorname{Hom}(\T_mM,\T_nM)$ of linear maps $\T_mM\longrightarrow \T_nM$.

\begin{comment}
\begin{definition}
Let $M$ be a smooth manifold, and $\mathscr U$ an open neighbourhood around $\Delta_M$. For $(m,n)\in\mathscr U$, a \underline{parallelism} from $m$ to $n$ is an element $P_{(m,n)}\in\mathrm{Iso}(\mathbf T_mM\to\mathbf T_nM)$ such that, for $m=n$, one has $P_{(m,m)}=\mathrm{id}_{\mathbf T_mM}$.
\end{definition}
\end{comment}

\begin{definition}
Let $M$ be a smooth manifold. A \underline{parallelism} on $M$ is a local section $P$ of the vector bundle \eqref{Eq: Hom bundle} defined on an open neighbourhood $\mathscr U$ of $\Delta_M$ in $M\times M$ such that:
\begin{enumerate}[($a$)]
\item $P_{(m,n)}\colon\T_mM\to\T_nM$ is an $\mathbb R$-linear isomorphism for any $(m,n)\in\mathscr U$.
\item\label{Item: normalization} $P_{(m,m)}=\id_{\T_mM}$ for any $m\in M$;
\end{enumerate}
\end{definition}

\begin{remark}\label{rem1}
A parallelism can also be defined as a local section $P$ of the vector bundle
\begin{equation}
\pi_L^*\T^*M\otimes\pi_R^*\T M\longrightarrow M\times M\,, 
\end{equation}
\begin{comment}
that satisfies the following conditions
\begin{enumerate}[(a)]
\item $P_{(m,n)}\colon\T_mM\to\T_nM$ is a $\mathbb R$-linear isomorphism for any $(m,n)\in\mathscr U\subseteq \Delta_M
$.
\item\label{Item: normalization} $P_{(m,m)}=\id_{\T_mM}$ for any  $m\in M$;
\end{enumerate}
Moreover, it is clear that a parallelism on $(m,n)\in\mathscr U$ is a suitable element $P_{(m,n)}\in\Omega^1(M)\otimes_{C^{\infty}(M\times M)}\mathfrak{X}(M)$. A comparison with the analogous definitions for bi-forms (see \eqref{28.07.2}-\eqref{28.07.3}) shows that a parallelism is two-point tensor on $M$.
\end{comment}
This description follows from the canonical vector bundle isomorphisms\footnote{Given vector bundles $E,F\longrightarrow B$, there is a canonical vector bundle isomorphism $\operatorname{Hom}(E,F)\cong E^\ast \otimes_B F$.}
\begin{equation}
    \operatorname{Hom}\left(\pi_L^\ast\T M,\pi_R^\ast\T M\right)
    %\cong \left(\pi_L^\ast\T M\right)^\ast \otimes_{M\times M}\pi_R^\ast\T M
    \cong \pi_L^\ast\T^\ast M\otimes_{M\times M}\pi_R^\ast\T M\,.
\end{equation}
A comparison with the analogous definitions for bi-forms (see \eqref{28.07.2}-\eqref{28.07.3}) shows that a parallelism is a two-point tensor on $M$.
\end{remark}

Any parallelism $P$ on $M$ determines an affine connection $\nabla^P$ by
\begin{equation}\label{Eq: connection from P}
\left(\nabla^P_ZX\right)_m=\left.\frac{\mathrm d}{\mathrm dt}\right|_{t=0}P_{(\delta(t),m)}X_{\delta(t)}\,,
\end{equation}
where $Z$ is a vector field on $M$,  $\delta$ is any smooth curve satisfying $\delta(0)=m$ and $\dot\delta(0)=Z_m$ (see \cite{AzzaliBoutaibFrabettiPaycha2022}). One proves that the right-hand side does not depend on the whole curve $\delta$, and depends linearly only on  $\dot\delta(0)$.  In particular, one proves that  the condition $P_{(m,m)}=\mathrm{id}_{\mathbf T_mM}$ (which is a sort of  normalization of $P$ along $\Delta_M$) ensures the Leibniz rule for the $C^{\infty}(M)$ module structure on the entry $X$, and hence $\nabla^P$ is an affine connection.

\begin{comment}
This construction depends only on the first-order behaviour of $P$ in the left variable along the diagonal. Indeed, let $(U,q)$ be a coordinate chart and write
\begin{equation}
P=P_j^{\,k}(x,y)\,\mathrm d x^j\otimes\frac{\partial}{\partial y^k}\,.
\end{equation}
If
\begin{equation}
\nabla^P_{\partial_{q^i}}\partial_{q^j}=(\Gamma^P)_{ij}^{\,k}\partial_{q^k}\,,
\end{equation}
then
\begin{equation}\label{Eq: Christoffel symbols from P}
(\Gamma^P)_{ij}^{\,k}(c)=\frac{\partial P_j^{\,k}}{\partial x^i}\left(q(c),q(c)\right)\,.
\end{equation}
\end{comment}

The above correspondence between parallelisms and affine connections \emph{is not} injective, since the relation \eqref{Eq: connection from P} depends only on the first-order behaviour of the parallelism along $\Delta_M$, as the following example shows.

\begin{example}\label{Example: parallelisms on R}
Let $M=\mathbb R$ with global coordinate $q$, and consider the parallelisms
\begin{equation}
P=e^{t-s}\,\mathrm ds\otimes\partial_t\,,\qquad \widetilde P=e^{t-s+(t-s)^3}\,\mathrm ds\otimes\partial_t\,.
\end{equation}
Although $P$ and $\widetilde P$ are distinct, the cubic term has vanishing first derivative along the diagonal. Explicit computations read that both parallelisms induce the affine connection given by
\begin{equation}
\nabla^P_{\partial_q}\partial_q=\nabla^{\widetilde P}_{\partial_q}\partial_q=-\partial_q\,.
\end{equation}
\end{example}

Even if not injective, the correspondence given in \eqref{Eq: connection from P} \emph{is} indeed surjective, namely  every affine connection arises from a parallelism through \eqref{Eq: connection from P}. Guided by \eqref{Eq: connection from parallel transport}, we consider a family of curves $\Gamma_0(m,n,\cdot)$ on $M$ which are smoothly parametrized by pairs of sufficiently close points $(m,n)$, each joining $m$ to $n$ and reducing to the constant curve at each $c$ when $m=n=c$. Namely, let $\mathscr U_0$ be an open neighbourhood of $\Delta_M$ in $M\times M$ and consider a smooth map
\begin{equation}
\Gamma_0\colon\mathscr U_0\times[0,1]\longrightarrow M
\end{equation}
satisfying
\begin{equation}
\Gamma_0(m,n,0)=m\,,\qquad \Gamma_0(m,n,1)=n\,,\qquad \Gamma_0(c,c,\cdot)=c\,.
\end{equation}
With such a map, we set
\begin{equation}
P_{(m,n)}^{\nabla,\Gamma_0}=\mathscr P_{1\leftarrow0}^{\nabla,\Gamma_0(m,n,\cdot)}\colon\T_mM\longrightarrow\T_nM\,.
\end{equation}
By the well known properties of the parallel transport map, one easily sees that $P^{\nabla,\Gamma_0}$ is a parallelism. To determine the affine connection induced by it, we use the following variational lemma. We first introduce the notation we need to describe it.

Given intervals $I,J\subseteq\mathbb R$ with $0\in J$ and a smooth map $H\colon I\times J\to M$, we write
\begin{equation}
H(\cdot,s)(t)=H(t,s)\,,\qquad H(t,\cdot)(s)=H(t,s)\,
\end{equation}
to identify curves parametrised by $t$ for each fixed $s$  or  by $s$ for each fixed $t$.
A smooth vector field along $H$ is a smooth map $V\colon I\times J\to\T M$ such that $V(t,s)\in\T_{H(t,s)}M$ for every $(t,s)\in I\times J$. This means that  $V(\cdot,s)$ and $V(t,\cdot)$ are smooth vector fields along $H(\cdot,s)$ and $H(t,\cdot)$, respectively.

\begin{lemma}\label{Lemma: variation of parallel vector fields}
Let $M$ be a smooth manifold endowed with an affine connection $\nabla$, let $H\colon I\times J\to M$ be as above, and let $V$ be a smooth vector field along $H$. Assume that $V(t,\cdot)$ is $\nabla$-parallel along $H(t,\cdot)$ for every $t\in I$. For  any $(t,s)\in I\times J$, it is 
\begin{multline}\label{Eq: variation of parallel vector fields}
\left.\mathrm D_\tau^{H(\cdot,s)}\right|_{\tau=t}V(\tau,s)
-\mathscr P_{s\leftarrow0}^{\nabla,H(t,\cdot)}
\left.\mathrm D_\tau^{H(\cdot,0)}\right|_{\tau=t}V(\tau,0)\\
=-\int_0^s\mathscr P_{s\leftarrow\sigma}^{\nabla,H(t,\cdot)}
R^\nabla_{H(t,\sigma)}\left(\left.\partial_\tau\right|_{\tau=t}H(\tau,\sigma),
\left.\partial_{\sigma'}\right|_{\sigma'=\sigma}H(t,\sigma')\right)V(t,\sigma)\,\mathrm d\sigma\,,
\end{multline}
where $R^\nabla$ is the curvature tensor of $\nabla$.
\end{lemma}

\begin{proof}
The curvature identity for vector fields along $H$ \cite{Wilkins2010} gives
\begin{multline}
R^\nabla_{H(t,s)}\left(\left.\partial_\tau\right|_{\tau=t}H(\tau,s),
\left.\partial_{\sigma'}\right|_{\sigma'=s}H(t,\sigma')\right)V(t,s)\\
=\left.\mathrm D_\tau^{H(\cdot,s)}\right|_{\tau=t}
\left(\left.\mathrm D_{\sigma}^{H(\tau,\cdot)}\right|_{\sigma=s}V(\tau,\sigma)\right)
-\left.\mathrm D_{\sigma}^{H(t,\cdot)}\right|_{\sigma=s}
\left(\left.\mathrm D_\tau^{H(\cdot,\sigma)}\right|_{\tau=t}V(\tau,\sigma)\right)\,.
\end{multline}
Since $V(\tau,\cdot)$ is $\nabla$-parallel along $H(\tau,\cdot)$, the first term on the right-hand side vanishes, and therefore we can write 
\begin{equation}\label{Eq: curvature identity}
\left.\mathrm D_{\sigma}^{H(t,\cdot)}\right|_{\sigma=s}
\left(\left.\mathrm D_\tau^{H(\cdot,\sigma)}\right|_{\tau=t}V(\tau,\sigma)\right)
=-R^\nabla_{H(t,s)}\left(\left.\partial_\tau\right|_{\tau=t}H(\tau,s),
\left.\partial_{\sigma'}\right|_{\sigma'=s}H(t,\sigma')\right)V(t,s)\,.
\end{equation}
For fixed $t\in I$, define the vector field $E$ along $H(t,\cdot)$ by
\begin{equation}
E(\sigma)=\left.\mathrm D_\tau^{H(\cdot,\sigma)}\right|_{\tau=t}V(\tau,\sigma)\,.
\end{equation}
The fundamental theorem of covariant calculus along $H(t,\cdot)$ yields
\begin{equation}
E(s)=\mathscr P_{s\leftarrow0}^{\nabla,H(t,\cdot)}E(0)
+\int_0^s\mathscr P_{s\leftarrow\sigma}^{\nabla,H(t,\cdot)}
\left.\mathrm D_{\sigma'}^{H(t,\cdot)}\right|_{\sigma'=\sigma}E(\sigma')\,\mathrm d\sigma\,.
\end{equation}
Inserting the relation \eqref{Eq: curvature identity} into the formula above gives \eqref{Eq: variation of parallel vector fields}.
\end{proof}

We now compute the affine connection induced by the parallelism $P^{\nabla,\Gamma_0}$. Let $X,Z\in\mathfrak X(M)$, let $m\in M$, and choose a smooth curve $\delta$ such that $\delta(0)=m$ and $\dot\delta(0)=Z_m$. For $t$ in a sufficiently small interval containing $0$, set
\begin{align}
&H(t,s)=\Gamma_0\left(\delta(t),m,s\right),\\
&V(t,s)=\mathscr P_{s\leftarrow0}^{\nabla,H(t,\cdot)}X_{\delta(t)}.
\end{align}
For every $t$, the vector field $V(t,\cdot)$ is $\nabla$-parallel along $H(t,\cdot)$, and
\begin{align}
&V(t,0)=X_{\delta(t)}, \\ &V(t,1)=P_{\left(\delta(t),m\right)}^{\nabla,\Gamma_0}X_{\delta(t)}.
\end{align}
Moreover, $H(0,\cdot)$ is the constant curve at $m$, and hence
\begin{equation}
\left.\partial_{\sigma'}\right|_{\sigma'=\sigma}H(0,\sigma')=0
\end{equation}
for every $\sigma\in[0,1]$. Applying \cref{Lemma: variation of parallel vector fields} at $(t,s)=(0,1)$, we see that the curvature term vanishes and the parallel transport along $H(0,\cdot)$ is the identity. Therefore,
\begin{equation}
\left.\mathrm D_\tau^{H(\cdot,1)}\right|_{\tau=0}V(\tau,1)
=\left.\mathrm D_\tau^{H(\cdot,0)}\right|_{\tau=0}V(\tau,0)\,.
\end{equation}
Since $H(t,1)=m$, the left-hand side is the derivative of $V(t,1)$ in the fixed vector space $\T_mM$, whereas $H(t,0)=\delta(t)$ and $V(t,0)=X_{\delta(t)}$. It follows that
\begin{equation}
\left.\frac{\mathrm d}{\mathrm dt}\right|_{t=0}P_{\left(\delta(t),m\right)}^{\nabla,\Gamma_0}X_{\delta(t)}
=\left.\mathrm D_\tau^\delta\right|_{\tau=0}X_{\delta(\tau)}
=\left(\nabla_ZX\right)_m\,.
\end{equation}
By \eqref{Eq: connection from P}, since $X$, $Z$, and $m$ are arbitrary, we conclude that
\begin{equation}
\nabla^{P^{\nabla,\Gamma_0}}=\nabla\,.
\end{equation}

We conclude by giving a sufficient condition for the induced affine connection to be curvature-free. We say that a parallelism $P$ defined on an open neighbourhood $\mathscr U$ of $\Delta_M$ in $M\times M$ is \underline{multiplicative} if, for every $(m,n)\in\mathscr U$ and every $c\in M$ such that $(m,c),(c,n)\in\mathscr U$, one has
\begin{equation}\label{Eq: P is multiplicative}
P_{(c,n)}\circ P_{(m,c)}=P_{(m,n)}\,.
\end{equation}

\begin{proposition}\label{Prop: curvature-freeness criterion}
    If a parallelism $P$ on a smooth manifold $M$ is multiplicative, then $\nabla^P$ is curvature-free.
\end{proposition}
\begin{proof}
Let $c\in M$, and consider $\mathscr U^{(c)}=\left\{m\in M\mid(c,m)\in\mathscr U\right\}$. Using \eqref{Eq: P is multiplicative}, we construct a local frame of $\nabla^P$-parallel vector fields on $\mathscr U^{(c)}$ by setting
\begin{equation}\label{Eq: parallel v.f. P}
X_v(m)=P_{(c,m)}v\,, \qquad m\in \mathscr U^{(c)}\,,
\end{equation}
with $v\in\T_cM$. Fix $m\in\mathscr U^{(c)}$. If $Z\in\mathfrak X(M)$ and $\delta$ is a smooth curve such that $\delta(0)=m$ and $\dot\delta(0)=Z_m$, then
\begin{align}
\left(\nabla_Z^P X_v\right)_m
\stackrel{\eqref{Eq: connection from P}}{=}\left.\frac{\mathrm d}{\mathrm dt}\right|_{t=0}P_{\left(\delta(t),m\right)}\left(X_v\right)_{\delta(t)}
\stackrel{\eqref{Eq: parallel v.f. P}}{=}\left.\frac{\mathrm d}{\mathrm dt}\right|_{t=0}P_{\left(\delta(t),m\right)}P_{\left(c,\delta(t)\right)}v
\stackrel{\eqref{Eq: P is multiplicative}}{=}\left.\frac{\mathrm d}{\mathrm dt}\right|_{t=0}P_{(c,m)}v
=0\,.
\end{align}
Since $\left\{\mathscr U^{(c)}\mid c\in M\right\}$ is an open cover of $M$, $\nabla^P$ is curvature-free.
\end{proof}

The converse of the \cref{Prop: curvature-freeness criterion} does not hold for an arbitrary parallelism, as shown by the \cref{Example: parallelisms on R}. Indeed, the parallelism $\widetilde P$ induces a curvature-free affine connection, although it is not multiplicative.

\section{The Inverse Problem in Information Geometry in a two-point tensorial formalism}
\label{Sec: inverse problem in Information Geometry}

The stage is now set to formulate and study the inverse problem in information geometry within a two-point tensorial formalism. 

\begin{problem}
\label{Problem: inverse problem IG}
Let $(M,g,\nabla)$ be a Lauritzen manifold. The \underline{Inverse Problem in Information Geometry} consists in finding a contrast bi-form $\varpi\in \Omega^{1,1}_{\Delta_M}(M\mid M)$ such that
\begin{equation}
\label{Eq: inverse problem}
\begin{cases}
    g^\varpi = g\,,
    \\
    \nabla^\varpi = \nabla\,.
\end{cases}
\end{equation}
When this problem has a solution $\varpi$, we say that $\varpi$ \underline{integrates} $(M,g,\nabla)$, or that $\varpi$ is an \underline{integrating bi-form} for $(M,g,\nabla)$.
\end{problem}

The section is organized as follows. In \cref{Sec: existence}, by tailoring Matumoto's argument  to the bi-tensorial setting, we first prove that every Lauritzen manifold admits an integrating bi-form. Notice that such a result provides the existence of a solution to the inverse problem, but the construction proceeds through local coordinates, cut-off functions, and partitions of unity, and therefore does not provide an intrinsic procedure for constructing a solution directly from the given geometric data. Moreover, the  bi-form coming out from such a procedure need not recover the distinguished statistical potentials associated to paradigmatic models. In \cref{Sec: solution via parallelisms}, we address the first limitation by constructing integrating bi-forms explicitly from the given affine connection through parallelisms obtained by parallel transport. Finally, in \cref{Sec: inverse problem SMAT} and \cref{Sec: inverse problem statistical manifolds}, we use \cref{Prop: potentials solution bi-form} to pass from these integrating bi-forms to pre-contrast and contrast functions, and recover, for cases which are known in the existing literature, the usual potential functions. 
%Assume that \cref{Problem: inverse problem IG} is solvable for Lauritzen manifolds. Then it also yields solutions to the inverse problems for SMATs and statistical manifolds with respect to pre-contrast and contrast functions, respectively. Indeed, let $(M,g,\nabla)$ be a SMAT. Since every SMAT is a Lauritzen manifold, there is contrast bi-form $\varpi$ integrating $(M,g,\nabla)$. Since $\nabla$ is torsion-free, \cref{Thm: integration of a contrast bi-form} implies that $\varpi$ is statistically equivalent to a left-exact contrast bi-form of the form $\mathrm d^LS$. Therefore, the pre-contrast function corresponding to $S$ induces the same metric and the same affine connection as $\varpi$, and hence generates the prescribed SMAT. Similarly, let $(M,g,\nabla)$ be a statistical manifold. Since it is a Lauritzen manifold, there is an integrating contrast bi-form $\varpi$. Since both conjugate connections are torsion-free, \cref{Thm: integration of a contrast bi-form} implies that $\varpi$ is statistically equivalent to a bi-exact contrast bi-form of the type $\mathrm d^L\mathrm d^R D$. Therefore, the contrast function $D$ induces the same metric and the same pair of conjugate connections as $\varpi$, and hence generates the prescribed statistical manifold.

\subsection{An existence result}\label{Sec: existence}

We first show that \cref{Problem: inverse problem IG} is always solvable. The proof adapts Matumoto's construction for contrast functions on statistical manifolds \cite{Matumoto-1993} to the bi-form setting.
\begin{proposition}\label{Prop: matumoto}
    Every Lauritzen manifold can be integrated by a globally defined contrast bi-form.
\end{proposition}
\begin{proof}
Let $(M,g,\nabla)$ be a Lauritzen manifold. We first solve the inverse problem on a domain chart on $M$. Let $(U,q)$ be a chart of $M$, and consider the restricted Lauritzen manifold $(U,g_{\restriction U},\nabla_{\restriction U})$. One explicitly computes  that an integrating bi-form is
\begin{equation}\label{Eq: local contrast bi-form}
\varpi^{(U)}=\left(\pi_R^\ast g_{ij}+\pi_R^\ast\Gamma_{kij}\left(x^k-y^k\right)\right)\mathrm dq^i\boxtimes\mathrm dq^j\,,
\end{equation}
where $x^i=q^i\circ\pi_L$, $y^i=q^i\circ\pi_R$, and
\begin{align}
&g_{ij}=g(\partial_{q^i},\partial_{q^j})\,, \\ 
&\Gamma_{kij}=\Gamma^\ell_{ki}\,g_{\ell j}=g\left(\nabla_{\partial_{q^k}}\partial_{q^i},\partial_{q^j}\right)\,.
\end{align}
Consider now a locally finite open cover $\mathcal A=\{U_\alpha\}_{\alpha\in A}$ of $M$ given by the domains of local  charts of a smooth atlas, and denote by $\varpi^{(U_\alpha)}$ be the corresponding local bi-forms. Let $\{\varepsilon_\alpha\}_{\alpha\in A}$ be a partition of unity subordinate to $\mathcal A$, with $\operatorname{supp}(\varepsilon_\alpha)\subset U_\alpha\,$.
For each $\alpha$, choose a smooth cut-off function $\eta_\alpha$ such that $\operatorname{supp}(\eta_\alpha)\subset U_\alpha\times U_\alpha$ 
and $\eta_\alpha=1$ on an open neighbourhood $V_\alpha$ around 
$\Delta_{\operatorname{supp}(\varepsilon_\alpha)}=\left\{(c,c)\mid c\in\operatorname{supp}(\varepsilon_\alpha)\right\}$.
Define the bi-form $\varpi^{(\alpha)}$ on $M\times M$ by
\begin{equation}
\varpi^{(\alpha)}(m,n)=
\begin{cases}
\eta_\alpha(m,n)\,\varepsilon_\alpha(n)\,\varpi^{(U_\alpha)}(m,n)\,, & (m,n)\in U_\alpha\times U_\alpha\,,\\[2mm]
0\,, & (m,n)\notin U_\alpha\times U_\alpha\,.
\end{cases}
\end{equation}
Since $\operatorname{supp}(\eta_\alpha)\subset U_\alpha\times U_\alpha$, this zero extension is smooth. We then set
\begin{equation}
\varpi=\sum_{\alpha\in A}\varpi^{(\alpha)}\,.
\end{equation}
The sum is locally finite, hence $\varpi$ is a smooth $(1,1)$-bi-form on $M\times M$.
We claim that $\varpi$ integrates $(M,g,\nabla)$. 

In order to prove that, let $c\in M$. If $\varepsilon_\alpha(c)\ne0$, then $(c,c)\in V_\alpha$, so that $\eta_\alpha(c,c)=1$. Therefore, it is 
\begin{equation}
\varpi^{(\alpha)}_{\restriction\Delta_M}(c)=\varepsilon_\alpha(c)\,\varpi^{(U_\alpha)}_{\restriction\Delta_{U_\alpha}}(c)=\varepsilon_\alpha(c)\,g(c)\,.
\end{equation}
Using the local finiteness of the family and the identity $\sum_{\alpha\in A}\varepsilon_\alpha=1$, we obtain
\begin{equation}
\varpi_{\restriction\Delta_M}(c)=\sum_{\alpha\in A}\varepsilon_\alpha(c)\,g(c)=g(c)\,.
\end{equation}
Since $c$ is arbitrary, we can write
\begin{equation}
\varpi_{\restriction\Delta_M}=g\,.
\end{equation}
It remains to prove that $\nabla^\varpi=\nabla$. Let $X,Y,Z\in\mathfrak X(M)$ and $c\in M$. By the definition of the affine connection induced by $\varpi$, one has
\begin{equation}
g\left(\nabla^\varpi_ZX,Y\right)(c)=\left(\mathcal L_{Z^L}\left(\varpi(X\mid Y)\right)\right)(c,c)\,:
\end{equation}
such expression can, by local finiteness, be written as 
\begin{equation}
\left(\mathcal L_{Z^L}\left(\varpi(X\mid Y)\right)\right)(c,c)=\sum_{\alpha\in A}\left(\mathcal L_{Z^L}\left(\varpi^{(\alpha)}(X\mid Y)\right)\right)(c,c)\,.
\end{equation}
If $\varepsilon_\alpha(c)\ne0$, then $\eta_\alpha=1$ on a neighbourhood around  $(c,c)$, while $\varepsilon_\alpha\circ\pi_R$ is constant along $Z^L$. We can then write 
\begin{equation}
\left(\mathcal L_{Z^L}\left(\varpi^{(\alpha)}(X\mid Y)\right)\right)(c,c)=\varepsilon_\alpha(c)\left(\mathcal L_{Z^L}\left(\varpi^{(U_\alpha)}(X\mid Y)\right)\right)(c,c)\,.
\end{equation}
Since $\varpi^{(U_\alpha)}$ integrates $(U_\alpha,g_{\restriction U_\alpha},\nabla_{\restriction U_\alpha})$, it is 
\begin{equation}
\left(\mathcal L_{Z^L}\left(\varpi^{(U_\alpha)}(X\mid Y)\right)\right)(c,c)=g\left(\nabla_ZX,Y\right)(c)\,,
\end{equation}
and therefore,
\begin{equation}
\left(\mathcal L_{Z^L}\left(\varpi(X\mid Y)\right)\right)(c,c)=\sum_{\alpha\in A}\varepsilon_\alpha(c)\,g\left(\nabla_ZX,Y\right)(c)=g\left(\nabla_ZX,Y\right)(c)\,.
\end{equation}
It follows that
\begin{equation}
g\left(\nabla^\varpi_ZX,Y\right)(c)=g\left(\nabla_ZX,Y\right)(c)\,
\end{equation}
from which, since $g$ is non-degenerate, we conclude that $\nabla^\varpi=\nabla$.
\end{proof}

Although the construction above gives a global existence result for the inverse problem in Information Geometry, it is not constructive in general. It yields an explicit formula for an integrating bi-form only when the manifold is covered by a single global chart. Even in this case, however, the resulting solution need not coincide with the potential expected in paradigmatic examples, as the following example shows.

\begin{example}
 Let $M=\{\lambda\in \mathbb R\mid \lambda>0\}$ be the manifold of parameters of the univariate exponential distribution. Following \cite{Calin-Udriste-2014}, in terms of the Cartesian coordinate $\lambda$ on $\mathbb R$, consider the \emph{Fisher-Rao metric} 
 \begin{equation}
     g=\frac{1}{\lambda^2}\, \mathrm d\lambda\otimes \mathrm d\lambda\,,
 \end{equation}
 and the \emph{mixture affine connection}
 \begin{equation}
     \nabla_{\partial_\lambda}\partial_\lambda=-\frac{2}{\lambda}\partial_\lambda\,.
 \end{equation}

 Let us specialize the contrast bi-form \eqref{Eq: local contrast bi-form} to the Lauritzen manifold $(M,g,\nabla)$. Since $g_{\lambda\lambda}=\frac{1}{\lambda^2}$ and $\Gamma_{\lambda\lambda\lambda}=-\frac{2}{\lambda^3}$, we obtain
 \begin{equation}
     \varpi=\left(\frac{1}{y^2}-\frac{2}{y^3}\left(x-y\right)\right)\mathrm d\lambda\boxtimes \mathrm d\lambda=\left(\frac{3}{y^2}-\frac{2x}{y^3}\right)
    \mathrm d\lambda\boxtimes \mathrm d\lambda\,,
 \end{equation}
 where $(x,y)$ are the Cartesian coordinates of $\mathbb R^2$. Since the parameter manifold is one-dimensional, $\varpi$ is bi-exact. The most general potential is
\begin{equation}
    D_{A,B}(x,y)=-\frac{3x}{y}+\frac{x^2}{2y^2}+A(x)+B(y)\,,
\end{equation}
with $A$ and $B$ arbitrary smooth functions of one variable.

On the other hand, the Kullback-Leibler relative entropy of the exponential family is
\begin{equation}
    D_{\operatorname{KL}}(x,y)=\log\frac{x}{y}+ \frac{y}{x}-1\,.
\end{equation}
According to our sign convention, the expected potential is therefore $-D_{\operatorname{KL}}$. However,
\begin{equation}
    \mathrm d^L\mathrm d^R\left(-D_{\operatorname{KL}}\right)
    =\frac{1}{x^2}\,\mathrm d\lambda\boxtimes\mathrm d\lambda\,,
\end{equation}
which is not equal to $\varpi$. Hence no potential $D_{A,B}$ of $\varpi$ can coincide with $-D_{\operatorname{KL}}$. Therefore, although $\varpi$ integrates
$(M,g,\nabla)$, it does not recover the expected canonical potential for this model.
\end{example}

%The previous discussion shows that the inverse problem admits solutions, but \cref{Prop: matumoto} neither provides a constructive procedure for selecting one nor ensures compatibility with the distinguished potentials arising in paradigmatic examples. 
Our aim is now to describe an approach based on parallelisms that yields explicit integrating bi-forms and, through suitable choices of the auxiliary data, recovers the established constructions in the principal model cases.

\subsection{Solution via parallelisms}\label{Sec: solution via parallelisms}

The solution proposed here is based on a suitable correspondence between contrast bi-forms and parallelisms. We begin by noticing that any  contrast bi-form $\varpi$ determines a unique parallelism\footnote{The normalisation condition, i.e. the condition $P_{(m,m)}=\mathrm{id}_{\mathbf T_mM}$ directly comes since the metric $g^\varpi$ is defined to be the restriction to the diagonal $\Delta_M$ of the contrast bi-form $\varpi$; the invertibility of the maps $P_{(m,n)}$ defined in \eqref{28.07.4} (for $(m,n)$ in a suitable open subset around the diagonal $\Delta_M$) comes from the non-degeneracy of the diagonal restriction of the contrast bi-form $\varpi$ and the lower semi-continuity of the rank function.} $P^\varpi$, via 
\begin{equation}
\label{28.07.4}
\varpi(X\mid Y)(m,n)=g^\varpi_n\left(P^\varpi_{(m,n)}X_m,Y_n\right)\,,
\end{equation}
while  a pseudo-Riemannian metric $g$ together with a parallelism $P$ determines the contrast bi-form
\begin{equation}\label{Eq: canonical form}
\varpi^{g,P}(X\mid Y)(m,n)=g_n\left(P_{(m,n)}X_m,Y_n\right)\,.
\end{equation}
%The normalisation condition for the maps $P_{(m,m)}$ directly comes since the metric $g^\varpi$ is defined to be the restriction to the diagonal $\Delta_M$ of the contrast bi-form $\varpi$; the invertibility of the maps $P_{(m,n)}$ defined in \eqref{28.07.4} (for $(m,n)$ in a suitable open subset around the diagonal $\Delta_M$) comes from the non-degeneracy of the diagonal restriction of the contrast bi-form $\varpi$ and the lower semicontinuity of the rank function. 
\begin{comment}
\begin{remark}
If $(U\subseteq M,q)$ gives a (local) chart for $M$, with $\pi_L^*q=x$ and $\pi_R^*q=y$ on $U\times U$, one can write a contrast bi-form as $\varpi=\omega_{ij}(x,y)\dd x^i\otimes\dd y^j$, the corresponding metric tensor on $U$ as $g^\varpi=g_{ij}(x)\dd x^i\otimes \dd x^j\,=\,\omega_{ij}(x, y=x)\dd x^i\otimes \dd y_{|y=x}^j$, so that  the above expressions can be written as (with $g^{jk}g_{ki}=\delta^j_i$)
\begin{equation}
P_{(x,y)}\,=\,\omega_{ij}(x,y)\,g^{jk}(y)\,\dd x^i\otimes \red{\partial_{y^k}}\,,
\end{equation}
%This line shows that the  two constructions are inverse to each other.
\end{remark}
\end{comment}

\begin{remark}
    Let us write $P^\varpi$ in the square chart $(U\times U,x,y)$ associated to a coordinate chart $(U,q)$ of $M$. If
    \begin{equation}
        \varpi=\varpi_{ij}\,\mathrm dq^i\boxtimes\mathrm dq^j=\varpi_{ij}\,\mathrm dx^i\otimes\mathrm dy^j
    \end{equation}
    is the local expression of a contrast bi-form $\varpi$, and
    \begin{equation}
        g=g_{ij}\,\mathrm dq^i\otimes\mathrm dq^j\,,\qquad g_{ij}=\iota^\ast\varpi_{ij}\,,
    \end{equation}
    is the local expression of the induced metric, then
    \begin{equation}
        P^\varpi=\left(\pi_R^\ast g^{jk}\right)\varpi_{ij}\,\mathrm dx^i\otimes\partial_{y^k}\,,
    \end{equation}
    where $(g^{jk})$ is the inverse of $(g_{ij})$, namely $g^{jk}g_{ki}=\delta^j_i$. Conversely, if
    \begin{equation}
        P=P^k_i\,\mathrm dx^i\otimes\partial_{y^k}
    \end{equation}
    is the local expression of a parallelism $P$, then
    \begin{equation}
        \varpi^{g,P}=\left(\pi_R^\ast g_{kj}\right)\,P^k_i\,\mathrm dq^i\boxtimes\mathrm dq^j=\left(\pi_R^\ast g_{kj}\right)\,P^k_i\,\mathrm dx^i\otimes\mathrm dy^j\,.
    \end{equation}
\end{remark}

One notices that the contrast bi-form $\varpi^{g,P}$ induces the Lauritzen manifold $(M,g,\nabla^P)$, where $\nabla^P$ is the affine connection associated to $P$ via  \eqref{Eq: connection from P}. Indeed, the normalization condition \ref{Item: normalization} implies
\begin{equation}
\left(\varpi^{g,P}\right)_{\restriction\Delta_M}=g\,.
\end{equation}
On the other hand, for any $X,Y,Z\in\mathfrak X(M)$ and $m\in M$, the affine connection induced by $\varpi^{g,P}$ satisfies
\begin{equation}
g\left(\nabla^{\varpi^{g,P}}_ZX,Y\right)(m)=\left(\mathcal L_{Z^L}\left(\varpi^{g,P}(X\mid Y)\right)\right)(m,m)\,.
\end{equation}
The kinematic definition of the Lie derivative gives
\begin{equation}
\left(\mathcal L_{Z^L}\left(\varpi^{g,P}(X\mid Y)\right)\right)(m,m)=\left.\frac{\mathrm d}{\mathrm dt}\right|_{t=0}\varpi^{g,P}(X\mid Y)\left(\Phi^Z_t(m),m\right)\,,
\end{equation}
where $\Phi^Z$ is the flow of $Z$. By the definition of $\varpi^{g,P}$, the right-hand side reads
\begin{equation}
\left.\frac{\mathrm d}{\mathrm dt}\right|_{t=0}g_m\left(P_{\left(\Phi^Z_t(m),m\right)}X_{\Phi^Z_t(m)},Y_m\right)=g_m\left(\left.\frac{\mathrm d}{\mathrm dt}\right|_{t=0}P_{\left(\Phi^Z_t(m),m\right)}X_{\Phi^Z_t(m)},Y_m\right)\,.
\end{equation}
Since $t\mapsto\Phi^Z_t(m)$ is a smooth curve satisfying $\Phi^Z_0(m)=m$ and
\begin{equation}
\left.\frac{\mathrm d}{\mathrm dt}\right|_{t=0}\Phi^Z_t(m)=Z_m\,,
\end{equation}
the definition of $\nabla^P$ (cf. \eqref{Eq: connection from P}) yields
\begin{equation}
\left.\frac{\mathrm d}{\mathrm dt}\right|_{t=0}P_{\left(\Phi^Z_t(m),m\right)}X_{\Phi^Z_t(m)}=\left(\nabla^P_ZX\right)_m\,.
\end{equation}
Therefore,
\begin{equation}
g\left(\nabla^{\varpi^{g,P}}_ZX,Y\right)(m)=g\left(\nabla^P_ZX,Y\right)(m)\,.
\end{equation}
Since $g$ is non-degenerate, we conclude that
\begin{equation}
\nabla^{\varpi^{g,P}}=\nabla^P\,.
\end{equation}

The above lines show that, in order to solve the inverse problem for a Lauritzen manifold $(M,g,\nabla)$, it is sufficient to encode the given affine connection $\nabla$ in a parallelism. Recall from the discussion in \cref{Section: parallelisms} that this can be achieved by choosing a smooth map
\begin{equation}
\Gamma_0\colon\mathscr U_0\times[0,1]\longrightarrow M
\end{equation}
that satisfies
\begin{equation}
\Gamma_0(m,n,0)=m\,,\qquad \Gamma_0(m,n,1)=n\,,\qquad \Gamma_0(m,m,t)=m\,,
\end{equation}
and defining $P^{\nabla,\Gamma_0}$ to be the $\nabla$-parallel transport along the curves $\Gamma_0(m,n,\cdot)$. Such a map can be obtained by choosing an auxiliary affine connection $\nabla_0$, and a representative of the corresponding germ of $\nabla_0$-geodesic interpolations. We denote by  $P^{\nabla,\nabla_0}$ the resulting parallelism.

\begin{definition}\label{Def: parallel geodesic propagator}
Let $M$ be a smooth manifold endowed with affine connections $\nabla$, and consider another affine connection $\nabla_0$ on $M$. The $\nabla$-\underline{parallel transport along} $\nabla_0$-\underline{geodesics} on a strongly $\nabla_0$-convex neighbourhood $\mathscr U_0$ of $\Delta_M$ is the parallelism $P^{\nabla,\nabla_0}$ defined, for every $(m,n)\in\mathscr U_0$, by
\begin{equation}
P^{\nabla,\nabla_0}_{(m,n)}=P_{1\leftarrow0}^{\nabla,\gamma^{(0)}_{n\leftarrow m}}\colon\T_mM\longrightarrow\T_nM\,,
\end{equation}
where $\gamma^{(0)}_{n\leftarrow m}$ is the unique $\nabla_0$-geodesic segment joining $m$ to $n$ selected by the corresponding $\nabla_0$-geodesic interpolation.
\end{definition}

Notice that, since the germ of the $\nabla_0$-geodesic interpolation along $\Delta_M\times[0,1]$ depends only on $\nabla_0$, the germ of $P^{\nabla,\nabla_0}$ along $\Delta_M$ is determined by the pair $(\nabla,\nabla_0)$.

Pairing the parallelism $P^{\nabla,\nabla_0}$ to the metric $g$ as in \eqref{Eq: canonical form} yields a solution to the inverse problem.

\begin{proposition}\label{Thm: inverse problem}
Let $(M,g,\nabla)$ be a Lauritzen manifold, let $\nabla_0$ be an affine connection on $M$, and let $P^{\nabla,\nabla_0}$ be the $\nabla$-parallel transport along $\nabla_0$-geodesics defined on a strongly $\nabla_0$-convex neighbourhood $\mathscr U_0$ of $\Delta_M$. Then the assignment
\begin{equation}\label{Eq: solution bi-form}
\varpi^{\nabla,\nabla_0}=\varpi^{g,P^{\nabla,\nabla_0}}
\end{equation}
defines an integrating bi-form for $(M,g,\nabla)$ on $\mathscr U_0$. 
\end{proposition}
Notice that the germ of $\varpi^{\nabla,\nabla_0}$ along $\Delta_M$ depends only on $g,\nabla$ and $\nabla_0$.

In the rest of this section, we study further properties of these solution bi-forms.

\begin{proposition}\label{Prop: compatibility test}
Let $(M,g,\nabla)$ be a Lauritzen manifold, and let $\nabla_0$ be an affine connection on $M$. It is 
\begin{equation}\label{Eq: duality bi-forms}
   \left(\varpi^{\nabla,\nabla_0}\right)^\dag=\varpi^{\nabla^\dag,\nabla_0}\,.
\end{equation}
In particular, $\nabla$ is $g$-compatible if and only if $\varpi^{\nabla,\nabla_0}$ is self-dual.
\end{proposition}

\begin{proof}
The claim is equivalent to
\begin{equation}
    g_n\left(P^{\nabla,\gamma_{n\leftarrow m}^{(0)}}_{1\leftarrow 0}X_m,Y_n\right)
    =
    g_m\left(X_m,P^{\nabla^\dag,\gamma_{m\leftarrow n}^{(0)}}_{1\leftarrow 0}Y_n\right)\,,
\end{equation}
for any $(m,n)\in \mathscr U_0$ and $X,Y\in \mathfrak X(M)$. This identity follows from the parallel-transport formulation of the $g$-compatibility relation between $\nabla$ and $\nabla^\dag$ \cite{NomizuSimon1992}, that is
\begin{equation}
g_{\gamma(1)}\left(P^{\nabla,\gamma}_{1\leftarrow0}v,w\right)
=
g_{\gamma(0)}\left(v,P^{\nabla^\dag,\gamma}_{0\leftarrow1}w\right)\,,
\end{equation}
applied to $\gamma=\gamma_{n\leftarrow m}^{(0)}$, together with
\begin{equation}
P^{\nabla^\dag,\gamma_{n\leftarrow m}^{(0)}}_{0\leftarrow1}
= 
P^{\nabla^\dag,\gamma_{m\leftarrow n}^{(0)}}_{1\leftarrow0}\,.
\end{equation}
The claim follows from the behaviour of the induced geometric data under duality (see \cite{CMPSZ2026}). Indeed, a bi-form $\varpi$ is a contrast bi-form if and only if $\varpi^\dag$ is a contrast bi-form, and
\begin{equation}\label{Eq: duality relations}
g^{\varpi^\dag}=g^\varpi\,,\qquad 
\nabla^{\varpi^\dag}=\left(\nabla^\varpi\right)^\dag\,:
\end{equation}
this means that $\varpi^{\nabla,\nabla_0}$ is self-dual if and only if the affine connection it induces is $g$-compatible, that is, if and only if $\nabla=\nabla^\dag$.
\end{proof}

When the connection of a given Lauritzen manifold $(M,g,\nabla)$ has vanishing curvature, the construction recovers the solution bi-form written in \cite{CMPSZ2026}. Recall that, if $\nabla$ is curvature-free, then $\nabla$-parallel transport along a curve depends only on its homotopy class relative to the endpoints. Now let $(m,n)$ belong to a strongly $\nabla_0$-convex neighbourhood $\mathscr U_0$ of the diagonal. By construction, there exists a $\nabla_0$-convex open $C_0$ containing $m$ and $n$, and the unique $\nabla_0$-geodesic segment $\gamma_{n\leftarrow m}^{(0)}$ is contained in $C_0$. Since $C_0$ is contractible, any two curves in $C_0$ joining $m$ to $n$ are homotopic relative to the endpoints. %Hence the $\nabla$-parallel transport along $\gamma_{n\leftarrow m}^{(0)}$ depends only on $\nabla$ and on the pair $(m,n)$. %, and not on the auxiliary connection $\nabla_0$.
Accordingly, the germ of $P^{\nabla,\nabla_0}$ along $\Delta_M$ depends only on $\nabla$, and we write $P^{\nabla,\nabla_0}=P^\nabla$ and $\varpi^{\nabla,\nabla_0}=\varpi^\nabla$.

\begin{proposition}\label{Prop: dually curvature-free}
Let $(M,g,\nabla)$ be a dually curvature-free Lauritzen manifold, and let $\nabla_0$ be an affine connection on $M$. Let $\varpi^\nabla$ be the solution bi-form associated to the parallelism $P^\nabla$ defined on a strongly $\nabla_0$-convex neighbourhood $\mathscr U_0$ of $\Delta_M$. Then, for every $X,X_0,X_1,Y,Y_0,Y_1\in\mathfrak X(M)$ and every $(m,n)\in \mathscr U_0$, one has
\begin{align}
\mathrm d^L\varpi^\nabla(X_0,X_1\mid Y)(m,n)&=g_n\left(P^\nabla_{(m,n)}\left(\Tor^\nabla(X_0,X_1)\right)_m,Y_n\right)\,,\label{Eq: exact dLvarpi}\\
\mathrm d^R\varpi^\nabla(X\mid Y_0,Y_1)(m,n)&=g_m\left(X_m,P^{\nabla^\dag}_{(n,m)}\left(\Tor^{\nabla^\dag}(Y_0,Y_1)\right)_n\right)\,.\label{Eq: exact dRvarpi}
\end{align}
\end{proposition}

\begin{proof}
Let $\mathcal C_0$ be a strongly $\nabla_0$-convex cover whose associated strongly convex neighbourhood is $\mathscr U_0$, and fix $(m,n)\in\mathscr U_0$. Choose $C_0\in\mathcal C_0$ containing $m$ and $n$. Since $\nabla$ is curvature-free and $C_0$ is contractible, for every $c\in C_0$ one has
\begin{equation}
P^\nabla_{(m,n)}\circ P^\nabla_{(c,m)}=P^\nabla_{(c,n)}\,.
\end{equation}
Therefore,
\begin{equation}
\left(\mathcal L_{Z^L}\left(\varpi^\nabla(X\mid Y)\right)\right)(m,n)=g_n\left(P^\nabla_{(m,n)}\left(\nabla_ZX\right)_m,Y_n\right)\,.
\end{equation}
By the definition of the left exterior derivative,
\begin{align}
\mathrm d^L\varpi^\nabla(X_0,X_1\mid Y)(m,n)&=g_n\left(P^\nabla_{(m,n)}\left(\nabla_{X_0}X_1-\nabla_{X_1}X_0-[X_0,X_1]\right)_m,Y_n\right)\notag\\
&=g_n\left(P^\nabla_{(m,n)}\left(\Tor^\nabla(X_0,X_1)\right)_m,Y_n\right)\,,
\end{align}
which proves \eqref{Eq: exact dLvarpi}. 
The identity \eqref{Eq: exact dRvarpi} follows by the same argument applied in the case of  the curvature-free conjugate connection $\nabla^\dag$ and by \eqref{Eq: dual bi-form}.
\end{proof}

Once an integrating bi-form has been constructed, the next step is to understand whether it admits further integration to a pre-contrast tensor or a contrast function. As shown by \cref{Prop: torsion bi-form}, this depends precisely on the torsion properties of the affine connections involved. In the next lines we make these constructions explicit for the canonical solution bi-forms \eqref{Eq: solution bi-form}, thereby recovering standard formulas proposed in the literature for pre-contrast and contrast functions.
\subsubsection{The Inverse Problem for SMATs}\label{Sec: inverse problem SMAT}
Let $(M,g,\nabla)$ be a SMAT, so that $\nabla$ is torsion-free, and let $\varpi^{\nabla,\nabla_0}$ be the solution bi-form associated to an auxiliary affine connection $\nabla_0$ on $M$. By \cref{Thm: integration of a contrast bi-form}, $\varpi^{\nabla,\nabla_0}$ is statistically equivalent to a left-exact bi-form of the form
\begin{equation}
\mathrm d^L S\,, \qquad S=J_L\varpi^{\nabla,\nabla_0}\,,
\end{equation}
where $J_L$ is  a statistical left homotopy operator. To obtain an explicit expression, we choose a second auxiliary affine connection $\nabla_1$ and use the corresponding operator $J_L^{\nabla_1}$ (cf. \eqref{Eq: explicit statistical left homotopy operator}).

Specializing the general formula \eqref{Eq: explicit statistical left homotopy operator} to $(1,1)$-bi-forms, let $\mathscr U_1$ be a strongly $\nabla_1$-convex neighbourhood of $\Delta_M$, and let $\Gamma_1$ be the corresponding $\nabla_1$-geodesic interpolation. Then, for every $(1,1)$-bi-form $\varpi$, one has
\begin{equation}\label{Eq: JL example refined}
\left(J_L^{\nabla_1}\varpi\right)(\mid Y)(m,n)=\int_0^1\varpi\left(\dot{\gamma}^{(1)}_{m\leftarrow n}(t)\mid Y\right)\left(\gamma^{(1)}_{m\leftarrow n}(t),n\right)\,\mathrm dt\,,
\end{equation}
 where we have denoted 
$\gamma_{m\leftarrow n}^{(1)}(t)=\Gamma_1(n,m,t)$. The following result provides a family of local pre-contrast functions integrating the given SMAT.
\begin{proposition}\label{Thm: canonical pre-contrasts refined}
Let $(M,g,\nabla)$ be a SMAT, let $\nabla_0$ and $\nabla_1$ be affine connections on $M$, and let $\varpi^{\nabla,\nabla_0}$ be defined on a strongly $\nabla_0$-convex neighbourhood $\mathscr U_0$ of $\Delta_M$. There exists a strongly $\nabla_1$-convex neighbourhood $\mathscr U_{01}\subseteq\mathscr U_0$ of $\Delta_M$ such that $\varpi^{\nabla,\nabla_0}$ is statistically equivalent on $\mathscr U_{01}$ to the bi-form
\begin{equation}
\mathrm d^L S^{\nabla,\nabla_0,\nabla_1}\,,
\end{equation}
where
\begin{equation}
S^{\nabla,\nabla_0,\nabla_1}=J_L^{\nabla_1}\varpi^{\nabla,\nabla_0}\,.
\end{equation}
Explicitly, it is 
\begin{equation}\label{Eq: canonical pre-contrast general refined}
S^{\nabla,\nabla_0,\nabla_1}(\mid Y)(m,n)=\int_0^1g_n\left(P^{\nabla,\nabla_0}_{\left(\gamma^{(1)}_{m\leftarrow n}(t),n\right)}\dot{\gamma}^{(1)}_{m\leftarrow n}(t),Y_n\right)\,\mathrm dt\,.
\end{equation}
In particular, for $\nabla_0=\nabla_1=\nabla$, one obtains
\begin{equation}
S^\nabla=S^{\nabla,\nabla,\nabla}\,,
\end{equation}
with
\begin{equation}\label{Eq: canonical pre-contrast refined}
S^\nabla(\mid Y)(m,n)=g_n\left(\dot{\gamma}_{m\leftarrow n}(0),Y_n\right)\,.
\end{equation}
%where $\gamma_{m\leftarrow n}$ is the unique $\nabla$-geodesic segment from $n$ to $m$ determined by the strongly $\nabla$-convex neighbourhood.
\end{proposition}

\begin{proof}
Since $(M,g,\nabla)$ is a SMAT, the connection $\nabla$ is torsion-free. 
By \cref{Thm: integration of a contrast bi-form},  $\varpi^{\nabla,\nabla_0}$ is statistically equivalent to the bi-form
$\mathrm d^LS^{\nabla,\nabla_0,\nabla_1}$ 
in a strongly $\nabla_1$-convex neighbourhood $\mathscr U_{01}\subseteq \mathscr U_0$.
To obtain the explicit expression \eqref{Eq: canonical pre-contrast general refined}, let $Y\in\mathfrak X(M)$ and $(m,n)\in\mathscr U_{01}$. Applying \eqref{Eq: JL example refined} to $\varpi^{\nabla,\nabla_0}$ gives
\begin{equation}
S^{\nabla,\nabla_0,\nabla_1}(\mid Y)(m,n)=\int_0^1g_n\left(P^{\nabla,\nabla_0}_{\left(\gamma^{(1)}_{m\leftarrow n}(t),n\right)}\dot{\gamma}^{(1)}_{m\leftarrow n}(t),Y_n\right)\,\mathrm dt\,,
\end{equation}
which proves \eqref{Eq: canonical pre-contrast general refined}.

Suppose now that $\nabla_0=\nabla_1=\nabla$. Since $\gamma_{m\leftarrow n}=\gamma^{(1)}_{m\leftarrow n}$ is a $\nabla$-geodesic segment, its velocity field $\dot\gamma_{m\leftarrow n}(t)$ is $\nabla$-parallel, hence
\begin{equation}
 P^{\nabla,\nabla_0}_{\left(\gamma_{m\leftarrow n}(t),n\right)}\dot{\gamma}_{m\leftarrow n}(t)=\dot{\gamma}_{m\leftarrow n}(0)\,.
\end{equation}
Therefore,
\begin{equation}
S^\nabla(\mid Y)(m,n)=\int_0^1g_n\left(\dot{\gamma}_{m\leftarrow n}(0),Y_n\right)\,\mathrm dt=g_n\left(\dot{\gamma}_{m\leftarrow n}(0),Y_n\right)\,,
\end{equation}
which proves \eqref{Eq: canonical pre-contrast refined}.
\end{proof}

The previous  proposition shows that the expression for the pre-contrast function associated to \eqref{Eq: canonical pre-contrast refined} solves the inverse problem for SMATs, in agreement with the proposal of Henmi and Matsuzoe in \cite{H-M-2019}. 

\begin{example}
Assume that $(M,g,\nabla)$ is a partially flat SMAT, namely that $\nabla$ is flat. By \cref{Prop: dually curvature-free}, the solution bi-form $\varpi^\nabla$ is left-exact and there is a unique germ of a local $(0,1)$-bi-form $S$, vanishing along the diagonal, such that
\begin{equation}
    \varpi^\nabla=\mathrm d^L S\,.
\end{equation}
Therefore, $S$ coincides with \eqref{Eq: canonical pre-contrast refined}, and the associated pre-contrast function is
\begin{equation}
    \dot S^\nabla(m\mid n,w)
    =g_n\left(\dot\gamma_{m\leftarrow n}(0),w\right)\,,
\end{equation}
where $\gamma_{m\leftarrow n}$ is the $\nabla$-geodesic from $n$ to $m$ within a strongly $\nabla$-convex neighbourhood. This coincides with the canonical pre-contrast function of Henmi and Matsuzoe in the partially flat case \eqref{Eq: HM}, modulo our conventions (\cref{Remark: smat}, \cref{Remark: sign}).
\end{example}

\subsubsection{The inverse problem for statistical manifolds}\label{Sec: inverse problem statistical manifolds}
We now turn to the statistical-manifold case. Let $(M,g,\nabla)$ be a statistical manifold, so that both $\nabla$ and $\nabla^\dag$ are torsion-free, and let $\varpi^{\nabla,\nabla_0}$ be the solution bi-form associated to an auxiliary affine connection $\nabla_0$ on $M$. By \cref{Thm: integration of a contrast bi-form}, the bi-form $\varpi^{\nabla,\nabla_0}$ is statistically equivalent to a bi-exact bi-form which we write as 
\begin{equation}
    \mathrm d^L\mathrm d^RD\,, \qquad D=J_RJ_L\varpi^{\nabla,\nabla_0}\,,
\end{equation}
where $(J_L,J_R)$ is a pair of statistical homotopy operators for $(\mathrm d^L,\mathrm d^R)$. To obtain an explicit expression, we choose two more auxiliary affine connections $\nabla_1$, $\nabla_2$, and consider the corresponding operators $(J_L^{\nabla_1},J_R^{\nabla_2})$.

Specializing the general formula \eqref{Eq: explicit statistical right homotopy operator} to $(0,1)$-bi-forms, let $\mathscr U_2$ be a strongly $\nabla_2$-convex neighbourhood of $\Delta_M$, and let $\Gamma_2$ be the corresponding $\nabla_2$-geodesic interpolation. Then, for every $(0,1)$-bi-form $S$, one has
\begin{equation}\label{Eq: JR example refined}
\left(J_R^{\nabla_2}S\right)(m,n)=\int_0^1S\left(\mid\dot{\gamma}^{(2)}_{n\leftarrow m}(t)\right)\left(m,\gamma^{(2)}_{n\leftarrow m}(t)\right)\,\mathrm dt\,,
\end{equation}
where we denote 
$\gamma^{(2)}_{n\leftarrow m}(t)=\Gamma_2(m,n,t)\,$.

\begin{proposition}\label{Prop: potentials solution bi-form}
Let $(M,g,\nabla)$ be a statistical manifold, let $\varpi^{\nabla,\nabla_0}$ be the solution bi-form defined on a strongly $\nabla_0$-convex neighbourhood $\mathscr U_0$ of $\Delta_M$ in $M\times M$, and let $\nabla_1,\nabla_2$ be affine connections on $M$. There  exists an open neighbourhood $\mathscr U_{012}\subseteq \mathscr U_0$ of $\Delta_M$, strongly $\nabla_i$-convex for every $i\in\{0,1,2\}$, such that $\varpi^{\nabla,\nabla_0}$ is statistically equivalent to $\mathrm d^L\mathrm d^RD^{\nabla,\nabla_0,\nabla_1,\nabla_2},$
where
\begin{equation}
D^{\nabla,\nabla_0,\nabla_1,\nabla_2}=J_R^{\nabla_2}J_L^{\nabla_1}\varpi^{\nabla,\nabla_0}\,.
\end{equation}
In particular, 
\begin{enumerate}[(1)]
    \item if one selects $\nabla_0=\nabla_1=\nabla_2=\nabla$, then $D^\nabla=D^{\nabla,\nabla,\nabla,\nabla}$ is given by
 \begin{equation}\label{Eq: Ay-Amari}
D^\nabla(m,n)=-\int_0^1t\,g_{\gamma_{n\leftarrow m}(t)}\left(\dot\gamma_{n\leftarrow m}(t),\dot \gamma_{n\leftarrow m}(t)\right)
\,\mathrm dt\,,
\end{equation}
where %for any $(m,n)$ in a strongly $\nabla$-convex neighbourhood of $\Delta_M$ in $M\times M$, $\gamma_{n\leftarrow m}$ is the $\nabla$-geodesic segment joining $m$ with $n$, and $\|\cdot \|$ is the $g$-norm, and 
for any $(m,n)\in \mathscr U_{012}$, $\gamma_{n\leftarrow m}=\gamma_{n\leftarrow m}^{(i)}$;
\item if one selects $\nabla_0=\nabla_1=\nabla$ and $\nabla_2=\nabla^\dag$, then $D^{\nabla,\nabla^\dag}=D^{\nabla,\nabla,\nabla,\nabla^\dag}$ is given by
\begin{equation}\label{Eq: Hook law}
    D^{\nabla,\nabla^\dag}(m,n)=\int_0^1 g_{\gamma^\dag_{n\leftarrow m}(t)}\left(\frac{\partial}{\partial s}\bigg|_{s=0}H(t,s),\dot \gamma^\dag_{n\leftarrow m}(t)\right) \mathrm dt,
\end{equation}
%\blue{UNA PAROLA SUL FATTO CHE ABBIAMO IL SEGNO OPPOSTO RISPETTO A QUELLA DI HENMI-KOBAYASHI DELL'INTRODUZIONE LA DIREI} 
where for every $(m,n)\in \mathscr U_{012}$, $\gamma_{n\leftarrow m}=\gamma_{n\leftarrow m}^{(0)}=\gamma_{n\leftarrow m}^{(1)}$, $\gamma_{n\leftarrow m}^\dag=\gamma_{n\leftarrow m}^{(2)}$, $\Gamma=\Gamma_1$, and %for every $(m,n)$ in an open neighbourhood $\mathscr U$ of $\Delta_M$ in $M\times M$ that is both $\nabla$-convex and $\nabla^\dag$-convex, $\gamma^\dag_{n\leftarrow m}$ denotes the unique $\nabla^\dag$-geodesic segment joining $m$ with $n$, and:
\begin{equation}
    H(t,s)=\Gamma\left(\gamma^\dag_{n\leftarrow m}(t),m,s\right)
\end{equation}
%\blue{DUE RIGHE SOPRA COMPARE COME $H(s,t)$, UNIFORMEREI PERCHE' COSI' E' UN PO' BRUTTO. INOLTRE $\Gamma$ PRENDE IN PASTO DUE PUNTI E UN PARAMETRO, QUINDI IL SECONDDO ARGOMENTO NON DOVREBBE AVERE IL DOT.} $\Gamma$ denoting the $\nabla$-geodesic interpolation map defined on $\mathscr U$.
\end{enumerate}
\end{proposition}

\begin{proof}
The first part follows from \cref{Thm: integration of a contrast bi-form}. Note that
\begin{equation}
    D^{\nabla,\nabla_0,\nabla_1,\nabla_2}=J_R^{\nabla_2}S^{\nabla,\nabla_0,\nabla_1}\,,
\end{equation}
where $S^{\nabla,\nabla_0,\nabla_1}$ is the local $(0,1)$-bi-form defined in \cref{Thm: canonical pre-contrasts refined}. For every $(m,n)\in\mathscr U_{012}$, using \eqref{Eq: JR example refined}, one obtains
\begin{equation}\label{Eq: general D expanded once}
\begin{aligned}
D^{\nabla,\nabla_0,\nabla_1,\nabla_2}(m,n)
&=\int_0^1\left(S^{\nabla,\nabla_0,\nabla_1}\right)_{(m,\gamma^{(2)}_{n\leftarrow m}(s))}\left(\mid \dot \gamma^{(2)}_{n\leftarrow m}(s)\right)\,\mathrm ds\\
&=\iint_{[0,1]^2}
g_{\gamma^{(2)}_{n\leftarrow m}(s)}
\Bigl(
P^{\nabla,\nabla_0}_{\left(H_{12}(t,s),\,\gamma^{(2)}_{n\leftarrow m}(s)\right)}
\frac{\partial}{\partial \tau}\bigg|_{\tau=t}H_{12}(\tau,s),
\dot\gamma^{(2)}_{n\leftarrow m}(s)
\Bigr)
\,\mathrm ds\,\mathrm dt\,,
\end{aligned}
\end{equation}
where
\begin{equation}
H_{12}(t,s)=\Gamma_1\left(\gamma^{(2)}_{n\leftarrow m}(s),m,t\right)
\end{equation}
and $\Gamma_1$ is the $\nabla_1$-geodesic interpolation.

Choosing $\nabla_0=\nabla_1=\nabla$ gives
\begin{align}
    D^{\nabla,\nabla_2}(m,n)%&=\iint_{[0,1]^2} g_{H_{12}(0,s)} \left(\frac{\partial}{\partial \tau}\bigg|_{\tau=0}H_{12}(\tau,s), \dot \gamma^{(2)}_{n\leftarrow m}(s)\right)\,\mathrm ds\,\mathrm dt\\
    &=\int_0^1 g_{\gamma^{(2)}_{n\leftarrow m}(s)}\left(\frac{\partial}{\partial \tau}\bigg|_{\tau=0}H_{12}(\tau,s),\dot \gamma_{n\leftarrow m}^{(2)}(s)\right)\,\mathrm ds \,. 
\end{align}
%\blue{L'APICE SULLE CURVE $\gamma$ FINO AD ORA L'HAI SEMPRE MESSO TRA PARENTESI. INOLTRE LA $g$ DOVREBBE ESSERE VALUTATA IN $\gamma(s)$ NON $\gamma(t)$.}
This gives in particular the result \eqref{Eq: Hook law} for $\nabla_2=\nabla^\dag$.

   Finally, let us assume that $\nabla_2=\nabla$. The claim follows by the affine reparametrization property of $\nabla$-geodesics, that is
    \begin{equation}
    \Gamma\left(\Gamma(a,b,t_1),\Gamma(a,b,t_2),t_0\right)= \Gamma\left(a,b,(1-t_0)\,t_1+t_0\,t_2\right)
\end{equation}
    for any $(a,b)\in \mathscr U$, and $t_0,t_1,t_2\in [0,1]$. Applying this to $a=m$, $b=n$, $t_0=\tau$, $t_1=s$, and $t_2=0$, one has 
    \begin{equation}
        H_{12}(\tau,s)=\gamma_{n\leftarrow m}\left((1-\tau)\,s\right)\,.
    \end{equation}
    By the chain rule, it is 
    \begin{equation}
        \frac{\partial}{\partial \tau}\bigg|_{\tau=0}H_{12}(\tau,s)=-s\, \dot \gamma_{n\leftarrow m}(s)\,,
    \end{equation}
    from which the formula \eqref{Eq: Ay-Amari} follows.
\end{proof}

The two previous choices recover, according to our sign convention, two important contrast functions defined for arbitrary statistical manifolds. The uniform choice 
$\nabla_0=\nabla_1=\nabla_2=\nabla$
yields the \emph{Ay-Amari canonical divergence} $D^\nabla$ \cite{A-A-2015}. By contrast, the choice
$\nabla_0=\nabla_1=\nabla$ with
$\nabla_2=\nabla^\dagger$
recovers the \emph{Henmi-Kobayashi contrast function} $D^{\nabla,\nabla^\dag}$ \cite{H-R-2000}. 

\begin{example}[Dually flat statistical manifold]
Assume that $(M,g,\nabla)$ is a dually flat statistical manifold. It follows that both $\nabla$ and $\nabla^\dag$ are flat,  the solution bi-form $\varpi^\nabla$ is bi-exact, and there is a unique germ of a local $(0,0)$-bi-form $D$, vanishing up to first order along the diagonal, such that
\begin{equation}
    \varpi^\nabla
    =
    \mathrm d^L\mathrm d^R D\,.
\end{equation}
This reads that $D$ coincides with the Ay-Amari canonical divergence \eqref{Eq: Ay-Amari}, up to the sign convention adopted here.
\end{example}

\begin{example}[Pseudo-Riemannian manifolds]
Any pseudo-Riemannian manifold $(M,g)$ induces a statistical manifold
$(M,g,\nabla^g)$, where $\nabla^g$ is the Levi-Civita connection of $g$. Since $\nabla^g$ is $g$-self-conjugate, the expressions \eqref{Eq: Ay-Amari} and \eqref{Eq: Hook law} coincide and reduce to the opposite of the Synge world function \cite{Synge-1960, Moretti-2021}. In the Riemannian case, according to our sign convention, this coincides with the restriction of the negative half squared Riemannian distance to a strongly $\nabla^g$-convex neighbourhood. %We also recall that in the Riemannian case, for every $c\in M$ there is $r>0$ such that the open Riemannian ball $B_r^g(c)$ is \emph{geodesically convex}, namely for any $m,n\in B_r^g(c)$ there is a unique minimizing $\nabla^g$-geodesic joining $m$ with $n$ within $B_r^g(c)$. Therefore $\{B_r^g(c)\mid c\in M\}$ is a strongly $\nabla^g$-convex cover of $M$.
\end{example}

\section{Reductive homogeneous metric-affine manifolds}
\label{Sec: examples}
It is time to describe the bi-form solution to the inverse problem in information geometry for a class of -- to us -- interesting cases. We consider homogeneous pseudo-Riemannian manifolds $M\cong G/G_o$ endowed with a one-parameter family of canonical affine connections $\nabla^{(\lambda)}$ associated to a reductive decomposition of the Lie algebra of $G$. As representative examples, we consider semisimple Lie groups \cite{Postnikov-2001}, including the special unitary group $\SU(d+1)$, and odd-dimensional spheres endowed with Berger metrics \cite{Berger1961, D-G-P-2016}. In the latter case, we also examine their relation with complex projective spaces through the Hopf fibration.

\subsection{The canonical affine connections on a reductive space}
The Nomizu correspondence provides an algebraic description of invariant pseudo-Riemannian metrics and invariant affine connections on reductive homogeneous spaces \cite{Nomizu-1954}. We begin by recalling it.

Let $M$ be a smooth manifold, and let  $\tau\colon G\times M\to M$ denote a smooth and transitive action on $M$ of a Lie group $G$, that we write as $\tau_z(m)=z\cdot m$ for $z\in\,G$ and $m\in M$. For an element  $o\in M$, let $G_o=\{z\in G\mid z\cdot o=o\}$ denote the corresponding isotropy subgroup, with  $\mathfrak g$ and $\mathfrak g_o$ the Lie algebras of $G$ and $G_o$, respectively. The homogeneous space $(M,G,\tau)$ is said to be \underline{reductive} if there exists a vector-space decomposition
\begin{equation}
    \mathfrak g=\mathfrak g_o\oplus\mathfrak m
\end{equation}
such that\footnote{Notice that we are considering the adjoint action of a Lie group $G$ on its Lie algebra $\mathfrak{g}$ to be defined in terms of the Lie group exponential map $\exp\,:\,\mathfrak g\,\to\,G$ by  $$\Ad_z(X)=\left.\frac{\dd}{\dd t}\right|_{t=0}z\,(\exp\,tX)\,z^{-1}$$ where $z\in G$ and $X\in\mathfrak{g}$.}  $\Ad_h\mathfrak m=\mathfrak m$ for $h\in G_o$. If $G_o$ is connected, such a condition turns out to be equivalent   to
\begin{equation}
    [\mathfrak g_o,\mathfrak m]\subseteq\mathfrak m\,.
\end{equation}
The action of $G$ on $M$ induces the so-called fundamental vector fields on $M$, which provide local generators for the module $\mathfrak{X}(M)$ of vector fields on $M$. For $X\in\mathfrak g$, we denote by $\tilde X$ the vector field on $M$ defined by
\begin{equation}
    \tilde X_m=\left.\frac{\mathrm d}{\mathrm dt}\right|_{t=0}(\exp tX)\cdot m\,, \qquad m\in M\,
\end{equation}
so to have a correspondence  $X\longmapsto \tilde X$ which is a  Lie algebra anti-homomorphism, namely 
\begin{equation}
    -\widetilde{[X,Y]}=[\tilde X,\tilde Y]\,, \qquad X,Y\in\mathfrak g\,.
\end{equation}
A pseudo-Riemannian metric $g$ on $M$ is called $G$-invariant if every smooth map  $\tau_z$ on $M$ is an isometry, namely
\begin{equation}
    \tau_z^\ast g=g\,, \qquad \forall\,z\in G\,.
\end{equation}
Under the action of $G$, fundamental vector fields transform as $(\tau_z)_*\tilde X=\widetilde{\mathrm{Ad}_z(X)}$, which can be written locally as 
\begin{equation}
    (\tau_z)'_m\tilde X_m=(\widetilde{\Ad_zX})_{z\cdot m}\,, \qquad z\in G\,,\quad X\in\mathfrak g\,
\end{equation}
so that the invariance of $g$ can be written as
\begin{equation}\label{Eq: invariance metric 2nd}
    g_{z\cdot o}\left(\tilde X_{z\cdot o},\tilde Y_{z\cdot o}\right)
    =
    g_o\left(\widetilde{(\Ad_{z^{-1}}X)}_o,\widetilde{(\Ad_{z^{-1}}Y)}_o\right)\,, \qquad X,Y\in\mathfrak g\,.
\end{equation}
This shows that  a $G$-invariant pseudo-Riemannian metric is determined by its value at $o$. Similarly, an affine connection $\nabla$ on $M$ is called $G$-invariant if
\begin{equation}\label{Eq: nabla invariance 2nd}
    \left(\nabla_{\tilde X}\tilde Y\right)_{z\cdot o}=(\tau_z)'_o\left(\nabla_{\widetilde{\Ad_{z^{-1}} X}}\widetilde{\Ad_{z^{-1}}Y}\right)_o\,, \qquad X,Y\in\mathfrak g\,.
\end{equation}
It follows that a $G$-invariant affine connection is determined by its action at $o$ on fundamental vector fields. The precise algebraic characterization of these invariant metrics and affine connections is given by the \emph{Nomizu correspondence}, whose results we describe as follows, while referring the reader to \cite{Nomizu-1954} for a proof.

%\begin{proposition}[Nomizu's correspondence]
Let $M$ be a reductive homogeneous space with respect to the action of a Lie group $G$, with reductive decomposition
\begin{equation}
    \mathfrak g=\mathfrak g_o\oplus\mathfrak m\,,
\end{equation}
for a given point $o\in M$. The following holds. 
\begin{enumerate}[($a$)]
\item  Pseudo-Riemannian metrics on $M$ which are $G$-invariant are in one-to-one correspondence with non-degenerate bilinear forms 
\begin{equation}
    \langle \cdot,\cdot\rangle \colon \mathfrak m\times \mathfrak m\to \mathbb R
\end{equation}
in $\mathfrak{m}$ which are $\Ad(G_o)$-equivariant, namely
\begin{equation}\label{Eq: Ad(Go)-invariance 1st}
    \langle \Ad_hX,\Ad_hY\rangle=\langle X,Y\rangle\,, \qquad h\in G_o\,,\qquad X,Y\in \mathfrak m\,.
\end{equation}
Equivalently, if $G_o$ is connected, this condition is
\begin{equation}\label{Eq: Ad(Go)-invariance 2nd}
    \langle[A,X]_{\mathfrak m},Y\rangle
    +
    \langle X,[A,Y]_{\mathfrak m}\rangle
    =
    0\, ,
    \qquad A\in\mathfrak g_o\, ,
    \qquad X,Y\in\mathfrak m\, .
\end{equation}
The invariant pseudo-Riemannian metric associated to $\langle\cdot,\cdot\rangle$ is characterized by
\begin{equation}
    g_o(\tilde X,\tilde Y)=\langle X_{\mathfrak m},Y_{\mathfrak m} \rangle\,,\qquad X,Y\in \mathfrak g\,.
\end{equation}
\item Affine connections on $M$ which are $G$-invariant  are in one-to-one correspondence to the  bilinear maps
\begin{equation}
    \Lambda\colon\mathfrak m\times\mathfrak m\longrightarrow\mathfrak m
\end{equation}
which are $\Ad(G_o)$-equivariant, namely
\begin{equation}
    \Lambda(\Ad_hX,\Ad_hY)=\Ad_h\Lambda(X,Y)\, ,
    \qquad h\in G_o\, ,
    \qquad X,Y\in\mathfrak m\, .
\end{equation}
Equivalently, if $G_o$ is connected, this condition is
\begin{equation}
    \Lambda([A,X]_{\mathfrak m},Y)
    +
    \Lambda(X,[A,Y]_{\mathfrak m})
    =
    [A,\Lambda(X,Y)]_{\mathfrak m}\, ,
    \qquad A\in\mathfrak g_o\, ,
    \qquad X,Y\in\mathfrak m\, .
\end{equation}
The invariant affine connection associated to $\Lambda$ is characterized by the convention
\begin{equation}
    \widetilde{(\Lambda(X,Y))}_o
    =
    \left(\nabla_{\tilde X}\tilde Y-[\tilde X,\tilde Y]\right)_o\, ,
    \qquad X,Y\in\mathfrak m\, ;
    \label{Eq: Nomizu operator convention}
\end{equation}
with this convention, the action of the corresponding torsion tensor at $o\in M$ is
\begin{equation}
    (\Tor^\nabla(\tilde X,\tilde Y))_o
    =
    (\widetilde{\Lambda(X,Y)-\Lambda(Y,X)-[X,Y]_{\mathfrak m}})_o\, ,
    \qquad X,Y\in\mathfrak m\, ,
    \label{Eq: Nomizu torsion formula}
\end{equation}
while the action of the curvature  at $o$ is
\begin{align}
    \left(R^\nabla(\tilde X,\tilde Y)\,\tilde Z\right)_o
    &=
    \widetilde{\left(
        \Lambda(X,\Lambda(Y,Z))
        -
        \Lambda(Y,\Lambda(X,Z))
        -
        \Lambda([X,Y]_{\mathfrak m},Z)
        -
        [[X,Y]_{\mathfrak g_o},Z]
    \right)}_o\, .
    \label{Eq: Nomizu curvature formula}
\end{align}
\end{enumerate}
%\end{proposition}

Among the invariant affine connections associated to the given  reductive decomposition, we consider the one-parameter family $\{\nabla^{(\lambda)}\}_{\lambda\in\mathbb R}$ whose bilinears, according to the convention \eqref{Eq: Nomizu operator convention}, are given by 
\begin{equation}\label{Eq: lambda Nomizu multiplication}
    \Lambda^{(\lambda)}(X,Y)=\lambda\,[X,Y]_{\mathfrak m}\,, \qquad X,Y\in\mathfrak m\,.
\end{equation}
From the identity $[\tilde X,\tilde Y]_o=-(\widetilde{[X,Y]_{\mathfrak m}})_o$, this is equivalent to
\begin{equation}\label{Eq: lambda connection on fundamental fields}
    \left(\nabla^{(\lambda)}_{\tilde X}\tilde Y\right)_o
    =(1-\lambda)[\tilde X,\tilde Y]_o
    =(\lambda-1)\widetilde{([X,Y]_{\mathfrak m})}_o\,, \qquad X,Y\in\mathfrak m\,.
\end{equation}
Following the nomenclature introduced in \cite{Nomizu-1954}, $\nabla^{(1/2)}$ is the \underline{canonical affine connection of the first kind}, whereas $\nabla^{(0)}$ is the \underline{canonical affine connection of the second kind}. 
The action of the  torsion tensor corresponding to $\nabla^{(\lambda)}$ is given by
\begin{equation}\label{Eq: lambda canonical torsion}
    \left(\Tor^{\nabla^{(\lambda)}}(\tilde X,\tilde Y)\right)_o
    =(2\lambda-1)\widetilde{([X,Y]_{\mathfrak m})}_o\,, \qquad X,Y\in\mathfrak m\,.
\end{equation}
In particular, the canonical connection of the first kind has a vanishing torsion. For every $\lambda\in\mathbb R\setminus\{\frac{1}{2}\}$, the connection $\nabla^{(\lambda)}$ is torsion-free if and only if
\begin{equation}
    [\mathfrak m,\mathfrak m]\subseteq\mathfrak g_o\,.
\end{equation}
The action of the curvature corresponding to $\nabla^{(\lambda)}$ reads
\begin{equation}\label{Eq: lambda canonical curvature}
    \left(R^{\nabla^{(\lambda)}}(\tilde X,\tilde Y)\tilde Z\right)_o
    =\widetilde{(\lambda^2\,[X,[Y,Z]_{\mathfrak m}]_{\mathfrak m}
        -\lambda^2\,[Y,[X,Z]_{\mathfrak m}]_{\mathfrak m}-\lambda\,[[X,Y]_{\mathfrak m},Z]_{\mathfrak m}
        -[[X,Y]_{\mathfrak g_o},Z])}_o\,.
\end{equation}

We finally determine the $g$-conjugate connections. Since both $g$ and $\nabla^{(\lambda)}$ are $G$-invariant, the conjugate connection $\left(\nabla^{(\lambda)}\right)^\dagger$ is also $G$-invariant. From
\begin{equation}
    \mathcal L_{\tilde Z}\left(g(\tilde X,\tilde Y)\right)
    =
    g([\tilde Z,\tilde X],\tilde Y)+g(\tilde X,[\tilde Z,\tilde Y])\,,
\end{equation}
one finds that its corresponding bilinear  is given by
\begin{equation}\label{Eq: lambda dual Nomizu multiplication}
    \left\langle X,\left(\Lambda^{(\lambda)}\right)^\dagger(Z,Y)\right\rangle
    =-\lambda\left\langle[Z,X]_{\mathfrak m},Y\right\rangle\,, \qquad X,Y,Z\in\mathfrak m\,
\end{equation}
and therefore one has that $\nabla^{(\lambda)}$ is $g$-self-conjugate if and only if
\begin{equation}\label{Eq: lambda metric compatibility}
    \lambda\,\left(
        \left\langle[Z,X]_{\mathfrak m},Y\right\rangle
        +\left\langle X,[Z,Y]_{\mathfrak m}\right\rangle
    \right)=0
\end{equation}
for every $X,Y,Z\in\mathfrak m$. In particular, this condition is always satisfied for $\lambda=0$, so the canonical connection of the second kind is compatible with every $G$-invariant pseudo-Riemannian metric. For $\lambda\neq0$, compatibility is equivalent to the natural reductivity condition
\begin{equation}\label{Eq: naturally reductive condition}
    \left\langle[Z,X]_{\mathfrak m},Y\right\rangle
    +\left\langle X,[Z,Y]_{\mathfrak m}\right\rangle
    =0\,, \qquad X,Y,Z\in\mathfrak m\,.
\end{equation}
This shows that any reductive homogeneous pseudo-Riemannian manifold $(M,g)$ determines the one-parameter family of metric-affine manifolds
\begin{equation}\label{Eq: Nomizu metric-affine manifold}
   \left(M,g,\nabla^{(\lambda)}\right)\,,\qquad \lambda\in \mathbb R\,.
\end{equation}
\subsubsection{A solution to the inverse problem}

In order to write the bi-form which solves the inverse problem for \eqref{Eq: Nomizu metric-affine manifold}, we begin by computing the action of the  exponential map corresponding to the  canonical affine connections $\nabla^{(\lambda)}$ on $M$. Consider the element $Z\in\mathfrak{m}$, and let 
\begin{equation}\label{Eq: gamma Z}
    \gamma_Z(t)=\exp(tZ)\cdot o
\end{equation}
be the integral curve of the (fundamental) vector field $\tilde Z$ on $M$ through $o\in M$. Consider also an element $X\in\mathfrak{m}$ and $E$ to be a vector field on $M$ whose restriction along $\gamma_Z(t)$ gives $E(0)=\tilde X_o$. The action of the parallel transport induced by $\nabla^{(\lambda)}$ along $\gamma_Z$ can be explicitly analysed and is described as the claim of the following lemma, in terms of the map 
$\ad_Z^{\mathfrak m}\in\operatorname{End}(\mathfrak m)$ given by 
    $\ad_Z^{\mathfrak m}(Y)=[Z,Y]_{\mathfrak m}$ and the corresponding exponential, whose action can be expanded as
\begin{equation}
e^{-t\lambda\ad_Z^{\mathfrak{m}}}X
=
X-t\lambda\,[Z,X]_{\mathfrak{m}}
+\frac{(t\lambda)^2}{2}[Z,[Z,X]_{\mathfrak{m}}]_{\mathfrak{m}}
+\cdots\,.
\end{equation}
\begin{lemma}\label{Lemma: homogeneous parallel field}
Along the previous setting, one has 
\begin{equation}\label{Eq: homogeneous parallel field}
    E(\gamma_Z(t))=\left(\tau_{\exp(tZ)}\right)'_o\left(\widetilde{(e^{-t\lambda\,\ad_Z^{\mathfrak m}}X)}_o\right)\,.
\end{equation}
\end{lemma}
\begin{proof}
Write
\begin{equation}\label{Eq: homogeneous Lie algebra representative}
    E(t)=\left(\tau_{\exp(tZ)}\right)'_o\left(\tilde Y(t)_o\right)\,, \qquad Y(t)\in\mathfrak m\,.
\end{equation}
Let $\{Z_j\}_{j=1}^{\dim M}$ be a basis of $\mathfrak m$, so that
\begin{align}
    Y(t)=Y^j(t)Z_j\,,\\  \dot Y(t)=\dot Y^j(t)Z_j\,.
\end{align}
As proven in \cite{Schlarb2024}, the parallel transport comes as the solution of the Cauchy problem
\begin{align}
    &\dot Y(t)+\Lambda^{(\lambda)}(Z,Y(t))=0\,, \\ & Y(0)=X\,,
\end{align}
which is a linear ODE whose solution one can write as 
    $Y(t)=e^{-t\,\lambda\,\ad_Z^{\mathfrak m}}X.$
Inserting this expression into \eqref{Eq: homogeneous Lie algebra representative} reads the claim \eqref{Eq: homogeneous parallel field}.
\end{proof}

Taking $X=Z$ in \eqref{Eq: homogeneous parallel field}, and noticing  that $\ad_Z^{\mathfrak m}(Z)=0$, we obtain
\begin{equation}
    E(t)=\left(\tau_{\exp(t\,Z)}\right)'_o\tilde Z_o=\dot\gamma_Z(t)\,.
\end{equation}
This shows that $\gamma_Z$ is a $\nabla^{(\lambda)}$-geodesic and therefore one has, for the exponential map corresponding to $\nabla^{(\lambda)}$ for any $\lambda\in\mathbb R$, that   at $o$, it is  
\begin{equation}\label{Eq: exponential}
    \exp^{\nabla^{(\lambda)}}_o(\tilde Z_o)=\exp(Z)\cdot o\,.
\end{equation}
We next show that a suitable $\nabla^{(\lambda)}$-convex neighbourhood $C$ around $o$ gives, under the action of $G$, a strongly $\nabla^{(\lambda)}$-convex covering
\begin{equation}
    \mathcal C_o=\{z\cdot C\mid z\in G\}
\end{equation}
for every $\lambda\in\mathbb R$. The resulting covering is $G$-invariant, and therefore the associated strongly convex neighbourhood is invariant under the diagonal action of $G$ on $M\times M$.
\begin{lemma}\label{Lemma: invariant strongly convex covering}
Let $G$ be a Lie group acting transitively on a smooth manifold $M$, let $o\in M$, and let $\mathfrak g=\mathfrak g_o\oplus\mathfrak m$ be a reductive decomposition. Assume that $M$ admits a $G$-invariant Riemannian metric. There exists an open neighbourhood $C$ around $o\in M$ which is $\nabla^{(\lambda)}$-convex for every $\lambda\in\mathbb R$ and such that
\begin{equation}
    \mathcal C_o=\{z\cdot C\mid z\in G\}=\{z\cdot m\,\mid m\in C\,,z\in G\}
\end{equation}
is a strongly $\nabla^{(\lambda)}$-convex covering of $M$ for every $\lambda\in\mathbb R$.
\end{lemma}
\begin{proof}
    Let $g$ be the $G$-invariant Riemannian metric determined by the $\Ad(G_o)$-invariant inner product $\langle\cdot,\cdot\rangle$ on $\mathfrak m$, and let $\delta$ denote the induced Riemannian distance. For every $s>0$, set
    \begin{equation}
        C_s=\exp^G\left(B_s^{\mathfrak m}(0)\right)\cdot o\,,\qquad B_s^{\mathfrak m}(0)=\left\{X\in\mathfrak m\mid\lVert X\rVert<s\right\}\,.
    \end{equation}
    By the existence of arbitrarily small convex neighbourhoods for affine connections \cite{Postnikov-2001}, and since by \eqref{Eq: exponential} all the connections $\nabla^{(\lambda)}$ have the same geodesics, there exists $R>0$ such that $C_R$ is a $\nabla^{(\lambda)}$-convex neighbourhood of $o$ for every $\lambda\in\mathbb R$. Since $C_R$ is open, there exists $\rho>0$ such that
    \begin{equation}\label{Eq: B(o,rho) in CR}
        B_\delta(o,\rho)\subseteq C_R\,:
    \end{equation}
    we may then choose $r>0$ such that $3r<\rho$ and $C_r$ is $\nabla^{(\lambda)}$-convex for every $\lambda\in\mathbb R$.
    We first notice  that
    \begin{equation}\label{Eq: Cr in B(o,r)}
        C_r\subseteq B_\delta(o,r)\,.
    \end{equation}
    Indeed, for every $X\in B_r^{\mathfrak m}(0)$, the curve $\gamma_X(t)=\tau_{\exp(tX)}(o)$ has tangent vector given by 
    \begin{equation}
        \dot{\gamma}_X(t)=\left(\tau_{\exp(tX)}\right)'_o\tilde X_o\,.
    \end{equation}
    Since the Lie group $G$ acts by isometries, the norm of such a vector  $\dot\gamma_X$ is constant: 
    \begin{equation}
        g_{\gamma_X(t)}\left(\dot{\gamma}_X(t),\dot{\gamma}_X(t)\right)=g_o\left(\tilde X_o,\tilde X_o\right)=\lVert X\rVert^2\,.
    \end{equation}
    Therefore,
    \begin{equation}
        \delta\left(o,\exp(X)\cdot o\right)\leq\operatorname{Length}_g(\gamma_X)=\lVert X\rVert<r\,.
    \end{equation}
    To prove the claim, we check that, for every $z\in G$, the intersection $(z\cdot C_r)\cap C_r$ is either empty or $\nabla^{(\lambda)}$-convex for every $\lambda\in\mathbb R$. Suppose that it is nonempty, and choose $m\in (z\cdot C_r)\cap C_r$. Given any $n\in z\cdot C_r$, since  $\delta$ is $G$-invariant and $z^{-1}\cdot m,z^{-1}\cdot n\in C_r$, the triangle inequality gives
    \begin{align}
        \delta(o,n)
        &\leq\delta(o,m)+\delta(m,n)\\ 
        &<r+\delta\left(z^{-1}\cdot m,z^{-1}\cdot n\right) \\
        &\leq r+\delta\left(z^{-1}\cdot m,o\right)+\delta\left(o,z^{-1}\cdot n\right)\\
        &<3r<\rho\,.
    \end{align}
    This shows that $z\cdot C_r\subseteq B_\delta(o,\rho)\subseteq C_R$. On the other hand, $C_r\subseteq C_R$ by \eqref{Eq: Cr in B(o,r)} and \eqref{Eq: B(o,rho) in CR}. By construction, $C_r$ is $\nabla^{(\lambda)}$-convex, while the $G$-invariance of $\nabla^{(\lambda)}$ implies that $z\cdot C_r$ is $\nabla^{(\lambda)}$-convex. Since they are contained in the same $\nabla^{(\lambda)}$-convex neighbourhood $C_R$, their intersection $z\cdot C_r\cap C_r$ is $\nabla^{(\lambda)}$-convex.
    Finally, for every $z_1,z_2\in G$,
    \begin{equation}
        (z_1\cdot C_r)\cap (z_2\cdot C_r)=z_1\cdot\left(C_r\cap (z_1^{-1}z_2\cdot C_r)\right)
    \end{equation}
    is either empty or $\nabla^{(\lambda)}$-convex for every $\lambda\in\mathbb R$. Since the images of $C_r$ under the action of $G$ cover $M$, the family $\mathcal C_o$ is a strongly $\nabla^{(\lambda)}$-convex covering for every $\lambda\in\mathbb R$.
\end{proof}

Let $C_r=\exp(B_r^{\mathfrak m}(0))\cdot o$ and $C_R=\exp(B_R^{\mathfrak m}(0))\cdot o$ be as in the previous lemma, and consider the strongly $\nabla^{(\lambda)}$-convex covering $\mathcal C_o$ associated to $C_r$. The corresponding neighbourhood of the diagonal is
\begin{equation}\label{Eq: diagonal neighbourhood homogeneous}
    \mathscr U_o=\left\{\left(z\cdot\exp(A)\cdot o,z\cdot\exp(B)\cdot o\right)\mid z\in G,\ A,B\in B_r^{\mathfrak m}(0)\right\}\,.
\end{equation}
The construction of $C_r$ ensures that $\exp(-A)\cdot\exp(B)\cdot o\in C_R$. Hence, the $\nabla^{(\lambda)}$-geodesic interpolation on $\mathscr U_o$ is well defined by
\begin{equation}\label{Eq: homogeneous geodesic interpolation}
    \Gamma^{(\lambda)}\left(z\cdot\exp(A)\cdot o,z\cdot\exp(B)\cdot o,t\right)
    =
    z\cdot\exp(A)\cdot\mathcal E\left(t\,\mathcal L\left(\exp(-A)\cdot\exp(B)\cdot o\right)\right)\,.
\end{equation}
Here $\mathcal E\colon\mathfrak m\to M$ is given by
\begin{equation}
    \mathcal E(Z)=\exp(Z)\cdot o\,,
\end{equation}
and $\mathcal L$ denotes the inverse of the restriction of $\mathcal E$ to $B_R^{\mathfrak m}(0)$.
We now compute the $\nabla^{(\lambda)}$-geodesic $\nabla^{(\lambda)}$-parallel propagator
\begin{equation}
    P^{(\lambda)}=P^{\nabla^{(\lambda)},\nabla^{(\lambda)}}
\end{equation}
on $\mathscr U_o$, and the associated contrast bi-form $\varpi^{(\lambda)}$. Since $\nabla^{(\lambda)}$ is $G$-invariant, its parallel transport is $G$-equivariant. More precisely, let $\gamma\colon[0,1]\to M$ satisfy $\gamma(0)=m$ and $\gamma(1)=n$. Then the following diagram commutes:
\begin{equation}\label{Diagram: G-invariance parallel transport}
\begin{tikzcd}[column sep=large,row sep=large]
    \T_mM
    \arrow[r,"P^{\nabla^{(\lambda)},\gamma}_{1\leftarrow0}"]
    \arrow[d,"(\tau_z)'_m"']
    &
    \T_nM
    \arrow[d,"(\tau_z)'_n"]
    \\
    \T_{z\cdot m}M
    \arrow[r,"P^{\nabla^{(\lambda)},\tau_z\circ\gamma}_{1\leftarrow0}"']
    &
    \T_{z\cdot n}M
\end{tikzcd}
\end{equation}
Equivalently, it is 
\begin{equation}\label{Eq: equinvariance}
    P^{\nabla^{(\lambda)},\tau_z\circ\gamma}_{1\leftarrow0}\circ(\tau_z)'_m
    =
    (\tau_z)'_n\circ P^{\nabla^{(\lambda)},\gamma}_{1\leftarrow0}\,.
\end{equation}
It  is therefore enough to evaluate the action of the parallel transport along curves of the form $\gamma_Z$, with $Z\in\mathfrak m$ (cf. \eqref{Eq: gamma Z}). By \eqref{Eq: homogeneous parallel field}, for every $X,Z\in\mathfrak m$ such that $(o,\exp(Z)\cdot o)\in\mathscr U_o$, one has
\begin{equation}\label{Eq: parallel transport lambda homogeneous}
    P^{\nabla^{(\lambda)},\gamma_Z}_{1\leftarrow0}\tilde X_o
    =
    \left(\tau_{\exp(Z)}\right)'_o
    \left(\widetilde{(e^{-\lambda\,\ad_Z^{\mathfrak m}}X)}_o\right)\,:
\end{equation}
let
\begin{equation}
    (m,n)=\left(z\cdot\exp(A)\cdot o,z\cdot\exp(B)\cdot o\right)\in\mathscr U_o
\end{equation}
and set
\begin{equation}
    Z=\mathcal L\left(\exp(-A)\cdot\exp(B)\cdot o\right)\,.
\end{equation}
The relation \eqref{Eq: equinvariance} gives, for the interpolation \eqref{Eq: homogeneous geodesic interpolation}, 
\begin{equation}
    P^{\nabla^{(\lambda)},\Gamma^{(\lambda)}(m,n,\cdot)}_{1\leftarrow0}\tilde X_m
    =
    \left(\tau_{z\cdot\exp(A)}\right)'_{\exp(-A)\cdot\exp(B)\cdot o}
    P^{\nabla^{(\lambda)},\gamma_Z}_{1\leftarrow0}
    \left(\tau_{\exp(-A)\cdot z^{-1}}\right)'_m\tilde X_m\,.
\end{equation}
Combining this identity with \eqref{Eq: parallel transport lambda homogeneous} and the transformation law of fundamental vector fields, we obtain
\begin{equation}
    P^{\nabla^{(\lambda)},\Gamma^{(\lambda)}(m,n,\cdot)}_{1\leftarrow0}\tilde X_m
    =
    \widetilde{(\Ad_{z\cdot\exp(A)\cdot\exp(Z)}
    (e^{-\lambda\,\ad_Z^{\mathfrak m}}
    (\Ad_{\exp(-A)\cdot z^{-1}}X)_{\mathfrak m}))}_n\,.
\end{equation}
The $G$-invariance of $g$ then gives
\begin{align}
    \varpi^{(\lambda)}(\tilde X\mid \tilde Y)(m,n)
    &=
    g_n\left(
    \widetilde{(\Ad_{z\cdot\exp(A)\cdot\exp(Z)}
    (e^{-\lambda\,\ad_Z^{\mathfrak m}}
    (\Ad_{\exp(-A)\cdot z^{-1}}X)_{\mathfrak m}))}_n,
    \tilde Y_n\right)\\
    &=
    g_o\left(\widetilde{(e^{-\lambda\,\ad_Z^{\mathfrak m}}
    (\Ad_{\exp(-A)\cdot z^{-1}}X)_{\mathfrak m})}_o,
    \widetilde{(\Ad_{\exp(-Z)\cdot\exp(-A)\cdot z^{-1}}Y)}_o\right)\,.
\end{align}
Hence, the solution bi-form reads
\begin{equation}\label{Eq: homogeneous lambda bi-form}
    \varpi^{(\lambda)}(\tilde X\mid \tilde Y)
    \left(z\cdot\exp(A)\cdot o,z\cdot\exp(B)\cdot o\right)
    =
    \left\langle
    e^{-\lambda\,\ad_Z^{\mathfrak m}}
    \left(\Ad_{\exp(-A)\cdot z^{-1}}X\right)_{\mathfrak m},
    \left(\Ad_{\exp(-Z)\cdot\exp(-A)\cdot z^{-1}}Y\right)_{\mathfrak m}
    \right\rangle\,.
\end{equation}
%In particular, taking $z=e$, $A=0$ and $B=Z$, we obtain\begin{equation} \varpi^{(\lambda)}(X^+\mid Y^+)\left(o,\exp^G(Z)\cdot o\right)=\left\langlee^{(1-\lambda)\ad_Z^{\mathfrak m}}X,\left(\Ad_{\exp^G(-Z)}Y\right)_{\mathfrak m}\right\rangle\,.\end{equation}

\subsection{Cartan-Schouten connections on semisimple Lie groups}\label{Sec: CS on Lie groups}
Any Lie group $G$ is a reductive homogeneous space under the action on itself by left translations $L_a(b)=a\cdot b$. Indeed, the isotropy subgroup at the identity $e$ is trivial, and the corresponding reductive decomposition is
\begin{equation}
    \mathfrak g=\{0\}\oplus\mathfrak g\,.
\end{equation}
In this setting, $G$-invariant metrics and affine connections are precisely left-invariant metrics and affine connections. On the other hand, the fundamental vector field generated by $X\in\mathfrak g$ is the right-invariant vector field associated to $X$
\begin{equation}\label{Eq: from right to left}
    \tilde X_a=(L_a)_e' \left(\Ad_{a^{-1}}X\right)_e\,, \qquad a\in G\,.
\end{equation}
If $G$ is semisimple, there is a distinguished pseudo-Riemannian metric, the \emph{Cartan-Killing form}, which is 
\begin{equation}
    \kappa(X,Y)=\Tr\left(\ad_X\circ\ad_Y\right)\,,
    \qquad X,Y\in\mathfrak g.
\end{equation}
The form $\kappa$  is $\Ad(G)$-invariant, namely
\begin{equation}
    \kappa(\Ad_zX,\Ad_zY)=\kappa(X,Y)\,,
    \qquad X,Y\in\mathfrak g\,,\quad z\in G\,.
\end{equation}
Equivalently, if $G$ is connected, it is 
\begin{equation}
    \kappa([Z,X],Y)=-\kappa(X,[Z,Y])\,,
    \qquad X,Y,Z\in\mathfrak g\,.
\end{equation}
In particular, by \eqref{Eq: lambda metric compatibility}, all the canonical affine connections $\nabla^{(\lambda)}$ are compatible with $\kappa$. Since $\nabla^{(1/2)}$ is also torsion-free, then it is the Levi-Civita connection corresponding to $\kappa$. Formula \eqref{Eq: lambda canonical curvature} reduces to
\begin{equation}\label{Eq: CS curvature}
    R^{\nabla^{(\lambda)}}(\tilde X,\tilde Y)\,\tilde Z
    =
    \lambda(\lambda-1)[[\tilde X,\tilde Y],\tilde Z]\,,
    \qquad X,Y,Z\in\mathfrak g\,.
\end{equation}
This shows that $\nabla^{(\lambda)}$ has a vanishing curvature  if and only if $\lambda\in\{0,1\}$. The curvature-free connections $\nabla^{(0)}$ and $\nabla^{(1)}$ are known, respectively, as the \underline{left Cartan} and \underline{right Cartan} affine connections.

\subsubsection{A solution to the inverse problem}

Let us write the bi-form which solves the inverse problem for the metric-affine manifolds 
$\left(G,\kappa,\nabla^{(\lambda)}\right).$
Using \eqref{Eq: homogeneous lambda bi-form}, the germ of the solution bi-form $\varpi^{(\lambda)}$ along the diagonal submanifold is\footnote{Notice that $\log\,:\,G\,\to\,\mathfrak{g}$ is the inverse, where it exists, of the exponential map for a given Lie group. If $G$ is compact and connected, for example, the exponential is known to be surjective.} 
\begin{equation}\label{Eq: CS solution biform}
    \varpi^{(\lambda)}(\tilde X\mid \tilde Y)(a,b)
    =
    \kappa\left(
    \Ad_{\exp\left(-\lambda\,\log\left(a^{-1}\cdot b\right)\right)}
    \Ad_{a^{-1}}X,
    \Ad_{b^{-1}}Y
    \right)\,,
    \qquad X,Y\in\mathfrak g\,,\quad (a,b)\in\mathscr U_e\,,
\end{equation}
where $\mathscr U_e$ is the $\nabla^{(\lambda)}$-convex neighbourhood constructed as in \cref{Lemma: invariant strongly convex covering}.
For $\lambda=0$ and $\lambda=1$, this recovers the left and right Cartan cases considered in \cite{CMPSZ2026}, where the solution bi-forms were expressed in terms of left-invariant vector fields. Here, instead, we evaluate them on the right-invariant fundamental vector fields associated to the left action of $G$ on itself.

\begin{example}\label{Example: SU}
Let $G=\SU(d+1)$, and identify its Lie algebra $\su(d+1)$ with the real vector space of traceless skew-Hermitian matrices. With this convention, the Cartan-Killing form is
\begin{equation}
    \kappa(X,Y)=2\,(d+1)\,\Tr(XY)\,.
\end{equation}
By diagonal left-invariance, the solution bi-form in \eqref{Eq: CS solution biform} is determined by its action on pairs of the form $(\mathbb I_{d+1},e^Z)$ in a diagonally left-invariant strongly $\nabla^{(\lambda)}$-convex domain, and it reads
\begin{equation}\label{Eq: SU solution biform}
    \varpi^{(\lambda)}(\tilde X\mid \tilde Y)(\mathbb I_{d+1},e^Z)
    =
    \kappa\left(e^{-\lambda Z}Xe^{\lambda Z},e^{-Z}Ye^Z\right)
    =
    2\,(d+1)\,\Tr\left(e^{-\lambda Z}\,X\,e^{\lambda Z}\,e^{-Z}Ye^Z\right)\,.
\end{equation}
Note that in the curvature-free case, we have
\begin{align}
    \varpi^{(0)}(\tilde X\mid \tilde Y)(\mathbb I_{d+1},e^Z)
    &=2\,(d+1)\,\Tr\left(X\,e^{-Z}\,Y\,e^Z\right)\,,\\
     \varpi^{(1)}(\tilde X\mid \tilde Y)(\mathbb I_{d+1},e^Z)
    &=2\,(d+1)\,\Tr\left(XY\right)\,.
\end{align}
By \eqref{Eq: from right to left}, when written in terms of left-invariant vector fields, these expressions recover the bi-forms obtained in \cite{CMPSZ2026}.
\end{example}

In the rest of this section,  we use this example to study the metric-affine manifolds induced by the left-exact and bi-exact components of $\varpi^{(\lambda)}$. As we shall see, passing to these exact components removes the dependence on $\lambda$ and loses the information on the action of the corresponding torsion, which is encoded by the full solution bi-forms. 

Since $\nabla^{(\lambda)}$ is $\kappa$-self-conjugate, the potential functions \eqref{Eq: canonical pre-contrast general refined}, \eqref{Eq: Ay-Amari}, and \eqref{Eq: Hook law} become, for every $(a,b)\in\mathscr U_e$ and $Y\in\mathfrak g$,
\begin{align}
    S(\mid Y)(a,b)&=\kappa\left(\log(b^{-1}\cdot a),\Ad_{b^{-1}}Y\right)\,,\label{Eq: CS precontrast}\\
    D(a,b)&=-\frac12\,\kappa\left(\log(a^{-1}\cdot b),\log(a^{-1}\cdot b)\right)\,.\label{Eq: CS contrast}
\end{align}
Both $S$ and $D$ do not depend on $\lambda$.

\begin{lemma}
On $\mathscr U_e$, the local potentials \eqref{Eq: CS precontrast} and \eqref{Eq: CS contrast} satisfy
\begin{equation}
    S=\mathrm d^R D\,.
\end{equation}
Moreover, the contrast bi-form
\begin{equation}
    \mathrm d^L S=\mathrm d^L\mathrm d^R D
\end{equation}
is statistically equivalent to $\varpi^{(1/2)}$.
\end{lemma}

\begin{proof}
Fix $(a,b)\in\mathscr U_e$ and set $Z=\log(a^{-1}\cdot b)$. We use the Baker-Campbell-Hausdorff differential formula \cite[Section 5.5]{Hall}
\begin{equation}\label{Eq: BCH differential}
    \frac{\mathrm d}{\mathrm dt}\bigg|_{t=0}\log\left(\exp(A)\cdot\exp(t\,B)\right)=\frac{\ad_A}{1-\exp(-\ad_A)}B\,.
\end{equation}
Since the vector field $\tilde Y$ is right-invariant, its integral curve through $b$ is $\gamma(t)=\exp(t\,Y)\cdot b$. Moreover,
\begin{equation}
    a^{-1}\cdot\exp(tY)\cdot b=a^{-1}\cdot b\cdot\exp\left(t\,\Ad_{b^{-1}}Y\right)=\exp(Z)\cdot\exp\left(t\,\Ad_{b^{-1}}Y\right)\,.
\end{equation}
It follows from \eqref{Eq: BCH differential} that
\begin{equation}
    \mathrm d^R D(\mid Y)(a,b)=-\kappa\left(Z,\frac{\ad_Z}{1-\exp(-\ad_Z)}\Ad_{b^{-1}}Y\right)=-\kappa\left(Z,\Ad_{b^{-1}}Y\right)=S(\mid Y)(a,b)\,.
\end{equation}
Indeed, the $\ad$-invariance of $\kappa$ implies that every term containing a positive power of $\ad_Z$ vanishes when paired to $Z$. This reads $S=\mathrm d^R D$.
We next compute $\mathrm d^L S$. For $X\in \mathfrak g$,
\begin{equation}
    b^{-1}\cdot\exp(sX)\cdot a=b^{-1}\cdot a\cdot\exp\left(s\,\Ad_{a^{-1}}X\right)=\exp(-Z)\cdot\exp\left(s\,\Ad_{a^{-1}}X\right)\,.
\end{equation}
Thus, using again the relation \eqref{Eq: BCH differential}, we obtain
\begin{equation}\label{Eq: CS dLS}
    \mathrm d^L S(X\mid Y)(a,b)=\kappa\left(\frac{-\ad_Z}{1-\exp(\ad_Z)}\Ad_{a^{-1}}X,\Ad_{b^{-1}}Y\right)\,.
\end{equation}
On the other hand, the relation \eqref{Eq: CS solution biform} reads
\begin{equation}\label{Eq: CS varpi lambda}
    \varpi^{(\lambda)}(\tilde X\mid \tilde Y)(a,b)=\kappa\left(e^{-\lambda\ad_Z}\Ad_{a^{-1}}X,\Ad_{b^{-1}}Y\right)\,.
\end{equation}
Comparing the scalar Taylor expansions
\begin{align}
    \frac{-u}{1-\exp(u)}&=1-\frac12\,u+\frac{1}{12}\,u^2+o(u^3)\,,\label{Eq: CS scalar dLS expansion}\\
    \exp(-\lambda u)&=1-\lambda u+\frac{\lambda^2}{2}\,u^2-\frac{\lambda^3}{6}\,u^3+o(u^3)\,,\label{Eq: CS scalar varpi expansion}
\end{align}
we see that their constant and first-order terms agree precisely for $\lambda=\frac12$. Hence $\mathrm d^L S-\varpi^{(1/2)}$ vanishes to first order along the diagonal, and the two contrast bi-forms are statistically equivalent.
\end{proof}
\subsection{Odd-dimensional spheres with Berger metrics}\label{Sec: Berger metrics}

We conclude our analysis upon studying the Berger spheres \cite{Berger1961,D-G-P-2016} endowed with the canonical affine connection of the second kind. They give a  family of homogeneous metric-affine manifolds whose curvature and torsion are, in general, both non vanishing.

Let $d\geq 1$. The odd-dimensional sphere $S^{2d+1}\subseteq\mathbb C^{d+1}$ is a $\SU(d+1)$-homogeneous space with respect to the standard matrix action of $\mathrm{GL}(\mathbb{C}, d+1)$. The isotropy subgroup at $\mathbf e_0=(1,0,\dots,0)\in\mathbb C^{d+1}$ is isomorphic to $\SU(d)$, and its Lie algebra is
\begin{equation}
    \su(d+1)_{\mathbf e_0}
    =
    \left\{
        \begin{pmatrix}
            0 & \mathbf 0^\ast\\
            \mathbf 0 & A
        \end{pmatrix}
        \,\middle|\,
        A\in\su(d)
    \right\}\,.
\end{equation}
Using the Cartan-Killing form $\kappa$ corresponding to $\su(d+1)$, we select the orthogonal complement
\begin{equation}
    \mathfrak m=\left(\su(d+1)_{\mathbf e_0}\right)^{\perp_\kappa}
    =
    \left\{
        X(a,\mathbf b)\mid
        a\in\mathbb R\,,\ \mathbf b\in\mathbb C^d
    \right\}\,, \qquad  X(a,\mathbf b)
    =
    \begin{pmatrix}
        \mathrm i\,a & -\mathbf b^\ast\\
        \mathbf b & -\frac{\mathrm i\,a}{d}\,\mathbb I_d
    \end{pmatrix}\,.
\end{equation}
The $\ad$-invariance of $\kappa$ implies that $\mathfrak m$ is $\Ad\left(\SU(d)\right)$-invariant, and therefore
\begin{equation}
    \su(d+1)
    =
    \su(d+1)_{\mathbf e_0}\oplus\mathfrak m
    \cong
    \su(d)\oplus\mathfrak m
\end{equation}
is a reductive decomposition. As a vector space, $\mathfrak m$ decomposes further as
\begin{equation}
    \mathfrak m
    =
    \mathfrak m_{\mathrm v}\oplus\mathfrak m_{\mathrm h}
    \cong
    \mathbb R\oplus\mathbb C^d\,,
\end{equation}
where $\mathfrak m_{\mathrm v}$ is the one-dimensional subspace generated by $X(1,\mathbf 0)$, while $\mathfrak m_{\mathrm h}$ is the $2d$-dimensional real vector space given by $\{X(0,\mathbf b)\mid \mathbf b\in \mathbb C^d\}$.

The \emph{Berger metrics} provide  the family of $\SU(d+1)$-invariant pseudo-Riemannian metrics
\begin{equation}\label{Eq: Berger metrics}
    \left\langle X(a,\mathbf b),X(c,\mathbf d)\right\rangle_\varepsilon
    =
    \Re\left\langle\mathbf b,\mathbf d\right\rangle_{\mathbb C^d}
    +
    \varepsilon\,a\,c\,,
    \qquad
    \varepsilon\ne 0\,.
\end{equation}
For $\varepsilon>0$, the metric $g_\varepsilon$ is Riemannian, whereas for $\varepsilon<0$ it is Lorentzian. For $d\geq2$, this family exhausts, up to an overall positive constant, the $\SU(d+1)$-invariant pseudo-Riemannian metrics on $S^{2d+1}$.
For simplicity, among the family of homogeneous metric-affine manifolds determined by the Berger metrics and the canonical connections, we restrict our attention to
\begin{equation}\label{Eq: Berger second canonical metric-affine manifold}
    \left(S^{2d+1},g_\varepsilon,\nabla^{(0)}\right)\,,
    \qquad\varepsilon\ne 0\,.
\end{equation}
The connection $\nabla^{(0)}$ is the \emph{canonical connection of the second kind}. With the convention \eqref{Eq: Nomizu operator convention}, its bilinear (i.e. Nomizu multiplication) vanishes identically, i.e. 
\begin{equation}
    \Lambda^{(0)}(X,Y)=0\,,
    \qquad X,Y\in\mathfrak m\,,
\end{equation}
or, equivalently, one has 
\begin{equation}\label{Eq: Berger second canonical connection}
    \left(\nabla^{(0)}_{\tilde X}\tilde Y\right)_{\mathbf e_0}
    =
    [\tilde X,\tilde Y]_{\mathbf e_0}
    =
    -\widetilde{([X,Y]_{\mathfrak m})}_{\mathbf e_0}\,,
    \qquad X,Y\in\mathfrak m\,.
\end{equation}
Since the zero multiplication is compatible with any $\Ad\left(\SU(d)\right)$-invariant non-degenerate symmetric bilinear form on $\mathfrak m$ (cf. \eqref{Eq: lambda metric compatibility}), the connection $\nabla^{(0)}$ is compatible with any Berger metric, and therefore one has
\begin{equation}
    \left(\nabla^{(0)}\right)^\dagger=\nabla^{(0)}\,.
\end{equation}
Before turning our attention  to writing the solution to the inverse problem, we show that, for $d\geq2$, $\nabla^{(0)}$ has non vanishing torsion and curvature. The computation is based on the commutation relation
\begin{align}\label{Eq: commutation relation sphere}
    [X(a,\mathbf b),X(c,\mathbf d)]=
    X\left(-2\,\Im\langle\mathbf b,\mathbf d\rangle_{\mathbb C^d},\frac{d+1}{d}\,\mathrm i(c\,\mathbf b-a\,\mathbf d)\right)+
    \begin{pmatrix}
        0 & \mathbf 0^\ast\\
        \mathbf 0 & \mathbf d\,\mathbf b^\ast-\mathbf b\,\mathbf d^\ast-\frac{2\,\mathrm i}{d}\Im\langle\mathbf b,\mathbf d\rangle_{\mathbb C^d}\mathbb I_d
    \end{pmatrix}\,.
\end{align}
Since the torsion of $\nabla^{(0)}$ is given by
\begin{equation}
    \left(\Tor^{\nabla^{(0)}}(\tilde X,\tilde Y)\right)_{\mathbf e_0}
    =
    -\widetilde{([X,Y]_{\mathfrak m})}_{\mathbf e_0}\,,
\end{equation}
we obtain
\begin{equation}\label{Eq: torsion second canonical connection sphere}
    \left(\Tor^{\nabla^{(0)}}\left(\tilde X(a,\mathbf b),\tilde X(c,\mathbf d)\right)\right)_{\mathbf e_0}
    =
    \tilde X\left(2\,\Im\langle\mathbf b,\mathbf d\rangle_{\mathbb C^d},-\frac{d+1}{d}\,\mathrm i(c\,\mathbf b-a\,\mathbf d)\right)_{\mathbf e_0}\,,
\end{equation}
which shows that $\nabla^{(0)}$ has a non vanishing torsion.
If we consider the relation \eqref{Eq: lambda canonical curvature} for  $\lambda=0$, we find
\begin{equation}\label{Eq: Berger second canonical curvature}
    \left(R^{\nabla^{(0)}}(\tilde X,\tilde Y)\tilde Z\right)_{\mathbf e_0}
    =
    -\widetilde{([[X,Y]_{\su(d+1)_{\mathbf e_0}},Z])}_{\mathbf e_0}\,,
    \qquad X,Y,Z\in\mathfrak m\,.
\end{equation}
For $d=1$ the isotropy algebra is trivial,  and one has  $S^3\cong\SU(2)$. In this case, \eqref{Eq: Berger second canonical curvature} vanishes identically, and $\nabla^{(0)}$ coincides with the curvature-free left Cartan connection on $\SU(2)$.

For $d\geq2$, let $\mathbf e_1,\mathbf e_2\in\mathbb C^d$ be two vectors of the standard basis. Then
\begin{equation}
    [X(0,\mathbf e_1),X(0,\mathbf e_2)]_{\su(d+1)_{\mathbf e_0}}
    =
    \begin{pmatrix}
        0 & \mathbf 0^\ast\\
        \mathbf 0 & \mathbf e_2\mathbf e_1^\ast-\mathbf e_1\mathbf e_2^\ast
    \end{pmatrix}
\end{equation}
and
\begin{equation}
    \left[[X(0,\mathbf e_1),X(0,\mathbf e_2)]_{\su(d+1)_{\mathbf e_0}},X(0,\mathbf e_1)\right]=X(0,\mathbf e_2)\,.
\end{equation}
It follows from \eqref{Eq: Berger second canonical curvature} that
\begin{equation}
    \left(R^{\nabla^{(0)}}\left(\tilde X(0,\mathbf e_1),\tilde X(0,\mathbf e_2)\right)\tilde X(0,\mathbf e_1)\right)_{\mathbf e_0}
    =
    -\tilde X(0,\mathbf e_2)_{\mathbf e_0}\neq0\,.
\end{equation}
This shows that the connection $\nabla^{(0)}$ has non vanishing curvature for every $d\geq2$. Since $\nabla^{(0)}$ is $g_\varepsilon$-self-conjugate, for any  $d\geq2$ both conjugate connections have non vanishing torsion and curvature. Thus, \eqref{Eq: Berger second canonical metric-affine manifold} provides a family of reductive homogeneous metric-affine manifolds whose conjugate connections both have non vanishing torsion and curvature.

\begin{remark}
The metric-affine structure of the Berger sphere $S^{2d+1}$ is naturally related to the Riemannian geometry of the complex projective space $\CP^d$. This relationship is realized by the Hopf fibration.
The action of $\SU(d+1)$ on $S^{2d+1}$ descends to the quotient via the \emph{Hopf fibration}
\begin{equation}
    \Pi\colon S^{2d+1}\longrightarrow\CP^d=S^{2d+1}/\U(1)\,,
\end{equation}
and turns $\CP^d$ into a reductive homogeneous $\SU(d+1)$-space. The isotropy subgroup at $[\mathbf e_0]$ is
\begin{equation}
    \SU(d+1)_{[\mathbf e_0]}=\mathrm S\bigl(\U(1)\times\U(d)\bigr)\,,
\end{equation}
and for its Lie algebra one has 
\begin{equation}
    \su(d+1)_{[\mathbf e_0]}=\su(d+1)_{\mathbf e_0}\oplus\mathfrak m_{\mathrm v}\,.
\end{equation}
It then follows that the reductive decomposition associated to $\CP^d$ is
\begin{equation}
    \su(d+1)=\su(d+1)_{[\mathbf e_0]}\oplus\mathfrak m_{\mathrm h}\,.
\end{equation}
Since the Hopf map is $\SU(d+1)$-equivariant, if $\hat X$ denotes the fundamental vector field induced by $X\in\su(d+1)$ on $\CP^d$, then
\begin{equation}\label{Eq: Hopf related fundamental fields}
    \Pi'\circ \tilde X=\hat{X}\circ\Pi\,.
\end{equation}
At the base point, this gives
\begin{equation}\label{Eq: Hopf related fundamental fields examples}
    \ker\Pi'_{\mathbf e_0}=\widetilde{\left(\mathfrak m_{\mathrm v}\right)}_{\mathbf e_0}\,,\qquad \Pi'_{\mathbf e_0}\tilde X(0,\mathbf b)_{\mathbf e_0}=\hat{X}(0,\mathbf b)_{[\mathbf e_0]}\,.
\end{equation}
In order to identify the component of the Berger metrics that descends through the Hopf map, consider the $\SU(d+1)$-invariant $1$-form on $S^{2d+1}$ determined by the covector on $\mathfrak m$ given by
\begin{equation}\label{Eq: contact form}
   \eta(X(a,\mathbf b))=a\,.
\end{equation}
By \eqref{Eq: Berger metrics}, the Berger metrics decompose as
\begin{equation}\label{Eq: Berger metric Hopf decomposition}
    g_\varepsilon=\Pi^\ast g^{\FS}+\varepsilon\,\eta\otimes \eta\,,%=g_{\varepsilon_\kappa}+(\varepsilon-\varepsilon_\kappa)\,\eta\otimes \eta
\end{equation}
where the Fubini-Study metric is normalized by
\begin{equation}
    g^{\FS}_{[\mathbf e_0]}\left(\hat{X}(0,\mathbf b)_{[\mathbf e_0]},\hat{X}(0,\mathbf d)_{[\mathbf e_0]}\right)=\Re\left\langle\mathbf b,\mathbf d\right\rangle_{\mathbb C^d}\,.
\end{equation}
In particular, the parameter $\varepsilon$ only modifies the metric along the fibres of $\Pi$, while all Berger metrics induce the Fubini-Study metric on the quotient.

Let $\overline{\nabla}^{(0)}$ denote the canonical connection of the second kind associated to the reductive presentation of $\CP^d$. The commutation relation \eqref{Eq: commutation relation sphere} gives
\begin{equation}
    [\mathfrak m_{\mathrm h},\mathfrak m_{\mathrm h}]\subseteq\mathfrak m_{\mathrm v}\oplus\su(d+1)_{\mathbf e_0}=\su(d+1)_{[\mathbf e_0]}\,,
\end{equation}
and hence
\begin{equation}
    [\mathfrak m_{\mathrm h},\mathfrak m_{\mathrm h}]_{\mathfrak m_{\mathrm h}}=0\,.
\end{equation}
Therefore, all the connections in the family $\left\{\overline{\nabla}^{(\lambda)}\right\}_{\lambda\in\mathbb R}$ coincide. In particular, $\overline{\nabla}^{(0)}=\overline{\nabla}^{(1/2)}$, and hence $\overline{\nabla}^{(0)}$ is torsion-free. Moreover, its zero Nomizu multiplication is compatible with every $\SU(d+1)$-invariant metric, and in particular with $g^{\FS}$. Consequently, $\overline{\nabla}^{(0)}$ coincides with the Levi-Civita connection of $g^{\FS}$
\begin{equation}
    \overline{\nabla}^{(0)}=\nabla^{g^{\FS}}\,.
\end{equation}
Thus, the reductive homogeneous structure induced on the quotient by the Hopf fibration is the Riemannian structure
\begin{equation}
    \left(\CP^d,g^{\FS},\nabla^{g^{\FS}}\right)
\end{equation}
determined by the Fubini-Study metric.
\end{remark}

\subsubsection{Solution bi-form}
Let us solve the inverse problem for the metric-affine manifolds
\begin{equation}
    \left(S^{2d+1},g_\varepsilon,\nabla^{(0)}\right)\,, \qquad \varepsilon\ne0\,.
\end{equation}
Let $\varpi_\varepsilon^{(0)}$ be the solution bi-form \eqref{Eq: homogeneous lambda bi-form}. Since $\varpi^{(0)}_\varepsilon$ is invariant under the diagonal action of $\SU(d+1)$, it is enough to evaluate it on pairs of the form
$\left(\mathbf e_0,e^Z \mathbf e_0\right)\in\mathscr U_{\mathbf e_0}$, where $Z\in\mathfrak m$ and 
$\mathscr U_{\mathbf e_0}$ is the strongly $\nabla^{(0)}$-convex neighbourhood given as in the proof of \cref{Lemma: invariant strongly convex covering}. If we set $\lambda=0$ into the relation  \eqref{Eq: homogeneous lambda bi-form}, we have 
\begin{equation}\label{Eq: Berger solution biform}
    \varpi^{(0)}_\varepsilon(\tilde X\mid\tilde  Y)\left(\mathbf e_0,e^Z \mathbf e_0\right)=\left\langle X,\left(e^{-Z}\,Y\,e^{Z}\right)_{\mathfrak m}\right\rangle_\varepsilon\,, \qquad X,Y,Z\in\mathfrak m\,.
\end{equation}
In order to elaborate such expression more explicitly,  consider first the case $\varepsilon_\kappa=\frac{d+1}{2d}$,  for which the Berger metric is proportional to the restriction of the Cartan-Killing form:
\begin{equation}\label{Eq: Berger Killing metric}
    \left\langle X(a,\mathbf b),X(c,\mathbf d)\right\rangle_{\varepsilon_\kappa}=\Re\left\langle\mathbf b,\mathbf d\right\rangle_{\mathbb C^d}+\varepsilon_\kappa\,a\,c=-\frac{1}{4(d+1)}\,\kappa\left(X(a,\mathbf b),X(c,\mathbf d)\right)\,.
\end{equation}
Since $\mathfrak m$ is the $\kappa$-orthogonal complement of the isotropy Lie algebra $\su(d+1)_{\mathbf e_0}$, the $\mathfrak m$-projection in \eqref{Eq: Berger solution biform} may be omitted when paired to an element of $\mathfrak m$. Therefore, we have
\begin{align}
    \varpi^{(0)}_{\varepsilon_\kappa}\left(X(a,\mathbf b)^+\mid X(c,\mathbf d)^+\right)\left(\mathbf e_0,e^Z \mathbf e_0\right)&=-\frac{1}{4(d+1)}\,\kappa\left(X(a,\mathbf b),e^{-Z}\,X(c,\mathbf d)\,e^Z\right)\\
    &=-\frac12\,\Tr\left(X(a,\mathbf b)e^{-Z}X(c,\mathbf d)e^Z\right)\,.
\end{align}
For a general Berger metric, we use the decomposition
\begin{equation}\label{Eq: Berger metric decomposition}  \left\langle\cdot,\cdot\right\rangle_\varepsilon=\left\langle\cdot,\cdot\right\rangle_{\varepsilon_\kappa}+\left(\varepsilon-\varepsilon_\kappa\right)\left\langle\cdot,\cdot\right\rangle_{\mathrm v}\,,
\end{equation}
where
$\left\langle X(a,\mathbf b),X(c,\mathbf d)\right\rangle_{\mathrm v}=a\,c\,$.
Write
\begin{equation}
    \left(e^{-Z}\,X(c,\mathbf d)\,e^Z\right)_{\mathfrak m}=X(c_Z,\mathbf d_Z)\,:
\end{equation}
the equations \eqref{Eq: Berger solution biform} and \eqref{Eq: Berger metric decomposition} correspondingly give
\begin{align}
    \varpi^{(0)}_{\varepsilon}\left(\tilde X(a,\mathbf b)\mid \tilde X(c,\mathbf d)\right)\left(\mathbf e_0,e^Z \mathbf e_0\right)
    =-\frac12\,\Tr\left(X(a,\mathbf b)\,e^{-Z}\,X(c,\mathbf d)\,e^Z\right)+\left(\varepsilon-\varepsilon_\kappa\right)\,a\,c_Z\,.
\end{align}
All elements in $\su(d+1)_{\mathbf e_0}$ have vanishing $(0,0)$-entry, and therefore the $\mathfrak m$-projection does not change this entry. Since the $(0,0)$-entry of $X(c_Z,\mathbf d_Z)$ is the term $\mathrm i\, c_Z$, we have
\begin{equation}
    c_Z=-\mathrm i\left(e^{-Z}X(c,\mathbf d)e^Z\right)_{00}\,.
\end{equation}
We thus obtain
\begin{align}\label{Eq: Berger solution biform explicit}
    \varpi^{(0)}_{\varepsilon}\left(\tilde X(a,\mathbf b)\mid \tilde X(c,\mathbf d)\right)\left(\mathbf e_0,e^Z \mathbf e_0\right)=-\frac12\,\Tr\left(X(a,\mathbf b)e^{-Z}X(c,\mathbf d)e^Z\right)-\mathrm i\left(\varepsilon-\frac{d+1}{2d}\right)a\left(e^{-Z}X(c,\mathbf d)e^Z\right)_{00}\,.
\end{align}
%$\varpi_{\mathrm v}$ is the non-contrast bi-form determined by\begin{equation}\label{Eq: Berger vertical biform}\varpi_{\mathrm v}\left(X(a,\mathbf b)^+\mid X(c,\mathbf d)^+\right)\left(\mathbf e_0,e^Z\mathbf e_0\right)=-\mathrm i\,a\left(e^{-Z}X(c,\mathbf d)e^Z\right)_{00}\,,\end{equation}and extended by diagonal $\SU(d+1)$-invariance.
The reduction of the Berger metric-affine manifold \eqref{Eq: Berger second canonical metric-affine manifold} to the Riemannian manifold determined by the Fubini-Study metric through the Hopf fibration can be recovered directly at the level of the corresponding solution bi-forms. Indeed, the solution bi-form $\varpi^{(0)}_\varepsilon$ decomposes as
\begin{equation}\label{Eq: Berger solution biform decomposition}
    \varpi^{(0)}_\varepsilon=\varpi^{(0)}_{\varepsilon_\kappa}+\left(\varepsilon-\varepsilon_\kappa\right)\eta\boxtimes\eta\,.
\end{equation}
Since the horizontal distribution induced by $\mathfrak m_{\mathrm h}$ coincides with $\ker\eta$, this decomposition shows that the restriction of $\varpi^{(0)}_\varepsilon$ to the horizontal distribution is independent of $\varepsilon$.

Let $Z\in\mathfrak m_{\mathrm h}$ be such that $(\mathbf e_0,e^Z\,\mathbf e_0)\in\mathscr U_{\mathbf e_0}$, and note that the $\nabla^{(0)}$-geodesic $\gamma_Z(t)=e^{t\,Z}\,\mathbf e_0$ projects onto the $\overline{\nabla}^{(0)}$-geodesic $\overline{\gamma}_Z(t)=e^{t\,Z}\,[\mathbf e_0]$. Using this projection property, together with \eqref{Eq: Hopf related fundamental fields examples} and the relation between the Berger and Fubini-Study metrics, we obtain
\begin{equation}\label{Eq: Hopf pullback solution biform}
    \varpi_{\varepsilon}^{(0)}\left(\tilde X(0,\mathbf b)\mid \tilde X(0,\mathbf d)\right)\left(\mathbf e_0,e^Z\,\mathbf e_0\right)=\left(\Pi^\ast\overline{\varpi}^{(0)}\right)\left(\tilde X(0,\mathbf b)\mid \tilde X(0,\mathbf d)\right)\left(\mathbf e_0,e^Z\,\mathbf e_0\right)\,.
\end{equation}
Here $\overline{\varpi}^{(0)}$ denotes the solution bi-form associated to $\left(\CP^d,g^{\FS},\overline{\nabla}^{(0)}\right)$. Its germ along the diagonal is determined by
\begin{equation}\label{Eq: projective solution biform}
    \overline{\varpi}^{(0)}\left(\hat{X}(0,\mathbf b)\mid\hat{X}(0,\mathbf d)\right)\left([\mathbf e_0],e^Z[\mathbf e_0]\right)=-\frac12\,\Tr\left(X(0,\mathbf b)\,e^{-Z}\,X(0,\mathbf d)\,e^Z\right)\,,
\end{equation}
whenever $\left([\mathbf e_0],e^Z[\mathbf e_0]\right)$ belongs to a strongly $\overline{\nabla}^{(0)}$-convex neighbourhood of the diagonal. Thus, the horizontal component of $\varpi^{(0)}_\varepsilon$ is the pullback of the projective solution bi-form, whereas the term involving $\eta\boxtimes\eta$ provides information on the vertical component of the torsion corresponding to the  Berger metric-affine structure.

\section{Conclusion}

The constructions developed in this paper provide a constructive solution to the inverse problem in Information Geometry. The problem is first solved at the metric-affine level, without imposing any curvature or torsion constraints, by integrating the given affine connection into a parallelism and pairing it to the metric. Statistical homotopy operators then recover pre-contrast and contrast functions whenever the given structure is a SMAT or a statistical manifold. The construction also makes explicit the dependence of the resulting potentials on the auxiliary affine connections. Suitable choices of these connections recover the pre-contrast function of Henmi and Matsuzoe and the contrast functions of Ay and Amari and of Henmi and Kobayashi. The construction is made explicit for reductive homogeneous spaces endowed with invariant metrics and invariant affine connections in the sense of Nomizu, including semisimple Lie groups and odd-dimensional spheres endowed with Berger metrics. 

We conclude with some observations on a possible extension of the framework and constructions developed in this paper to Lie groupoids. This is in line with the approach proposed in \cite{G-G-K-M-2019,G-G-K-M-2020}, where Lie groupoids are regarded as a natural setting for Information Geometry. In this formulation, the statistical structure associated to a Lie groupoid $\mathcal G\rightrightarrows M$ is defined on its Lie algebroid $A(\mathcal G)$. The metric is replaced by a fiber metric $g_A$ on $A(\mathcal G)$, while the affine connections are replaced by a pair of conjugate torsion-free $A(\mathcal G)$-connections $\nabla$ and $\nabla^\dagger$. Their compatibility is expressed in terms of the anchor $\rho\colon A(\mathcal G)\to\T M$ by
\begin{equation}
\mathcal L_{\rho(X)}\left(g_A(Y,Z)\right)=g_A\left(\nabla_XY,Z\right)+g_A\left(Y,\nabla^\dagger_XZ\right)\,,
\end{equation}
for all sections $X,Y,Z$ of $A(\mathcal G)$. When $\mathcal G$ is the pair groupoid of $M$, its Lie algebroid is canonically identified with $\T M$ and its anchor with the identity, so that the usual notion of a statistical manifold is recovered.

The construction of statistical structures on a Lie groupoid is formulated by differentiating contrast functions $F\colon\mathcal G\to\mathbb R$ along the left- and right-invariant lifts of sections of $A(\mathcal G)$. This suggests defining $(1,1)$-bi-forms on $\mathcal G$ as local sections, near the unit submanifold, of
\begin{equation}
t^\ast A(\mathcal G)^\ast\otimes_{\mathcal G}s^\ast A(\mathcal G)^\ast\longrightarrow\mathcal G\,.
\end{equation}
For the pair groupoid, the local sections of this bundle are precisely the elements of $\Omega^{1,1}_{\Delta_M}(M\mid M)$ considered in \cref{Sec: two-point formalism in IG}.

The parallelisms used in this paper admit a natural groupoid interpretation. A parallelism $P$ determines a smooth local map
\begin{equation}
\mathsf P\colon\operatorname{Pair}(M)\dashrightarrow\operatorname{Iso}(\T M)
\end{equation}
characterized by
\begin{equation}
\mathsf P(n,m)=P_{(m,n)}\colon\T_mM\longrightarrow\T_nM\,,
\end{equation}
and defined on a neighbourhood of the units. This suggests replacing the pair groupoid and its tangent bundle with a general Lie groupoid $\mathcal G\rightrightarrows M$ and its Lie algebroid $A(\mathcal G)$. The analogous object would then be a smooth local map
\begin{equation}
\mathsf P_{\mathcal G}\colon\mathcal G\dashrightarrow\operatorname{Iso}\left(A(\mathcal G)\right)
\end{equation}
preserving sources, targets and units.

The Lie group example treated in \cref{Sec: examples} suggests a direct instance of this construction. Any Lie group $G$ can be regarded as a Lie groupoid over a point, whose Lie algebroid is its Lie algebra $\mathfrak g$. A parallelism on this groupoid is then a smooth map
\begin{equation}
\mathsf P_G\colon G\dashrightarrow\operatorname{Iso}(\mathfrak g)
\end{equation}
defined on a neighbourhood of the identity and satisfying
\begin{equation}
\mathsf P_G(e)=\id_{\mathfrak g}\,.
\end{equation}
The bi-forms considered in \eqref{Eq: CS solution biform} provide a family of such parallelisms:
\begin{equation}
\mathsf P_G^{(\lambda)}(z)=e^{-\lambda\,\ad_{\log_G(z)}}\,.
\end{equation}
\section*{Acknowledgments}
This work has been supported by the Madrid Government through the project TEC-2024/COM-84 QUITEMADCM, by the Agencia Estatal de Investigación through the project PID2024-160539NB-I00, and by COST (European Cooperation in Science and Technology) through the COST Action CaLISTA CA21109.

%We acknowledge financial support from Next Generation EU through the project 2022XZSAFN – PRIN2022 CUP: E53D23005970006.

We thank our Institutions, the Universidad Carlos III de Madrid, the Università Federico II di Napoli, the Scuola Superiore Meridionale di Napoli, as well as Indam
and INFN, through Gnsaga and
the initiative GeoSymQFT and Quantum, for the financial support of the research visits we had during the period this work has been developed.

\addcontentsline{toc}{section}{References}
\bibliographystyle{plainurl}
\bibliography{scientific_bibliography_uniform}

\end{document}